\documentclass[11pt,oneside,reqno]{amsart}

\usepackage{color}
\usepackage{graphicx}
\usepackage{amssymb}
\usepackage{amsthm}
\usepackage{amsmath}
\usepackage[margin=1in]{geometry}
\usepackage{mathabx}
\usepackage{scalerel}
\usepackage{mathrsfs}
\usepackage{tikz}

\newcommand\norm[1]{\left\lVert#1\right\rVert}

\newtheorem{theorem}{Theorem}[section]
\newtheorem{thm}{Theorem}[section]

\newtheorem{definition}{Definition}[section]
\newtheorem{lemma}{Lemma}[section]
\newtheorem{corollary}{Corollary}[section]
\newtheorem{remark}{Remark}[section]
\newtheorem{proposition}{Proposition}[section]

\newcommand{\bea}{\begin{eqnarray*}}
\newcommand{\eea}{\end{eqnarray*}}                          
\newcommand{\ben}{\begin{eqnarray}}
\newcommand{\een}{\end{eqnarray}}
\newcommand{\beq}{\begin{equation}}
\newcommand{\eeq}{\end{equation}}

\newcommand{\CC}{\ensuremath{\mathbb{C}}}

\newcommand{\RR}{\ensuremath{\mathbb{R}}}

\newcommand{\TT}{\ensuremath{\mathbb{T}}}
\newcommand{\ZZ}{\ensuremath{\mathbb{Z}}}

\renewcommand{\Im}{\operatorname{Im}}
\renewcommand{\Re}{\operatorname{Re}}

\renewcommand{\hat}[1]{\widehat{#1}}

\allowdisplaybreaks[4]

\begin{document}

\title{Justification of Peregrine soliton from full  water waves}
\email{qingtang@umich.edu}
\address{University of Michigan, Ann Arbor, 
Department of Mathematics
}

\author{Qingtang Su}

\begin{abstract}
The Peregrine soliton $Q(x,t)=e^{it}(1-\frac{4(1+2it)}{1+4x^2+4t^2})$ is an exact solution of the 1d focusing nonlinear schr\"{o}dinger equation (NLS) $iB_t+B_{xx}=-2|B|^2B$, having the feature that it decays to $e^{it}$ at the spatial and time infinities, and with a peak and troughs in a local region. 
It is considered as a prototype of  the  rogue waves by the ocean waves community. The 1D NLS is related to the full water wave system in the sense that asymptotically it is the envelope equation for the full water waves. 
In this paper, working in the framework of water waves which decay non-tangentially, we give a rigorous justification of the NLS from the full water waves equation in a regime that allows for  the Peregrine soliton. As a byproduct, we prove long time existence of solutions for the full water waves equation with small initial data in space of the form $H^s(\mathbb{R})+H^{s'}(\mathbb{T})$, where $s\geq 4, s'>s+\frac{3}{2}$.

\end{abstract}

\maketitle

\section{Introduction}
The motion of the two dimensional inviscid incompressible irrotational infinite depth water waves without surface tension is described by the free boundary Euler equations (It's called the water wave equations, or water wave system)
\begin{equation}\label{system}
    \begin{cases}
    v_t+v\cdot\nabla v=-\nabla P-(0,1)\quad \quad & on ~\Omega(t),\quad t\geq 0  \\
    div~v=0,\quad curl~v=0,  & on~\Omega(t), \quad  t\geq 0\\
    P\Big|_{\Sigma(t)}\equiv 0,  & t\geq 0\\
    (1,v)~\text{ is tangent to the free surface } (t,\Sigma(t)).
    \end{cases}
\end{equation}
Here $\Omega(t)$ is the fluid region, with the free interface $\Sigma(t):=\partial \Omega(t)$, $v$ is the fluid velocity, and $P$ is the pressure. $\Sigma(t)$ separates the fluid region below $\Sigma(t)$ with density one from the air with density zero.  We identify a point $(x,y)\in \RR^2$ with $x+iy\in \CC$. It implies from $div~v=0$ and $curl~v=0$ that $\bar{v}$ is holomoprhic in $\Omega(t)$, so $v$ is completely determined by its boundary value on $\Sigma(t)$. Let  the interface $\Sigma(t)$ be given by $z=z(\alpha,t)$, with $\alpha\in\mathbb R$ the Lagrangian coordinate, so that $z_t(\alpha,t)=v(z(\alpha,t),t)$, and $v_t+v\cdot\nabla v\Big|_{\Sigma(t)}=z_{tt}$. Because $P(z(\alpha,t),t)\equiv 0$, we can write $\nabla p\Big|_{\Sigma(t)}=-iaz_{\alpha}$, where $a:=-\frac{\partial P}{\partial \boldmath{n}}\frac1{|z_\alpha|}$ is a real valued function. So the momentum equation $v_t+v\cdot \nabla v=-(0,1)-\nabla p$ along $\Sigma(t)$ can be written as
\begin{equation}
z_{tt}-iaz_{\alpha}=-i.
\end{equation}
Since $\bar{z}_t$ is the boundary value of $\bar{v}$, 
the water wave equations (\ref{system}) is equivalent to 
\begin{equation}\label{system_boundary}
\begin{cases}
z_{tt}-iaz_{\alpha}=-i\\
\bar{z}_t~is~holomorphic.
\end{cases}
\end{equation}
Here, by $\bar{z}_t$ holomorphic, we mean that there is a holomorphic function $\Phi(\cdot,t)$ on $\Omega(t)$ such that $\bar{z}_t(\alpha,t)=\Phi(z(\alpha,t),t)$.


The motion of water waves is a fascinating subject that has attracted the attention of scientists for centuries. For early works, see Newton \cite{newton}, 
Stokes\cite{Stokes}, Levi-Civita\cite{Levi-Civita}, and G.I. Taylor \cite{Taylor}. In recent years there have been numerous study on the wellposedness  of the periodic water waves or water waves which are at rest at spatial infinity.   Nalimov \cite{Nalimov}, Yosihara\cite{Yosihara} and Craig \cite{Craig} proved local well-posedness for 2d water waves equation (\ref{system}) for small initial data. In S. Wu's breakthrough works \cite{Wu1997}\cite{Wu1999}, she proved that for $n\geq 2$ the important strong Taylor sign condition
\begin{equation}
-\frac{\partial P}{\partial \boldmath{n}}\Big|_{\Sigma(t)}\geq c_0>0
\end{equation}
always holds for the infinite depth water wave system (\ref{system}), as long as the interface is non-self-intersecting and smooth, and she proved that the initial value problem for (\ref{system}) is locally well-posed in $H^s(\mathbb{R}), s\geq 4$ without smallness assumption. Since then, a lot of interesting local well-posedness results were obtaind, see for example \cite{alazard2014cauchy}, \cite{ambrose2005zero}, \cite{christodoulou2000motion}, \cite{coutand2007well}, \cite{iguchi2001well}, \cite{lannes2005well}, \cite{lindblad2005well}, \cite{ogawa2002free}, \cite{shatah2006geometry}, \cite{zhang2008free}.
 Recently, almost global and global well-posedness for water waves (\ref{system}) under irrotational assumption have also been proved, see \cite{Wu2009}, \cite{Wu2011}, \cite{germain2012global}, \cite{Ionescu2015}, \cite{AlazardDelort},  and see also \cite{HunterTataruIfrim1}, \cite{HunterTataruIfrim2},  and \cite{wang2018global}.  More recently, there are strong interests in understanding the singularities of water waves, see for example  \cite{kinsey2018priori}, \cite{Wu2}, \cite{Wu3}, \cite{Wu4}. For the formation of splash singularities, see for example \cite{castro2012finite}\cite{castro2013finite}\cite{coutand2014finite}\cite{coutand2016impossibility} . Note that all the aforementioned works assume either the water waves is periodic or at rest at spatial infinity. Regarding water waves that are nonvanishing at infinity\footnote{By nonvanishing, we mean that the water wave is neither periodic nor at rest at spatial infinity}, in \cite{Alazard2016337}, Alazard, Burq and Zuily showed that the water waves system is locally wellposed in Kato's uniform local Sobolev spaces $H_{ul}^s$.


Another important research direction concerns the behavior of the water waves in various asymptotic regimes, 
see for example \cite{craig1985existence}\cite{schneider2000long}\cite{alvarez2008large}. The 1d cubic NLS 
\begin{equation}\label{NLS}
iu_t+u_{xx}=-2|u|^2u
\end{equation}
is relevant in deep water regime. It
is completely integrable, and has many exact solutions.  
The 1d NLS is related to the full water wave system, in the sense that asymptotically it is the envelope equation for the free interface of the water waves. If one performs multiscale analysis to determine the modulation approximations to the solution of the finite or infinite depth 2d water waves equations,  i.e., a solution $z(\alpha,t)$ of the parametrized free interface which is to the leading order a wave packet of the form
\begin{equation}\label{amplitudeevolve}
W(\alpha,t):=\alpha+\epsilon B(X, T)e^{i(k\alpha+\gamma t)}  \quad  (\epsilon ~small,\quad k, \gamma: constant),
\end{equation}
then $B$ solves the 1d focusing cubic NLS.
Here $X=\epsilon(\alpha+\frac{1}{2\gamma} t), T=\epsilon^2t$, and $\gamma=\sqrt{k}$. So the envelope $B$ is a profile that travels at the group velocity $\frac{1}{2\gamma}=\frac{d\gamma}{dk}$ determined by the dispersion relation of the water wave equations on time scale $O(\epsilon^{-1})$, and evolves according to the NLS on time scale $O(\epsilon^{-2})$. 

This discovery was derived formally by Zakharov \cite{zakharov1968stability} for the infinite depth case,  and by
Hasimoto and Ono \cite{hasimoto1972nonlinear} for the finite depth case. In \cite{craig1992nonlinear}, Craig, Sulem and Sulem applied
modulation analysis to the finite depth 2D water wave equation, derived an approximate
solution of the form of a wave packet and showed that the modulation approximation
satisfies the 2D finite depth water wave equation to the leading order. In \cite{schneider2011justification}, Schneider and Wayne justified the NLS as the modulation approximation for a quasilinear model that
captures some of the main features of the water wave equations.

The rigorous justification of the NLS for the full water waves was given by Totz and Wu \cite{Totz2012} in infinite depth case, and the justification in a canal of finite depth was proved by D\"{u}ll, Schneider and Wayne \cite{dull2016justification} . See also \cite{ifrim2018nls}. All of these works assume the data vanish at spatial infinity. However, there are many important solitons of NLS that are neither periodic nor vanishing at $\infty$.  One such important example is  the Peregrine soliton discovered by Peregrine in 1983 \cite{Peregrine}, 
which is defined by
\begin{equation}
    Q(x,t)=e^{it}(1-\frac{4(1+2it)}{1+4x^2+4t^2}).
\end{equation}
Plug $Q$ in (\ref{amplitudeevolve}), one observes that $W$ a weakly oscillatory periodic wave at the time and spatial infinity, but $W$ has peaks and troughs at a local region. 
The Peregrine soliton is important to the ocean waves because  the feature of the corresponding wave packet $W$ is consistent with the qualitative description of a rogue wave in the ocean. We call the wave packet corresponding to the Peregrine soliton just by the Peregrine soliton. Indeed, the Peregrine soliton is conjectured to be one of the mechanisms for the formation of  rogue waves by the ocean waves community,  see \cite{Shrira2010} for more details.   In 2010, the Peregrine soliton was observed in fibre optics \cite{FibreOptics}, which shows that the Peregrine soliton is a nature phenomena rather than just a mathematical prediction! Stimulated by this discovery, there have been a lot of efforts to produce the Peregrine solitons in other backgrounds, for example, in  \cite{WaterTanks}, the authors carried out the first experiment to observe Peregrine-type breather solutions in a water tank.  These experiments suggest the Peregrine soliton is plausible description of the formation of rogue waves.  So it's desirable to have a mathematical theory to justify that the Peregrine soliton can be developed in water waves. Since the motion of the water waves is governed by the water wave equations, while the Peregrine soliton is an exact solution of the NLS, we ask the following question:

\vspace*{1ex}

\noindent \textbf{Question 1.} Is there any solution $z(\alpha,t)$ to the system (\ref{system_boundary}) with its envelope looks like the Peregrine soliton?

\vspace*{2ex}

Since the leading order of the envelope of the wave packet $W(\alpha,t)=\alpha+\epsilon B(X, T)e^{i(k\alpha+\gamma t)}$ evolves according to the NLS on time scale $O(\epsilon^{-2})$, in order to observe the evolution of the wave packet, the observer must focus on the water waves on time scale $O(\epsilon^{-2})$. So a more precise formulation of Question 1 is as follows:

\vspace*{2ex}

\noindent \textbf{Question 1'.} Is there any solution $z(\alpha,t)$ to (\ref{system_boundary}) such that 
\begin{equation}
    \sup_{t\in [0,O(\epsilon^{-2})]}\|(z-W, \quad z_t-W_t,\quad z_{tt}-W_{tt})\|=o(\epsilon) ?
\end{equation}
Here $\|\cdot\|$ denotes some norm. Since $B$ is neither periodic nor vanishing at $\infty$, the framework in \cite{Totz2012} or \cite{dull2016justification}  cannot be applied to justify the Peregrine soliton from the full water wave equations. 

In this paper, we give an affirmative answer to $\textbf{Question 1'}$. Denote 
\begin{equation}
    1+H^s:=\{f=1+g: g\in H^s\}.
\end{equation}
\vspace*{1ex}
\noindent \textbf{Notation.} Denote $\mathbb{T}:=[-\pi, \pi]$.

\vspace*{1ex}

\noindent Let $f=f_0+f_1$, where $f_0\in H^{s+s_0}(\TT), f_1\in H^s(\RR)$, where $s_0>3/2$. Define
\begin{equation}
    \norm{f}_{X^s}:=\norm{f_0}_{H^{s+s_0}(\mathbb{T})}+\norm{f_1}_{H^s(\mathbb{R})}.
\end{equation}
The main result of this paper is the following which gives a rigorous justification of the NLS with nonzero boundary values at spatial infinity from the full water waves.
\begin{theorem}\label{main}
Let $M_0>0$, $s\geq 4$ and $k> 0$ be given. $B(0)\in 1+H^{s+7}(\RR)$. Denote by $B(X,T)$ the solution of the NLS: $2iB_T+\frac{1}{4k^{3/2}}B_{XX}+k^{5/2} B|B|^2=0$ with initial data $B(X,T=0)$, and let $\zeta^{(1)}(\alpha,t)=B(X,T)e^{i\phi}$, where $X=\epsilon(\alpha+\frac{1}{2\sqrt{k}}t)$, $T=\epsilon^2t$, and $\phi=k\alpha+\sqrt{k} t$. There exists a constant $\epsilon_0=\epsilon_0(s,M_0,\|B(0)-1\|_{H^{s+7}})>0$ such that for all $\epsilon<\epsilon_0$, there exists initial data $(z(\cdot,0), z_t(\cdot,0), z_{tt}(\cdot,0))$ to the water wave system (\ref{system_boundary}) such that 
\begin{equation}
\begin{split}
&\|(z_{\alpha}(\cdot,0)-1, z_t(\cdot,0), z_{tt}(\cdot,0)) -\epsilon( \partial_{\alpha}\zeta^{(1)}(0),  \partial_t\zeta^{(1)}(0),  \partial_t^2 \zeta^{(1)}(0))\|_{X^{s-1/2}\times X^{s+1/2}\times X^{s}}\\
\leq &M_0\epsilon^{3/2},
\end{split}
\end{equation}
and there exists a constant $k_0=k_0(s)>0$ such that for all such initial data, the water waves system has a unique solution with $$(z_{\alpha}(\cdot,t)-1, z_t, z_{tt})\in C([0, k_0\epsilon^{-2}]; X^{s-1/2}\times X^{s+1/2}\times X^s)$$ satisfying 
\begin{equation}\label{lagrangejustify}
\begin{split}
\sup_{0\leq t\leq k_0\epsilon^{-2}}\|(Im\{z_{\alpha}-1\}, z_t, z_{tt})-\epsilon (Im\{\partial_{\alpha}\zeta^{(1)}\}, \zeta_t^{(1)}, \zeta_{tt}^{(1)})\|_{X^{s-1/2}\times X^{s+1/2}\times X^s}\leq C\epsilon^{3/2}, 
\end{split}
\end{equation}
for some constant $C=C(s, M_0, \norm{B(0)-1}_{H^{s+7}})$.
\end{theorem}

\begin{remark}
Theorem \ref{main} gives rigorous justification of the NLS with nonzero boundary vlaues at $\infty$ in Lagrangian coordinates.
In (\ref{lagrangejustify}), please note that we only justify the modulation approximation for the imaginary part of $z_{\alpha}-1$,, i.e.,  
\begin{equation}
    \sup_{t\in [0,k_0\epsilon^{-2}]}\norm{Im\{z_{\alpha}-1\}-\epsilon \Im\{\partial_{\alpha}\zeta^{(1)}\}}_{X^{s-1/2}}\lesssim \epsilon^{3/2}.
\end{equation}
In Theorem \ref{wuwu}, we give a full justification of the Peregrine soliton from full water waves in a different coordinates. We fail to rigorously justify $\sup_{t\in [0,k_0\epsilon^{-2}]}\norm{Re\{z_{\alpha}-1\}-\epsilon Re\{\partial_{\alpha}\zeta^{(1)}\}}_{X^{s-1/2}}\lesssim \epsilon^{3/2}$ in the Lagrangian coordinates because we are unable to obtain good control of the change of variables on time scale $O(\epsilon^{-2})$, please see Theorem \ref{wuwu} and Remark \ref{remarkchangeback} for the details.
\end{remark}

\begin{remark}
Let $B$ solves $iB_T+B_{XX}=-2|B|^2B$. Then $$U(X,T):=\frac{2}{\sqrt{k^{5/2}}}B( \sqrt{8k^{3/2}}X, T)$$ solves $2iU_T+\frac{1}{4k^{3/2}}U_{XX}+k^{5/2} U|U|^2=0$.

Applying Theorem  \ref{main} to $U(X,0)=\frac{2}{\sqrt{k^{5/2}}}Q( \sqrt{8k^{3/2}}X,0)$ gives an affirmative answer to Question 1'.
\end{remark}

\begin{remark}
$s_0>3/2$ is of course not optimal. 
We take $s_0>3/2$ to avoid getting into too many technical issues.
\end{remark}

\subsection{Challenges of the problem and the strategy.}

\subsubsection{First difficulty: find a right class of water waves to work with. } Suppose $B$ is the Peregrine soliton, then the wave packet $W$ is nonvanishing.
As a consequence, in order to justify the Peregrine soliton from the full water waves, we need to show that water waves with nonvanishing data of size $O(\epsilon)$ exist on time scale $O(\epsilon^{-2})$. In \cite{Alazard2016337}, Alazard, Burq, and Zuily proved local wellposedness of (\ref{system}) with nonvanishing data in Kato's uniform local space $H_{ul}^s(\RR)$. Their result implies that for initial data of size $O(\epsilon)$, the lifespan of the solution is at least of order $O(\epsilon^{-1})$, which is not enough for justifying the Peregrine soliton. 
Even though the long time existence has been well-known for periodic waves and localized waves,  to the author's best of knowledge, for nonvanishing water waves, no long time existence results with lifespan of the solution longer than the order of $O(\epsilon^{-1})$ exist, and the analytical tools developed for the vanishing or periodic data can not be directly used in this setting. 

In order to prove long time existence of the water wave system, one needs to find a cubic structure for the water wave equations. More precisely, we need to find some quantity $\theta$ such that $\partial_t\theta\approx z_t$ and 
\begin{equation}
    (\partial_t^2-ia\partial_{\alpha})\theta=F,
\end{equation}
with $F$ consists of cubic and higher order nonlinearities. For water waves with data in Sobolev spaces, 
there are two ways of doing this. The first one is the fully nonlinear transform constructed by S. Wu. In \cite{Wu2009}, S. Wu considered $\theta:=(I-\mathfrak{H})(z-\bar{z})$ and showed that $(\partial_t^2-ia\partial_{\alpha})\theta=cubic$. Here, 
\begin{equation}
    \mathfrak{H}f(\alpha)=\frac{1}{\pi i}\int_{-\infty}^{\infty}\frac{z_{\beta}}{z(\alpha,t)-z(\beta,t)}f(\beta)d\beta
\end{equation}
is the Hilbert transform associated with the free interface labeled by $z(\alpha,t)$. Using this fully nonlinear transform, S. Wu was able to prove the almost global existence for the irrotational water waves with small localized initial data. The method implies lifespan of order $O(\epsilon^{-2})$ for nonlocalized data of size $O(\epsilon)$ in Sobolev spaces. This nonlinear transform is also used in \cite{Wu2011}\cite{Totz2012}\cite{Totz2015}\cite{Ionescu2015}\cite{su2018long}. See also \cite{HunterTataruIfrim1}\cite{HunterTataruIfrim2} for similar ideas. The second way is to use the normal form transformation to construct a cubic structure, see for example \cite{AlazardDelort}\cite{Ionescu2015}\cite{wang2018global}\cite{ionescu2018long}\cite{berti2018birkhoff}. These two methods work well for water waves with periodic data or with data in Sobolev spaces. 

However, for nonvanishing water waves, such a cubic structure was unclear for both methods. The first difficulty we confront is to find a right class of water waves that we can work with. This class of water waves must be non-vanishing at spatial infinity along the free interface. However, if the water waves have too many activities at infinity, then it's not obvious at all that why the water waves should exist for a long time.

\subsubsection{The idea of resolving the first difficulty: water waves that decays nontangentially}
Let $B$ be the Peregrine soliton, then the wave packet $W$ can be decomposed as
\begin{equation}
    W=W_0+W_1,\quad \quad W_0=\epsilon e^{i\epsilon^2 t}e^{i\phi}, \quad \quad W_1=\alpha-\epsilon e^{i\epsilon^2t}\frac{1+8i\epsilon^2t}{1+4(\epsilon \alpha)^2+4(\epsilon^2t)^2}e^{i\phi}
\end{equation}
Note that $W_0$ is periodic, $W_1-\alpha$ vanishes at infinity, therefore, we consider water waves which is a superposition of periodic waves and waves which vanish at infinity. Moreover, since $W_0\in C^{\infty}(\mathbb{T})$, we can assume that the periodic waves has more regularity than the localized waves. This motivates us to 
work in the function space
$X^s:=H^s(\mathbb{R})+H^{s+s_0}(\mathbb{T})$, where $s_0>3/2$. 

\noindent \textbf{Key observation:} Although the velocity $v$ is nonvanishing along the free interface, however, away from the free interface, $v$ can  vanish at spatial infinity. In other words, although the water waves have a lot of activity at spatial infinity along the free interface, however, away from the interface, the water waves can be at rest at infinity. This observation suggests that, away from the free interface, the interaction between the periodic waves and the localized waves is weak.

To make the above discussion precise,  we use the notion of \emph{decay nontangentially}. \begin{definition}[Cone]
Let $z_0\in \mathbb{C}$. Let $\theta_0\in (0,\pi/2)$. Denote
$$C_{\theta_0}(z_0):=\Big\{z\in \mathbb{C}: \Big|\frac{\Re(z-z_0)}{\Im(z-z_0)}\Big|\leq \tan \theta_0\quad \&\quad \Im z\leq \Im z_0\Big\}.$$
That is, $C_{\theta_0}(z_0)$ is the cone with vertex $z_0$ and angle $\theta_0$.
\end{definition}

\begin{definition}[Decay nontangentially]
Let $\phi(z)$ be a function in $\Omega(t)$. Let $z_0\in \mathbb{C}$ be a fixed point. We say that $\phi(z)\rightarrow 0$ nontangentially as $z\rightarrow \infty$ if for any $0<\theta<\pi/2$, 
\begin{equation}
\lim_{\substack{
z\in \Omega(t)\cap C_{\theta_0}(z_0)\\
|z|\rightarrow \infty}}\phi(z)=0.
\end{equation}
\end{definition}
\begin{remark}
Note that the definition above is invariant if we use different $z_0$. As a consequence, we choose $z_0=0$ and write $C_{\theta_0}(0)$ as $C_{\theta_0}$.
\end{remark}

\begin{remark}
If $\phi$ is a periodic function in $\Omega(t)$, and 
$$\lim_{\Im z\rightarrow -\infty}\phi(z)=0,$$
then $\phi(z)$ decays nontangentially.
\end{remark}

It turns out that the \emph{decay nontangentially} is the right setting for nonvanishing water waves. If we assume the velocity field $v$ decays nontangentially, follow S. Wu's method in \cite{Wu2009}, at least formally (in BMO sense, because the Hilbert transform $\mathfrak{H}$ maps $L^{\infty}$  to $BMO$), we can show that the quantity $(I-\mathfrak{H})(z-\bar{z})$ satisfies
\begin{align}
    &(\partial_t^2-ia\partial_{\alpha})(I-\mathfrak{H})(z-\bar{z})\\
    =&-2[z_t,\mathfrak{H}\frac{1}{z_{\alpha}}+\bar{\mathfrak{H}}\frac{1}{\bar{z}_{\alpha}}]z_{t\alpha}+\frac{1}{\pi i}\int_{-\infty}^{\infty}\Big(\frac{z_t(\alpha,t)-z_t(\beta,t)}{z(\alpha,t)-z(\beta,t)}\Big)^2(z-\bar{z})_{\beta}d\beta:=g.
\end{align}
Formally, $g$ is cubic, while $a-1$ contains first order terms, so $(\partial_t^2+|D|)\theta$ contains quadratic terms, which does not imply cubic lifespan. To resolve the problem, we follow S. Wu's idea and consider the change of variables $\kappa:\mathbb{R}\rightarrow \mathbb{R}$ such that $\bar{\zeta}-\alpha$ is boundary value of a holomorphic function which decays nontangentially, where $\zeta=z\circ\kappa^{-1}$. Denote $b=\kappa_t\circ\kappa^{-1}$, $A=(a\kappa_{\alpha})\circ\kappa^{-1}$. 
In new variables, the system (\ref{system_boundary}) is written as 
\begin{equation}\label{system_new}
\begin{cases}
(D_t^2-iA\partial_{\alpha})\zeta=-i\\
(I-\mathcal{H}_{\zeta})D_t\bar{\zeta}=0,
\end{cases}
\end{equation}
and we have
\begin{align}\label{perfect_new}
&(D_t^2-iA\partial_{\alpha})(I-\mathcal{H}_{\zeta})(\zeta-\bar{\zeta})\\
=&-2[D_t\zeta,\mathcal{H}\frac{1}{\zeta_{\alpha}}+\bar{\mathcal{H}}\frac{1}{\bar{\zeta}_{\alpha}}]\partial_{\alpha}D_t\zeta+\frac{1}{\pi i}\int_{-\infty}^{\infty}\Big(\frac{D_t\zeta(\alpha,t)-D_t\zeta(\beta,t)}{\zeta(\alpha,t)-\zeta(\beta,t)}\Big)^2(\zeta-\bar{\zeta})_{\beta}d\beta,
\end{align}
where 
\begin{equation}
\mathcal{H}_{\zeta}f(\alpha):=\frac{1}{\pi i}p.v.\int_{-\infty}^{\infty}\frac{\zeta_{\beta}}{\zeta(\alpha,t)-\zeta(\beta,t)}f(\beta)d\beta.
\end{equation}
Moreover, we have
\begin{equation}\label{formulaforbbb}
(I-\mathcal{H})b=-[D_t\zeta,\mathcal{H}]\frac{\bar{\zeta}_{\alpha}-1}{\zeta_{\alpha}},
\end{equation}
\begin{equation}\label{formulaforaaa}
 (I-\mathcal{H})(A-1)=i[D_t\zeta,\mathcal{H}]\frac{\partial_{\alpha}D_t\bar{\zeta}}{\zeta_{\alpha}}+i[D_t^2\zeta,\mathcal{H}]\frac{\bar{\zeta}_{\alpha}-1}{\zeta_{\alpha}}.
\end{equation}
So $b$, $A-1$ are quadratic. Therefore, at least formally, we have 
\begin{equation}
    (\partial_t^2+|D|)(I-\mathcal{H}_{\zeta})(\zeta-\bar{\zeta})=cubic.
\end{equation}
If $\zeta-\alpha, D_t\zeta$ decay\footnote{In this paper,  by a function $f$ decays at $\infty$, we mean that $f\in H^s(\RR)$ for some $s\geq 0$, even though it could be possible that $\lim_{x\rightarrow \infty}f(x)$ does not exist. } at spatial infinity or periodic, then use S. Wu's method, we could prove that (\ref{system}) is wellposed on time scale $O(\epsilon^{-2})$.

\subsubsection{The second difficulty} If the water wave is nonvanishing,  then it's difficult to define an energy associated with (\ref{perfect_new}) which still preserves the cubic structure. Indeed, because $\theta\notin L^q(\RR)$ for any $q\neq \infty$, we cannot estimate $\theta$ in $H^s(\RR)$. If we estimate $\theta$ in $W^{s,\infty}(\RR)$, as is explained by Alazard, Burq and Zuily in \cite{Alazard2016337}, there is loss of derivative in such spaces.  One might try to estimate $\theta$ in Kato's uniform local Sobolev spaces as in \cite{Alazard2016337}. Assume $\{\chi_n\}$ is a partition of unity of $\RR$.  One needs to consider the quantity $\chi_n\theta$. It turns out that $\mathcal{P}\chi_n\theta$ has first and quadratic nonlinearities, which are difficult to get rid of.  

\subsubsection{Idea of resolving the second difficulty}
To resolve this problem, we note that if $\zeta_{\alpha}-1\in X^s$, then $\zeta$ can be decomposed uniquely as
\begin{equation}
    \zeta=\omega+\xi_1,
\end{equation}
where $\omega-\alpha$ is periodic, and $\xi_1$ decays at spatial infinity. Let $b_0$ and $A_0$ be determined by
\begin{equation}\label{bb0}
(I-\mathcal{H}_p)b_0=-[D_t^0\omega, \mathcal{H}_p]\frac{\bar{\omega}_{\alpha}-1}{\omega_{\alpha}},
\end{equation}
\begin{equation}\label{AA0}
(I-\mathcal{H}_p)(A_0-1)=i[(D_t^0)^2\omega,\mathcal{H}_p]\frac{\bar{\omega}_{\alpha}-1}{\omega_{\alpha}}+i[(D_t^0)\omega,\mathcal{H}_p]\frac{\partial_{\alpha}(D_t^0)\bar{\omega}}{\omega_{\alpha}}.
\end{equation}
where
\begin{equation}
\mathcal{H}_pf(\alpha):=\frac{1}{2\pi i}p.v.\int_{\TT} \omega_{\beta}(\beta)\cot(\frac{1}{2}(\omega(\alpha)-\omega(\beta))f(\beta)d\beta.
\end{equation}
Denote $D_t^0:=\partial_t+b_0\partial_{\alpha}$, then $\omega$ satisfies 
\begin{equation}
    (D_t^0)^2\omega-iA_0\omega_{\alpha}=-i.
\end{equation}
It has been well known that the periodic water waves with initial data of size $O(\epsilon)$ exists on lifespan of order at least $O(\epsilon^{-2})$. So it suffices to control $\xi_1$ and $D_t\zeta-D_t^0\omega$ on time scale $O(\epsilon^{-2})$. In $BMO$ sense, we have
\begin{equation}
    ((D_t^0)^2-iA_0\partial_{\alpha})(I-\mathcal{H}_{\omega})(\omega-\bar{\omega})=cubic,
\end{equation}
where $\mathcal{H}_{\omega}$ is the Hilbert transform associated with $\omega$, i.e.,
\begin{equation}
    \mathcal{H}_{\omega}f(\alpha,t)=p.v. \frac{1}{\pi i}\int_{-\infty}^{\infty} \frac{\omega_{\beta}(\beta,t)}{\omega(\alpha,t)-\omega(\beta,t)}f(\beta,t)d\beta.
\end{equation}
Now consider the quantity
\begin{equation}
    \lambda:=(I-\mathcal{H}_{\zeta})\Big((I-\mathcal{H}_{\zeta})(\zeta-\alpha)-(I-\mathcal{H}_{\omega})(\omega-\alpha)\Big).
\end{equation}
Then $\lambda\approx \xi_1$. Moreover, we can prove that 
\begin{equation}
    (D_t^2-iA\partial_{\alpha})\lambda=cubic.
\end{equation}
Since $\lambda$ is in Sobolev space, we can use energy method to prove the following result:
\begin{theorem}\label{theorem1}
Let $s\geq 4$. Let $\|(z_{\alpha}(\cdot,0)-1, z_t(t=0), z_{tt}(t=0))\|_{X^{s-1/2}\times X^{s+1/2}\times X^s}\leq \epsilon$. Assume $\bar{z}_t(t=0)\in \mathcal{H}ol_{\mathcal{N}}(\Omega(0))$. Then there exists $\epsilon_0=\epsilon_0(s)>0$ sufficiently small and a constant $C_1=C_1(s)>0$ such that for all $0\leq \epsilon\leq \epsilon_0$, the water wave equations (\ref{system_boundary}) admit a unique solution $(z_{\alpha}(\cdot,t)-1, z_t(\cdot,t), z_{tt}(\cdot,t))\in C([0,C_1\epsilon^{-2}]; X^{s-1/2}\times X^{s+1/2}\times X^s)$. Moreover, 
\begin{equation}
    \sup_{0\leq t\leq C_1\epsilon^{-2}}\|(z_{\alpha}-1, z_t, z_{tt})\|_{X^{s-1/2}\times X^{s+1/2}\times X^s }\leq C\epsilon,
\end{equation}
for some constant $C=C(s)$.
\end{theorem}
\vspace*{1ex}

\noindent To our best knowledge, Theorem \ref{theorem1} is the first long time existence for nonvanishing water waves. More importantly, using this long time existence result, we are able to justify the NLS from the full water waves in a regime that allows for Peregrine solitons, and prove Theorem \ref{main}.

\subsubsection{Rigorous justification of the Peregrine soliton from water waves .} We prove Theorem \ref{main} through the following steps.
\begin{itemize}
    \item [Step 1.] \underline{Construction of approximate solution.} 
    
  \noindent   Let $B$ be a solution to the NLS. Consider interface of the form 
    \begin{equation}\label{special}
        \zeta(\alpha,t)=\alpha+\sum_{n=1}^{\infty}\epsilon^n \zeta^{(n)}.
    \end{equation}
By multiscale analysis, we can choose $\zeta^{(n)}, n=1,2,3$ be such that $\zeta^{(1)}=B(X,T)e^{i\phi}$ and $\zeta^{(2)}$, $\zeta^{(3)}$ depend on $B$ and $\phi$ only. We define an approximate solution $\tilde{\zeta}$ to $\zeta$ by
\begin{equation}
    \tilde{\zeta}:=\alpha+\sum_{n=1}^3 \epsilon^n \zeta^{(n)},
\end{equation}
then formally\footnote{This is in $\infty$-norm sense, i.e., $\|\zeta-\tilde{\zeta}\|_{W^{s,\infty}}=O(\epsilon^4)$. In $X^s$ norm, $\|\zeta-\tilde{\zeta}\|_{X^s}=O(\epsilon^{7/2})$}, 
\begin{equation}
    |\zeta-\tilde{\zeta}|=O(\epsilon^4).
\end{equation}
Similarly, we approximate $D_t\zeta$, $D_t^2\zeta$ by some appropriate functions $\tilde{D}_t\tilde{\zeta}, \tilde{D}_t^2\tilde{\zeta}$ such that 
\begin{equation}
    |D_t\zeta-\tilde{D}_t\tilde{\zeta}|=O(\epsilon^4),\quad \quad |D_t^2\zeta-\tilde{D}_t\tilde{\zeta}|=O(\epsilon^4).
\end{equation}
    
\end{itemize}
Denote 
    \begin{equation}
        r:=\zeta-\tilde{\zeta}=r_0+r_1,  
    \end{equation}
    where $r_0$ is the periodic part of $r$, and $r_1$ decays at spatital infinity. Denote $\tilde{\xi}_0$ the periodic part of $\tilde{\zeta}-\alpha$. Denote 
    \begin{equation}
        \tilde{\omega}=\alpha+\tilde{\xi}_0, \quad \quad \tilde{\xi}_1:=\tilde{\zeta}-\tilde{\omega}.
    \end{equation}
    To rigorous justify the NLS from the water waves, we need to control the error $r$ on time scale $O(\epsilon^{-2})$. 
    
    \begin{itemize}
    \item [Step 2.] \underline{A priori error estimates for the periodic part}
    
\noindent     For data of the form (\ref{special}), we show that 
    \begin{equation}
        ((D_t^0)^2-iA_0\partial_{\alpha})(I-\mathcal{H}_p)r_0=\text{fourth order}.
    \end{equation}
  So we can obtain
  \begin{equation}
      \sup_{t\in [0, O(\epsilon^{-2})]}\|(\partial_{\alpha}r_0, D_t^2 r_0, (D_t^0)^2r_0)\|_{H^{s+s_0}(\TT)\times H^{s+s_0+1/2}(\TT)\times H^{s+s_0}(\TT)}\leq C\epsilon^{3/2}.
  \end{equation}

 \item [Step 3.] \underline{A priori error estimates for the vanishing part}
 
 \noindent Consider the quantity
 \begin{equation}
\rho_1:=(I-\mathcal{H}_{\zeta})\Big\{\Big((I-\mathcal{H}_{\zeta})\xi-(I-\mathcal{H}_{\omega})\xi_0\Big)-\Big((I-\mathcal{H}_{\tilde{\zeta}})\tilde{\xi}-(I-\mathcal{H}_{\tilde{\omega}})\tilde{\xi}_0\Big)\Big\}.
\end{equation}
 We remind the readers that $\mathcal{H}_{\tilde{\zeta}}$ and $\mathcal{H}_{\tilde{\omega}}$ are the Hilbert transforms associated with $\tilde{\zeta}$ and $\tilde{\omega}$, respectively. We can show that $\rho_1\approx r_1$.
 
 By exploring the structure of $\tilde{\zeta}$, we show that 
 \begin{equation}
     (D_t^2-iA\partial_{\alpha})\rho_1=\text{fourth order}.
 \end{equation}
So we can obtain
  \begin{equation}
      \sup_{t\in [0, O(\epsilon^{-2})]}\|(\partial_{\alpha}r_1, D_t r_1, D_t^2r_1)\|_{H^{s}(\RR)\times H^{s+1/2}(\RR)\times H^{s}(\RR)}\leq C\epsilon^{3/2}.
  \end{equation}
 
 \item [Step 4.] In Step 1, we've constructed an approximate solution $(\tilde{\zeta}, \tilde{D}_t\tilde{\zeta}, \tilde{D}_t^2\tilde{\zeta})$ which exists on time scale $O(\epsilon^{-2})$. In Step 2 and Step 3, we obtain a priori bound on the energy for the remainder $r$ on long time scale $O(\epsilon^{-2})$. However, since $\tilde{\zeta}$ does not in general satisfy the water wave equations, the wave packet like data $(\tilde{\zeta}(0), \tilde{D}_t\tilde{\zeta}(0), \tilde{D}_t^2\tilde{\zeta})(0)$ cannot be taken as the initial data of the water wave equations. Similar to that in \cite{Totz2012}, we show that there is initial data for the water wave system that is
within $O(\epsilon^{3/2})$ to the wave packet $(\tilde{\zeta}(0), \tilde{D}_t\tilde{\zeta}(0), \tilde{D}_t^2\tilde{\zeta})(0)$. By long time existence of (\ref{system_new}) with initial data $(\partial_{\alpha}(\zeta-\alpha), D_t\zeta, D_t^2\zeta)$ of size $\epsilon$ in $X^s\times X^{s+1/2}\times X^s$, the solution of the system (\ref{system_new}) exists on time scale $O(\epsilon^{-2})$.  The a priori bound
on $r$ gives the estimate of the error between $\zeta$ and the wave packet $\tilde{\zeta}$ on the order $O(\epsilon^{3/2})$
for time on the $O(\epsilon^{-2})$ scale. The appropriate wave packet approximation to $z$ is then
obtained upon changing coordinates back to the Lagrangian variable.
\end{itemize}

\subsection{Outline of this paper} In \S\ref{notation}, we introduce some basic notation and convention. Further notation and convention will be made throughout the paper if necessary. In \S\ref{prelim} we will provide some analytical tools and the basic definitions that will be used in later sections.
In section 3, we sketch a proof of long time existence of the periodic water waves system, which we will use in later sections.
In Section 4,  we set up the water waves system with data in $X^s$, derive formula for the corresponding quantities, and then prove long time existence of water waves in the function space $X^s$.
In Section 5, we formally derive NLS with non-vanishing boundary value at $\infty$ from non-vanishing water waves system that we set up in Section 4, and obtain an approximation $\tilde{\zeta}$ to water waves system. 
In Section 6, we derive governing equations for $r_0$, then we show that $r_0$ remains small on time scale $O(\epsilon^{-2})$.
In Section 7, we derive governing equations for $r_1$, and define corresponding energies that could be used to control norms of $r_1$. 
In Section 8, we obtain a priori bounds of a list of quantities that appear in the energy estimates, and in Section 9, we obtain energy estimates on time scale $O(\epsilon^{-2})$. As a consequence of the energy estimates, we prove our Main Theorem \ref{main} in Section 10. In the appendix, we show that $e^{-ik\alpha}$ cannot be the boundary value of a holomorphic function in the region below the curve $\{\omega(\alpha,t)=\alpha+c(t)e^{ik\alpha}\}$.

\subsection{Notation and convention}\label{notation} Assume $f$ a function on boundary of $\Omega(t)$. By saying $f$ holomorpihc,  we mean $f$ is boundary value of a holomorhpic function in $\Omega(t)$. Let $h\in L_{loc}^2(\RR)$, if $h$ is neither periodic  nor vanishing at spatial infinity, then we say that $h$ is nonvanishing.

We use $C(X_1, X_2,...,X_k)$ to denote a positive constant $C$ depends continuous on the parameters $X_1,...,X_k$. Throughout this paper, such constant $C(X_1,...,X_k)$ could be different even we use the same letter $C$. 
The commutator $[A,B]=AB-BA$. Given a function $g(\cdot,t):\mathbb{R}\rightarrow \mathbb{R}$, the composition $f(\cdot, t)\circ g:=f(g(\cdot,t),t)$. We identify the $\mathbb{R}^2$ with the complex plane. A point $(x,y)$ is identified as $x+iy$. For a point $z=x+iy$, $\bar{z}$ represents the complex conjugate of $z$.

\section{Preliminaries}\label{prelim}

In this section, we define the class of holomorphic functions that are considered in this paper, and define the function spaces, norms involved. Also, we collect the preliminary analytical tools such as double layer potential theory, commutator estimates and some basic identities.

\subsection{Two classes of holomorphic functions}  We define two classes of holomorphic functions. 
\begin{itemize}
\item [(1)]  Bounded holomorphic functions which decays nontangentially, 

\item [(2)] Periodic holomorphic functions which approaches 0 as $y\rightarrow -\infty$. 

\end{itemize}
\noindent Periodic holomorphic functions are used to explore the periodic water waves system, while bounded holomorphic functions which decays non-tangentially is a good setting for water waves with initial data of the form $X^s$.
For convenience, we introduce the following notation.
\begin{definition}
Denote  
\begin{equation}
\begin{split}
\mathcal{H}ol_{\mathcal{N}}(\Omega(t)):&= \Big\{F(\cdot,t):\Omega(t)\rightarrow \mathbb{C} \text{ bounded holomorphic,}  \text{ decays nontangentially in } \Omega(t).\Big\}
\end{split}
\end{equation}
Denote 
\begin{equation}
\begin{split}
\mathcal{H}ol_{\mathcal{P}}(\Omega^0(t)):=\Big\{\phi(\cdot,t):&\Omega^0(t)\rightarrow \CC~\text{ bounded, holomorphic, } 2\pi \text{ periodic in }\Omega^0(t),\\
&\lim_{\Im z\rightarrow -\infty}\phi(z,t)=0\Big\}.
\end{split}
\end{equation}
\end{definition}
\begin{remark}
Let $f\in L_{loc}^2(\RR)$. If $f=\Phi\circ \zeta$ for some $\Phi\in \mathcal{H}ol_{\mathcal{N}}(\Omega(t))$, then we say $f\in \mathcal{H}ol_{\mathcal{N}}(\Omega(t))$.
\end{remark}

\subsection{Fourier transform} In this subsection we define the Fourier transform on $\RR$ and on $\mathbb{T}:=[-\pi, \pi]$. 
\begin{definition}
Let $f\in L^2(\RR)$, then we Fourier transform of $f$ as 
$$\hat{f}(\xi)=\int_{-\infty}^{\infty} f(x)e^{-2\pi ix\xi}dx.$$
Let $g\in L^2(\mathbb{T})$. Then define Fourier transform of $f$ on $\mathbb{T}$, still denoted by $\hat{g}$:
$$\hat{g}(\xi):=\frac{1}{2\pi}\int_{\TT} g(x)e^{- ix\xi}dx.$$
\end{definition}


\subsection{Function spaces} In this subsection, we define some function spaces that we'll use in this paper.
\begin{definition}
(1) Let $s\geq 0$, we define 
$$H^s(\RR):=\{f\in L^2(\RR): \int_{-\infty}^{\infty}(1+|\xi|^2)^{s}|\hat{f}(\xi)|^2 d\xi<\infty\},$$
\noindent and we define the norm $\norm{\cdot }_{H^s}$ by
$$\norm{f}_{H^s}^2=\int_{-\infty}^{\infty}(1+|\xi|^2)^{s}|\hat{f}(\xi)|^2 d\xi.$$

(2) We define
$$H^s(\mathbb{T}):=\{f\in L^2(\mathbb{T}): \sum_{m\in\ZZ}(1+m^2)^{s}|\hat{f}(m)|^2<\infty\},$$

\noindent and we define the norm 
$$\norm{f}_{H^s(\mathbb{T})}^2:=\sum_{m\in \ZZ}(1+m^2)^{s}|\hat{f}(m)|^2.$$

(3) Let $J=\RR$ or $\mathbb{T}$. Without loss of generality, assume $s\geq 0$ is an integer. Define
$$W^{s,\infty}(J):=\{f\in L^{\infty}(J): \sum_{m=0}^s \norm{\partial_{\alpha}^m f}_{L^{\infty}(J)}<\infty\}.$$
Define the norm
$$\norm{f}_{W^{s,\infty}(J)}:= \sum_{m=0}^s \norm{\partial_{\alpha}^m f}_{L^{\infty}(J)}.$$
\end{definition}
\noindent We'll use the following Sobolev embedding a lot.
\begin{lemma}\label{L2toinfty}
(1) If $s>1/2$, and $f\in H^s(\RR)$, then $f\in L^{\infty}$, and 
$$\norm{f}_{L^{\infty}(\RR)}\leq C(s)\norm{f}_{H^s(\RR)}.$$

(2) If $s>1/2$, and $f\in H^s(\mathbb{T})$, then $f\in L^{\infty}$, and 
$$\norm{f}_{L^{\infty}(\mathbb{T})}\leq C(s)\norm{f}_{H^s(\mathbb{T})}.$$
\end{lemma}

\begin{definition}
Let $s\geq 0$. Let $s_0>3/2$ be fixed. Define
\begin{equation}
    X^{s}:=\{f=f_0+f_1:  f_1\in H^s(\mathbb{R}), f_0\in H^{s+s_0}(\TT)\}.
\end{equation}
Associate $X^s$ with the norm
\begin{equation}
    \norm{f}_{X^s}=\norm{f_0}_{H^{s+s_0}(\mathbb{T})}+\norm{f_1}_{H^s(\mathbb{R})}.
\end{equation}
\end{definition}
\begin{lemma}
Let $s\geq 0$. Then $X^s$ is a Banach space.
\end{lemma}
\begin{remark}
Let $f\in X^s$. The decomposition $f=f_0+f_1$ for $f_0\in H^{s+s_0}(\TT)$ and $f_1\in H^s(\RR)$ is unique.
\end{remark}

\subsection{Hilbert transform and double layer potential} Let $\zeta=\zeta(\alpha,t)$ be a chord-arc for every fixed time $t$, we define the Hilbert transform associated with $\zeta$ by
\begin{equation}
\mathcal{H}_{\zeta}f(\alpha):=\frac{1}{\pi i}p.v.\int_{-\infty}^{\infty} \frac{\zeta_{\beta}(\beta)}{\zeta(\alpha, t)-\zeta(\beta, t)}f(\beta)d\beta.
\end{equation}
\begin{remark}
We'll also use the notation  $\mathcal{H}_{\omega},  \mathcal{H}_{\tilde{\zeta}}, \mathcal{H}_{\tilde{\omega}}$ in this paper, represents Hilbert transform associated with $\omega, \tilde{\zeta}, \tilde{\omega}$, respectively. We denote $\mathbb{H}$ the Hilbert transform associated with $\zeta(\alpha)=\alpha$.
\end{remark}
The double layer potential operator $\mathcal{K}$ associated with $\zeta$ is given by
\begin{equation}
\mathcal{K}_{\zeta}f(\alpha):=p.v.\int_{-\infty}^{\infty} \Re\{\frac{1}{\pi i}\frac{\zeta_{\beta}}{\zeta(\alpha,t)-\zeta(\beta,t)}\} f(\beta)d\beta.
\end{equation}
The adjoint of the double layer potential operator $\mathcal{K}_{\zeta}^{\ast}$ associated with $\zeta$ is defined by
\begin{equation}
\mathcal{K}_{\zeta}^{\ast}f(\alpha):=p.v.\int_{-\infty}^{\infty} \Re\{-\frac{1}{\pi i}\frac{\zeta_{\alpha}}{|\zeta_{\alpha}|}\frac{|\zeta_{\beta}|}{\zeta(\alpha,t)-\zeta(\beta,t)}\} f(\beta)d\beta.
\end{equation}

For periodic functions, we use the following version of Hilbert transform. Let $\Gamma(t)$ be a chord-arc, $\Gamma(t)=\{\omega(\alpha,t):\alpha\in \RR\}$, where $\omega-\alpha$ is periodic. Let $\Omega^0(t)$ be the region below the curve $\Gamma(t)$. Define the periodic Hilbert transform associated with $\Gamma$ as
\begin{equation}
\mathcal{H}_pf(\alpha):=\frac{1}{2\pi i}p.v.\int_{\TT} \omega_{\beta}(\beta)\cot(\frac{1}{2}(\omega(\alpha,t)-\omega(\beta,t))f(\beta)d\beta.
\end{equation}
\begin{remark}
Please note the difference between  $\mathcal{H}_{\omega}$ ad $\mathcal{H}_p$. 
\end{remark}
\noindent The corresponding double layer potential operator $\mathcal{K}_{p}$ is given by
\begin{equation}
\mathcal{K}_p f(\alpha):=p.v. \int_{\TT}\Re\{\frac{1}{2\pi i}\omega_{\beta}(\beta)\cot(\frac{1}{2}(\omega(\alpha,t)-\omega(\beta,t)))\}f(\beta)d\beta.
\end{equation}
The corresponding adjoint $\mathcal{K}_{p}^{\ast}$ of $\mathcal{K}_{p}$ is given by
\begin{equation}
\mathcal{K}_p^{\ast} f(\alpha):=p.v. \int_{\TT}\Re\{-\frac{1}{2\pi i}\frac{\omega_{\alpha}}{|\omega_{\alpha}|}|\omega_{\beta}(\beta)|\cot(\frac{1}{2}(\omega(\alpha,t)-\omega(\beta,t)))\}f(\beta)d\beta.
\end{equation}
\subsubsection{Characterization of holomorphic functions} For holomorphic functions which decay nontangentially, we have the following description.
\begin{lemma}\label{goodgood}
Let $f\in \mathcal{H}ol_{\mathcal{N}}(\Omega(t))$. Then $\mathcal{H}_{\zeta}f$ is defined and 
$$(I-\mathcal{H}_{\zeta})f=0.$$
\end{lemma}
\begin{proof}
This is a consequence of Cauchy's theorem.
\end{proof}

We have the following well-known characterization of periodic holomorphic functions.
\begin{lemma}
Assume $f\in L^2(\TT)$. Then $f\in\mathcal{H}ol_{\mathcal{P}}(\Omega^0(t))$ if and only if 
$$(I-\mathcal{H}_p)f=0.$$
\end{lemma}

We'll use the following boundedness of Hilbert transform and double layer potential operators. Suppose that $\zeta, \omega$ exist on $[0,T_0]$ for some constant $T_0>0$, and satisfy the following chord-arc condition: There exist constants $\alpha_0, \beta_0, \alpha_0', \beta_0'$ such that for all $t\in [0,T_0]$,
\begin{equation}\label{chordarc_zeta}
    \alpha_0|\alpha-\beta|\leq |\zeta(\alpha,t)-\zeta(\beta,t)|\leq \beta_0|\alpha-\beta|,
\end{equation}
and
\begin{equation}\label{chordarc_omega}
    \alpha_0'|\alpha-\beta|\leq |\omega(\alpha,t)-\omega(\beta,t)|\leq \beta_0'|\alpha-\beta|,
\end{equation}

\begin{lemma}\label{layer}
Assume $\zeta(\alpha,t), \omega(\alpha,t)$ satisfy (\ref{chordarc_zeta}) and (\ref{chordarc_omega}), respectively. Then there exist constants $C_1=C_1(\alpha_0,\beta_0)$ and $C_2=C_2(\alpha_0', \beta_0')$ such that 
\begin{equation}
\norm{\mathcal{H}_{\zeta}f}_{L^2(\RR)}\leq C_1\norm{f}_{L^2(\RR)}.
\end{equation}
\begin{equation}
\norm{\mathcal{H}_{p}f}_{L^2(\mathbb{T})}\leq C_2\norm{f}_{L^2(\mathbb{T})}.
\end{equation}
\begin{equation}
\norm{(I-\mathcal{K}_{\zeta})^{-1}f}_{L^2(\RR)}\leq C_1\norm{f}_{L^2(\RR)}.
\end{equation}
\begin{equation}\label{doubleperiodic}
\norm{(I-\mathcal{K}_{p})^{-1}f}_{L^2(\mathbb{T})}\leq C_2\norm{f}_{L^2(\mathbb{T})}.
\end{equation}
\begin{equation}
\norm{(I-\mathcal{K}_{\zeta}^{\ast})^{-1}f}_{L^2(\RR)}\leq C_1\norm{f}_{L^2(\RR)}.
\end{equation}
\begin{equation}\label{doubleperiodic}
||(I-\mathcal{K}_{p}^{\ast})^{-1}f||_{L^2(\mathbb{T})}\leq C_2||f||_{L^2(\mathbb{T})}.
\end{equation}
\end{lemma}
\begin{proof}
See for example Chapter 4 of \cite{taylor2000tools} for the case on $L^2(\RR)$. The case on $L^2(\mathbb{T})$ can be proved in a similar way. 
\end{proof}
\begin{remark}
Because we consider smooth and small solution, indeed we have for real function $f$, for $\omega$ such that $\omega-\alpha$ small,  an easy calculation gives
$$\norm{\mathcal{K}_p f}_{H^s(\mathbb{T})}\leq C\epsilon \norm{f}_{H^s(\mathbb{T})}.$$
From this, the boundedness of $(I-\mathcal{K}_p)^{-1}$ follows immediately.
\end{remark}

\subsection{Some basic identities}
\begin{lemma}\label{lemmaseven}
Let $T_0>0$ be fixed. Assume $D_t\zeta, \zeta_{\alpha}-1\in C^1([0,T_0];X^1)$, $f\in X^2$. We have 
\begin{align}
[D_t,\mathcal{H}_{\zeta}]f=&[D_t\zeta,\mathcal{H}_{\zeta}]\frac{\partial_{\alpha}D_tf}{\zeta_{\alpha}}\\
[D_t^2,\mathcal{H}_{\zeta}]f=&[D_t^2\zeta,\mathfrak{H}]\frac{f_{\alpha}}{\zeta_{\alpha}}+2[D_t\zeta,\mathcal{H}_{\zeta}]\frac{\partial_{\alpha}D_tf}{\zeta_{\alpha}}\\
&-\frac{1}{\pi i}\int_{-\infty}^{\infty} (\frac{D_t\zeta(\alpha,t)-D_t\zeta(\beta,t)}{\zeta(\alpha,t)-\zeta(\beta,t)})^2f_{\beta}(\beta,t)d\beta\\
[A\partial_{\alpha},\mathcal{H}_{\zeta}]f=&[A\zeta_{\alpha},\mathcal{H}_{\zeta}]\frac{f_{\alpha}}{\zeta_{\alpha}},\quad \quad  \partial_{\alpha}\mathcal{H}_{\zeta}f=\zeta_{\alpha}\mathcal{H}_{\zeta}\frac{f_{\alpha}}{\zeta_{\alpha}}\\
[D_t^2-iA\partial_{\alpha},\mathcal{H}_{\zeta}]f=&2[D_t\zeta,\mathcal{H}_{\zeta}]\frac{\partial_{\alpha}D_tf}{\zeta_{\alpha}}-\frac{1}{\pi i}\int_{-\infty}^{\infty} (\frac{D_t\zeta(\alpha,t)-D_t\zeta(\beta,t)}{\zeta(\alpha,t)-\zeta(\beta,t)})^2f_{\beta}(\beta,t)d\beta \label{abcde}
\end{align}
\end{lemma}
\noindent For proof, see \cite{Wu2009}.

\begin{lemma}\label{lemmmaeight}
Let $T_0>0$ be fixed. Assume that $f\in C^2_{t,x}([0,T_0]\times \mathbb{T})$. We have 
\begin{equation}\label{U1}
[\partial_{\alpha}, \mathcal{H}_p]f=[\omega_{\alpha}, \mathcal{H}_p]\frac{f_{\alpha}}{\omega_{\alpha}}.
\end{equation}
\begin{equation}\label{U2}
[g\partial_{\alpha}, \mathcal{H}_p]f=[g\omega_{\alpha}, \mathcal{H}_p]\frac{f_{\alpha}}{\omega_{\alpha}}, \quad \quad \forall~g\in L^{\infty}(\mathbb{T}).
\end{equation}
\begin{equation}\label{U3}
[\partial_t, \mathcal{H}_p]f=[\omega_t, \mathcal{H}_p]\frac{f_{\alpha}}{\omega_{\alpha}}.
\end{equation}

\begin{equation}\label{U4}
[D_t^0, \mathcal{H}_p]f=[D_t^0\omega, \mathcal{H}_p]\frac{f_{\alpha}}{\omega_{\alpha}}.
\end{equation}

\begin{equation}\label{U5}
\begin{split}
[(D_t^0)^2, \mathcal{H}_p]f=&[(D_t^0)^2\omega, \mathcal{H}_p]\frac{f_{\alpha}}{\omega_{\alpha}}+2[D_t^0\omega, \mathcal{H}_p]\frac{\partial_{\alpha}D_t^0f}{\omega_{\alpha}}\\
&-\frac{1}{4\pi i}\int_{\TT}\Big(\frac{D_t^0\omega(\alpha)-D_t^0\omega(\beta)}{\sin(\frac{\pi}{2}(\omega(\alpha)-\omega(\beta)))}\Big)^2 f_{\beta}d\beta.
\end{split}
\end{equation}

\begin{equation}\label{U6}
\begin{split}
[(D_t^0)^2-iA_0\partial_{\alpha}, \mathcal{H}_p]f=&2[D_t^0\omega, \mathcal{H}_p]\frac{\partial_{\alpha}D_t^0f}{\omega_{\alpha}}-\frac{1}{4\pi i}\int_{\TT}\Big(\frac{D_t^0\omega(\alpha)-D_t^0\omega(\beta)}{\sin(\frac{\pi}{2}(\omega(\alpha)-\omega(\beta)))}\Big)^2 f_{\beta}d\beta.
\end{split}
\end{equation}
\end{lemma}
\begin{proof}
Note that
\begin{align*}
\mathcal{H}_pf(\alpha)=&-\frac{1}{ \pi i}\int_{\TT}\partial_{\beta}\log \sin(\frac{1}{2}(\omega(\alpha)-\omega(\beta))) f(\beta)d\beta\\
=&\frac{1}{\pi i}\int_{\TT}\log \sin(\frac{1}{2}(\omega(\alpha)-\omega(\beta))) f_{\beta}(\beta)d\beta.
\end{align*}
Using this, we obtain (\ref{U1}). (\ref{U2}) is proved exactly the same way.  (\ref{U3}) is prove similarly.  (\ref{U4}) is a direct consequence of (\ref{U2}) and (\ref{U3}).

To prove (\ref{U5}), by changing of variable, it suffices to prove 
\begin{equation}\label{V2}
\begin{split}
&[\partial_t^2, \mathcal{H}_p]f\\
=&[\partial_t^2\omega, \mathcal{H}_p]\frac{f_{\alpha}}{\omega_{\alpha}}+2[\omega_t, \mathcal{H}_p]\frac{f_{t\alpha}}{\omega_{\alpha}}-\frac{1}{4\pi i}\int_{\TT}\Big(\frac{\omega_t(\alpha)-\omega_t(\beta)}{\sin(\frac{\pi}{2}(\omega(\alpha)-\omega(\beta)))}\Big)^2 f_{\beta}d\beta
.
\end{split}
\end{equation}
(\ref{V2}) is a direct consequence of (\ref{U3}) and the following identity:
$$[\partial_t^2, \mathcal{H}_p]f=\partial_t[\partial_t, \mathcal{H}_p]f+[\partial_t, \mathcal{H}_p]\partial_t f.$$
\end{proof}

\begin{remark}
The identities in lemma \ref{lemmaseven} and lemma \ref{lemmmaeight} hold true in BMO sense.
\end{remark}

\subsection{Basic commutator estimates}
Let $m\geq 1$ be an integer. Define
\begin{equation}
S_1(A,f)=\int_{\RR} \prod_{j=1}^m \frac{A_j(\alpha)-A_j(\beta)}{\gamma_j(\alpha)-\gamma_j(\beta)} \frac{f(\beta)}{\gamma_0(\alpha)-\gamma_0(\beta)}d\beta.
\end{equation}

\begin{equation}
S_2(A,f)=\int_{\RR} \prod_{j=1}^m \frac{A_j(\alpha)-A_j(\beta)}{\gamma_j(\alpha)-\gamma_j(\beta)} f_{\beta}(\beta)d\beta.
\end{equation}
\noindent We have the following comutator estimates, which can be found in \cite{Totz2012}, \cite{Wu2009}.
\begin{proposition}\label{singular}
(1) Assume each $\gamma_j$ satisfies the chord-arc condition
\begin{equation}\label{chordarc1}
C_{0,j}|\alpha-\beta|\leq |\gamma_j(\alpha)-\gamma_j(\beta)|\leq C_{1,j}|\alpha-\beta|,
\end{equation}
where $C_{0,j}, C_{1,j}$ are positive constants and $C_{0,j}\leq C_{1,j}, 1\leq j\leq m$. 
Then both $\norm{S_1(A,f)}_{L^2}$ and $\norm{S_2(A,f)}_{L^2}$ are bounded by
$$C\prod_{j=1}^m \norm{A_j'}_{X_j}\norm{f}_{X_0},$$
where one of the $X_0, X_1, ...X_m$ is equal to $L^2$ and the rest are $L^{\infty}$. The constant $C$ depends on $\norm{\gamma_j'}_{L^{\infty}}^{-1}, j=1,..,m$.

(2) Let $s\geq 3$ be given, and suppose chord-arc condition (\ref{chordarc1}) holds for each $\gamma_j$, assume $\gamma_j-1\in H^{s-1}$.then 
$$\norm{S_2(A,f)}_{H^s}\leq C\prod_{j=1}^m \norm{A_j'}_{Y_j}\norm{f}_Z,$$
where for all $j=1,...,m$, $Y_j=H^{s-1}$ or $W^{s-2,\infty}$ and $Z=H^s$ or $W^{s-1,\infty}$. The constant $C$ depends on $\norm{\gamma_j'-1}_{H^{s-1}}, j=1,...,m$.
\end{proposition}
\vspace*{1ex}

\noindent Let $m\geq 1$ be  integer. Define

\begin{equation}
    \tilde{S}_1(f,g):=[g,\mathcal{H}_p]\frac{f_{\alpha}}{\omega_{\alpha}}.
\end{equation}
\begin{equation}
\tilde{S}_2(A,f)=\int_{\TT} \prod_{j=1}^m \frac{A_j(\alpha)-A_j(\beta)}{\sin(\frac{1}{2}(\omega(\alpha,t)-\omega(\beta,t))} f_{\beta}(\beta)d\beta.
\end{equation}
We have 
\begin{proposition}\label{singularperiodic}
Assume $\omega$ satisfies the chord-arc condition (\ref{chordarc_omega}). Then
\begin{equation}
    \norm{\tilde{S}_1(f,g)}_{H^s(\TT)}\leq C\norm{f}_{H^s(\TT)}\norm{g}_{H^s(\TT)},
\end{equation}
\begin{equation}
    \norm{\tilde{S}_2(A,f)}_{H^s}\leq C\prod_{j=1}^m\norm{A_j'}_{H^{s-1}(\TT)}\norm{f}_{H^s},
\end{equation}
where the constant $C$ depends on $\norm{\omega_{\alpha}}_{H^{s-1}(\TT)}, j=1,...,m$.
\end{proposition}
\begin{proof}
This can be derived from Proposition \ref{singular}.
\end{proof}

\section{Water wave system in periodic setting}
In this section, we use S. Wu's method (see  (\cite{Wu2009}, \cite{Wu2011}) to give a sketch of proof of long time existence of water wave system in periodic setting. Long time existence of periodic water waves is not new, the methods in \cite{Wu2009}\cite{Ionescu2015}\cite{AlazardDelort}\cite{HunterTataruIfrim1} all imply cubic lifespan for 2d gravity water waves with small initial data. We use S. Wu's method to sketch the proof here because we need to bound the quantities such as $b_0, A_0, \omega-\alpha, D_t^0\omega, (D_t^0)^2\omega$ on time scale $O(\epsilon^{-2})$ in later sections, and also because we need to use this method to prove the remainder term $r_0$ remains small for sufficiently long times.  Solution of this periodic water waves system has the same boundary values at spatial infinity as a water waves system whose initial data is assumed to be  in $X^s$. 
\subsection{Notation}
$D_t^0=\partial_t+b_0\partial_{\alpha}$, for some function $b_0$. $D_t=\partial_t+b\partial_{\alpha}$. $\Omega^0(t)$ be the region bounded above by the graph $\omega$.

\subsection{Set up of the  periodic water waves system}

Consider the periodic water waves system
\begin{equation}\label{special}
\begin{cases}
((D_t^0)^2-iA_0\partial_{\alpha})\omega=-i\\
\bar{\omega}-\alpha, D_t^0\bar{\omega}\in \mathcal{H}ol_{\mathcal{P}}(\Omega^0(t)).
\end{cases}
\end{equation}
Let $\kappa_0:\RR\rightarrow \RR$ be a diffeomorphism given by 
\begin{equation}\label{kappa0}
(\kappa_0)_t=b_0\circ\kappa_0.
\end{equation}
Then the change of coordinates $\alpha\mapsto \kappa_0(\alpha)$ brings the system (\ref{special}) back to Lagrangian coordinates, namely,  with $(a_0\partial_{\alpha}\kappa_0)\circ \kappa_0^{-1}=A_0$, we have 
\begin{equation}\label{back}
\begin{cases}
(\omega\circ \kappa_0)_{tt}-ia_0\partial_{\alpha}\omega\circ\kappa_0=-i\\
(\overline{\omega\circ\kappa_0})_t\in \mathcal{H}ol_{\mathcal{P}}(\Omega^0(t)).
\end{cases}
\end{equation}
Take $\partial_t$ on both sides of the above equation, we get :
\begin{equation}\label{ttttt}
(\partial_t^2-ia_0\partial_{\alpha})(\omega\circ \kappa_0)_t=i(a_0)_t \partial_{\alpha}\omega\circ\kappa_0=\frac{(a_0)_t}{a
_0}ia_0\omega_{\alpha}\circ\kappa.
\end{equation}
Precomposing with $\kappa_0^{-1}$ on both sides of the above equation, we obtain,
\begin{equation}\label{quaqua}
((D_t^0)^2-iA_0\partial_{\alpha})D_t^0\omega=i\frac{(a_0)_t}{a_0}\circ\kappa_0^{-1}A_0\omega_{\alpha}.
\end{equation}

Similar to the derivation of formula for $b$, $A$, $\frac{a_t}{a}\circ\kappa^{-1}$ in \cite{Wu2009}, we can derive formula for $b_0, A_0, \frac{(a_0)_t}{a_0}\circ \kappa_0^{-1}$.  We give the details for the derivation of formula for $b_0$. Formula for $A_0, \frac{(a_0)_t}{a_0}\circ \kappa_0^{-1}$ follow in a similar way.

\subsection{Formula for $b_0$}
We have 
\begin{equation}\label{bb0}
(I-\mathcal{H}_p)b_0=-[D_t^0\omega, \mathcal{H}_p]\frac{\bar{\omega}_{\alpha}-1}{\omega_{\alpha}}.
\end{equation}
\begin{proof}
By assumption, $\bar{\omega}-\alpha=\Phi_0(\omega(\alpha,t),t),  D_t\bar{\omega}=\Psi_0(\omega(\alpha,t),t)$, where $\Phi_0, \Psi_0\in \mathcal{H}ol_{\mathcal{P}}$. We have 
\begin{equation}\label{adw}
\begin{split}
D_t^0\bar{\omega}=&(\partial_t+b_0\partial_{\alpha})(\bar{\omega}-\alpha)+b_0\\
=& D_t^0\omega(\Phi_0)_{\omega}\circ \omega+ (\Phi_0)_t\circ \omega+b_0.
\end{split}
\end{equation}
Note that $D_t^0\bar{\omega}, (\Phi_0)_t\circ \omega\in \mathcal{H}ol_{\mathcal{P}}$, we have 
$$(I-\mathcal{H}_p)D_t\bar{\omega}=0,\quad \quad (I-\mathcal{H}_p)(\Phi_0)_t\circ\omega=0.$$
Also note that 
$$(\Phi_0)_{\omega}\circ \omega=\frac{\bar{\omega}_{\alpha}-1}{\omega_{\alpha}}.$$
Apply $I-\mathcal{H}_p$ on both sides of  (\ref{adw}),  we have 
\begin{align*}
(I-\mathcal{H}_p)b_0=(I-\mathcal{H}_p)D_t^0\omega \frac{\bar{\omega}_{\alpha}-1}{\omega_{\alpha}}=-[D_t^0\omega, \mathcal{H}_p]\frac{\bar{\omega}_{\alpha}-1}{\omega}.
\end{align*}
\end{proof}

\subsection{Formula for $A_0$}
\begin{equation}\label{AA0}
(I-\mathcal{H}_p)(A_0-1)=i[(D_t^0)^2\omega,\mathcal{H}_p]\frac{\bar{\omega}_{\alpha}-1}{\omega_{\alpha}}+i[(D_t^0)\omega,\mathcal{H}_p]\frac{\partial_{\alpha}(D_t^0)\bar{\omega}}{\omega_{\alpha}}.
\end{equation}

\subsection{Formula for $\frac{(a_0)_t}{a_0}\circ \kappa_0^{-1}$}
We have
\begin{equation}\label{b0b0b0}
\begin{split}
-i(I-\mathcal{H}_p)\Big(A_0 \bar{\omega}_{\alpha}(\frac{(a_0)_t}{a_0})\circ \kappa_0^{-1}\Big)=& 2[(D_t^0)^2\omega, \mathcal{H}_{p}]\frac{\partial_{\alpha}D_t^0\bar{\omega}}{\omega_{\alpha}}+2[D_t^0\omega, \mathcal{H}_{p}]\frac{\partial_{\alpha}(D_t^0)^2\bar{\omega}}{\omega_{\alpha}}\\
&-\frac{1}{4\pi i}\int_{\TT} \Big(\frac{D_t^0\omega(\alpha)-D_t^0\omega(\beta)}{\sin(\frac{\pi}{2}(\omega(\alpha)-\omega(\beta)))}\Big)^2\partial_{\beta}D_t^0\bar{\omega}d\beta.
\end{split}
\end{equation}


\subsection{Local well-posedness} (\ref{special}) is a fully nonlinear system.
To prove local well-posedness, one way is to quasilinearize this system. In \cite{Wu1997}, S. Wu showed that for water waves which vanish at infinity,  one can quasilinearize (\ref{special}) by just taking one derivative in time. For periodic case, we have quasilinearization
\begin{equation}\label{quasi1}
\begin{cases}
((D_t^0)^2-iA_0\partial_{\alpha})D_t^0\omega=i \frac{(a_0)_t}{a_0}\circ\kappa^{-1}\omega_{\alpha},\\
D_t^0\bar{\omega}\in \mathcal{H}ol_{\mathcal{P}}(\Omega^0(t)).
\end{cases}
\end{equation}
By formula (\ref{bb0}), (\ref{AA0}) together with Proposition \ref{singularperiodic}, the system (\ref{quasi1}) is a quasilinear system.  So the local well-posedness can be obtained similar to the work of S. Wu\cite{Wu1997}. We omit the details and state the result as follows:
\begin{thm}[local well-posedness]\label{localperiodic}
Let $s\geq 4$. Assume that $\omega(0), D_t^0\omega(0), (D_t^0)^2\omega(0)$  satify the compatiability condition, i.e., $\bar{\omega}(0)-\alpha,  D_t^0\bar{\omega}(0)\in \mathcal{H}ol_{\mathcal{P}}(\Omega^0(0))$, and 
$$(I-\mathcal{H}_{\mathcal{P}})(A_0(0)-1)=i[(D_t^0)^2\omega(0),\mathcal{H}_{\mathcal{P}}]\frac{\partial_{\alpha}\bar{\omega}(0)-\alpha}{\partial_{\alpha}\omega(0)}+i[D_t^0\omega(0),\mathcal{H}_{\mathcal{P}}]\frac{\partial_{\alpha} D_t^0\bar{\omega}(0)}{\partial_{\alpha}\omega(0)}.$$
Assume $(\partial_{\alpha}\omega(0)-1, D_t^0\omega(0), (D_t^0)^2\omega(0))\in H^{s}(\mathbb{T})\times H^{s+1/2}(\mathbb{T})\times H^s(\mathbb{T})$. Then there is $T>0$ depending on the norm of the initial data such that  the water waves system (\ref{special}) has a unique solution $\omega=\omega(\alpha,t)$ for $t\in [0,T]$, satisfying 
\begin{equation}
(\omega_{\alpha}(t)-1, D_t^0\omega(t), (D_t^0)^2\omega(t))\in C([0,T]; H^{s}(\mathbb{T})\times H^{s+1/2}(\mathbb{T})\times H^s(\mathbb{T})).
\end{equation}
Moreover, if $T_{max}$ is the supremum over all such times $T$, then either $T_{max}=\infty$, or $T_{max}<\infty$, but
\begin{equation}
\begin{split}
\lim_{t\uparrow T_{max}}  &\norm{(D_t^0\omega(t), (D_t^0)^2\omega(t))}_{H^{s}(\mathbb{T})\times H^s(\mathbb{T})}=\infty,
\end{split}
\end{equation}
or
\begin{equation}
    \sup_{\alpha\neq \beta}\Big|\frac{\omega(\alpha,t)-\omega(\beta,t)}{\alpha-\beta}\Big|+\sup_{\alpha\neq \beta}\Big|\frac{\alpha-\beta}{\omega(\alpha,t)-\omega(\beta,t)}\Big|=\infty,
\end{equation}
or \begin{equation}
    \lim_{t\uparrow T_{max}}\inf_{\alpha\in \RR}A_0(\alpha,t)|\omega_{\alpha}(\alpha,t)|\leq 0.
\end{equation}
\end{thm}

\subsection{Long time behavior}
We use S. Wu's method (\cite{Wu2009}) to study the long time behavior of periodic water waves with small initial data. Consider the quantity $\tilde{\theta}_0:=(I-\mathcal{H}_{\mathcal{P}})(\omega-\bar{\omega})$ and $\tilde{\sigma}_0:=D_t^0(I-\mathcal{H}_{\mathcal{P}})(\omega-\bar{\omega})$. One can show that 
\begin{equation}
    \norm{\partial_{\alpha}\tilde{\theta}_0}_{H^{s'}(\TT)}\approx \norm{\partial_{\alpha}\omega-1}_{H^{s'}(\TT)},\quad \quad \norm{\tilde{\sigma}_0}_{H^{s'+1/2}(\TT)}\approx \norm{D_t^0\omega}_{H^{s'+1/2}(\TT)}.
\end{equation}
Apply lemma \ref{lemmmaeight}, we obtain 
\begin{equation}
\begin{split}
&((D_t^0)^2-iA_0)\tilde{\theta}_0\\
=&-2[D_t^0\omega, \mathcal{H}_p\frac{1}{\omega_{\alpha}}+\bar{\mathcal{H}}_p\frac{1}{\bar{\omega}_{\alpha}}]\frac{\partial_{\alpha}D_t^0\omega}{\omega_{\alpha}}+\frac{1}{4 \pi i}\int_{\TT} \Big(\frac{D_t^0\omega(\alpha)-D_t^0\omega(\beta)}{\sin(\frac{\pi}{2}(\omega(\alpha)-\omega(\beta)))}\Big)^2\partial_{\beta}(\omega-\bar{\omega})d\beta\\
:=& G_0,
\end{split}
\end{equation}
and
\begin{equation}
((D_t^0)^2-iA_0\partial_{\alpha})\tilde{\sigma}_0=D_t^0 G_0+[(D_t^0)^2-iA_0\partial_{\alpha}, D_t^0]\tilde{\theta}_0.
\end{equation}
Note that 
\begin{equation}
[(D_t^0)^2-iA_0\partial_{\alpha}, D_t^0]\tilde{\theta}_0=i(\frac{(a_0)_t}{a_0})\circ\kappa_0^{-1} A_0\partial_{\alpha}\tilde{\theta}_{0}, \quad \quad 
\end{equation}
which is cubic. So we can prove long time existence.  We state the result as follows.
\begin{thm}[Long time existence]\label{longperiodic}
Let $s'\geq 6$. Let $\omega(0), D_t^0\omega(0), (D_t^0)^2\omega(0)$ satisfy the compatibility  condition as in Theorem \ref{localperiodic}. There exists $\epsilon_0=\epsilon_0(s')>0$ such that for all $\epsilon<\epsilon_0$, if 
$$\norm{(\partial_{\alpha}\omega(0)-1, D_t^0\omega(0), (D_t^0)^2\omega(0))}_{H^{s'}(\mathbb{T})\times H^{s'+1/2}(\mathbb{T})\times H^{s'}(\mathbb{T})}\leq \epsilon,$$
then there exists a positive constant $C_0=C_0(s')$ such that the solution to (\ref{special}) exists on $[0, C_0\epsilon^{-2}]$, and 
$$\sup_{t\in [0, C_0\epsilon^{-2}]}\norm{(\partial_{\alpha}\omega(t)-1, D_t^0\omega(t), (D_t^0)^2\omega(t))}_{H^{s'}(\mathbb{T})\times H^{s'+1/2}(\mathbb{T})\times H^{s'}(\mathbb{T})}\leq 2\epsilon.$$
\end{thm}
As a consequence, use formula (\ref{bb0}), (\ref{AA0}), (\ref{b0b0b0}),  and use  lemma \ref{layer}, we obtain bounds for $b_0, A_0, \frac{(a_0)_t}{a_0}\circ\kappa_0^{-1}$.
\begin{corollary}
With the assumptions in Theorem \ref{longperiodic}, there exists $C>0$ independent of $\epsilon$ and $\omega$ such that for all $t\in [0, C_0\epsilon^{-2}]$, 
\begin{equation}\label{wanli}
\norm{b_0(t)}_{H^{s'}(\mathbb{T})}+\norm{A_0(t)-1}_{H^{s'}(\mathbb{T})}+\norm{\frac{(a_0)_t}{a_0}\circ \kappa_0^{-1}}_{H^{s'}(\mathbb{T})}\leq C\epsilon^2.
\end{equation}
In particular, by Sobolev embedding, we have 
\begin{equation}\label{wuqiong}
\norm{b_0(t)}_{W^{s'-1,\infty}(\mathbb{T})}+\norm{A_0(t)-1}_{W^{s'-1,\infty}(\mathbb{T})}+\norm{\frac{(a_0)_t}{a_0}\circ \kappa_0^{-1}}_{W^{s'-1,\infty}(\mathbb{T})}\leq C\epsilon^2.
\end{equation}
\end{corollary}
\begin{proof}
For (\ref{wanli}), take real part of equations (\ref{bb0}), (\ref{AA0}), and (\ref{b0b0b0}), respectively,  then use (\ref{doubleperiodic}). For (\ref{wuqiong}), use lemma \ref{L2toinfty} and (\ref{wanli}).
\end{proof}
\noindent Another consequence is the following. 
\begin{corollary}\label{zerozero}
The quantities $\mathcal{H}_{\omega}(\bar{\omega}-\alpha)$, $\mathcal{H}_{\omega}D_t^0\bar{\omega}$, $\mathcal{H}_{\omega}\frac{\bar{\omega}_{\alpha}-1}{\omega_{\alpha}}$ and $\mathcal{H}_{\omega}\frac{\partial_{\alpha}D_t^0\bar{\omega}}{\omega_{\alpha}}$ are well-defined, and 
\begin{equation}
(I-\mathcal{H}_{\omega})(\bar{\omega}-\alpha)=0.
\end{equation}
\begin{equation}
(I-\mathcal{H}_{\omega})D_t^0\bar{\omega}=0.
\end{equation}
\begin{equation}
(I-\mathcal{H}_{\omega})\frac{\bar{\omega}_{\alpha}-1}{\omega_{\alpha}}=0.
\end{equation}
\begin{equation}
(I-\mathcal{H}_{\omega})\frac{\partial_{\alpha}D_t^0\bar{\omega}}{\omega_{\alpha}}=0.
\end{equation}
\end{corollary}
\begin{proof}
The functions $\bar{\omega}-\alpha$, $\frac{\bar{\omega}_{\alpha}-1}{\omega_{\alpha}}$, $D_t^0\bar{\omega}$, $\frac{\partial_{\alpha}D_t^0\bar{\omega}}{\omega_{\alpha}}$ are in $\mathcal{H}ol_{\mathcal{N}}(\Omega^0(t))$. Then the corollary follows from lemma \ref{goodgood}.
\end{proof}

\begin{remark}
Note that in the above corollary, the Hilbert transform is defined by
$$\mathcal{H}_{\omega}f(\alpha)=\frac{1}{\pi i}p.v. \int_{-\infty}^{\infty} \frac{\omega_{\beta}}{\omega(\alpha)-\omega(\beta)}f(\beta)d\beta.$$
For a bounded smooth function $f$, $\mathcal{H}_{\omega}f$ does not always define an $L^{\infty}$ function. In such cases, $\mathcal{H}_{\omega}f$ is interpreted in $BMO$ sense.
\end{remark}

\noindent \textbf{Notation} Denote 
$$\mathcal{P}:= D_t^2-iA\partial_{\alpha}, \quad \quad \mathcal{P}_0:=(D_t^0)^2-iA_0\partial_{\alpha}.$$

\section{Water waves system with data in $X^s$}\label{nnnnn}
We consider long time existence of non-vanishing water waves system with data of the form $X^s:=H^s(\RR)+H^{s'}(\mathbb{T})$. This is a natural generalization of the current known long time existence results for water waves. Moreover,  if we restrict ourselves to smooth water waves, then this class of water waves has included many physically relevant situations.

Let $(\omega, D_t^0\omega, (D_t^0)^2\omega)$ be the solution to the periodic water wave system in the previous section. We consider the class of solutions of water wave system with boundary values $\omega(\alpha,t)$ at $\alpha=\pm\infty$, i.e., we consider
\begin{equation}\label{water}
\begin{cases}
(D_t^2-iA\partial_{\alpha})\zeta=-i\\
\bar{\zeta}-\alpha,~D_t\bar{\zeta}\in \mathcal{H}ol_{\mathcal{N}}(\Omega(t))\\
\lim_{\alpha\rightarrow \pm\infty}(\zeta-\omega)=0\\
\lim_{\alpha\rightarrow \pm\infty}(D_t\zeta-D_t^0\omega)=0.
\end{cases}
\end{equation}
Recall that $D_t=\partial_t+b\partial_{\alpha}$. $b$ and $A$ cannot be arbitrary. Instead, they are determined by the water wave system and the constraint that $\bar{\zeta}-\alpha, D_t\bar{\zeta}\in \mathcal{H}ol_{\mathcal{N}}(\Omega(t))$.  Denote 
$$\bar{\zeta}-\alpha=\Phi(\zeta(\alpha,t),t),\quad \quad D_t\bar{\zeta}=\Psi(\zeta(\alpha,t),t),$$
where $\Phi, \Psi\in \mathcal{H}ol_{\mathcal{N}}(\Omega(t))$.

In the following  subsections, we derive formula for $b, A$ and $\frac{a_t}{a}\circ \kappa^{-1}$, etc. The derivation is almost the same as that in \cite{Wu2009}, except that we are dealing with functions which are not necessarily vanishing at $\infty$.
\subsection{Formula for $b$ and $b_0$}
\begin{equation}\label{formulab}
(I-\mathcal{H}_{\zeta})b=-[D_t\zeta, \mathcal{H}_{\zeta}]\frac{\bar{\zeta}_{\alpha}-1}{\zeta_{\alpha}},
\end{equation}
and
\begin{equation}\label{formulab0}
(I-\mathcal{H}_{\omega})b_0=-[D_t^0b_0, \mathcal{H}_{\omega}]\frac{\partial_{\alpha}\bar{\omega}-1}{\omega_{\alpha}}.
\end{equation}
\begin{proof}
We have 
$$D_t\bar{\zeta}=b+\Phi_t\circ \zeta+D_t\zeta\Phi_{\zeta}\circ \zeta=b+\Phi_t\circ \zeta+D_t\zeta\frac{\bar{\zeta}_{\alpha}-1}{\zeta_{\alpha}}.$$
Since $D_t\bar{\zeta}, \Phi_t, \Phi_{\zeta}\in \mathcal{H}_{\mathcal{N}}(\Omega(t))$, if we apply $(I-\mathcal{H}_{\zeta})$ on both sides of the above equation and use lemma \ref{goodgood}, we have
\begin{align*}
0=&(I-\mathcal{H}_{\zeta})b+(I-\mathcal{H}_{\zeta})D_t\zeta \frac{\bar{\zeta}_{\alpha}-1}{\zeta_{\alpha}}\\
=& (I-\mathcal{H}_{\zeta})b+[D_t\zeta,\mathcal{H}_{\zeta}] \frac{\bar{\zeta}_{\alpha}-1}{\zeta_{\alpha}}.
\end{align*}
So we have\footnote{Note that $\mathcal{H}_{\zeta}b$ and $\mathcal{H}_{\omega}b_0$ are defined as BMO functions. Moreover, $\mathcal{H}_{\zeta}b-\mathcal{H}_{\omega}b_0\in H^s(\mathbb{R})$. Similar properties hold for other quantities such as $A, A_0$. }
\begin{equation}
(I-\mathcal{H}_{\zeta})b=-[D_t\zeta, \mathcal{H}_{\zeta}]\frac{\bar{\zeta}_{\alpha}-1}{\zeta_{\alpha}}.
\end{equation}
So we obtain (\ref{formulab}).
Use completely the same proof, we have 
\begin{equation}
(I-\mathcal{H}_{\omega})b_0=-[D_t^0b_0, \mathcal{H}_{\omega}]\frac{\partial_{\alpha}\bar{\omega}-1}{\omega_{\alpha}}.
\end{equation}
So we obtain (\ref{formulab0}).
\end{proof}

\subsection{Formula for $A$}
Use the water wave system, we have $D_t^2\bar{\zeta}+iA\bar{\zeta}_{\alpha}=i$. Note that 
\begin{equation}\label{d1d2}
D_t^2\bar{\zeta}=D_t D_t\bar{\zeta}=D_t \Psi\circ\zeta=\Psi_t\circ \zeta+D_t\zeta \Psi_{\zeta}\circ \zeta  \quad \quad \Psi_{\zeta}=\frac{\partial_{\alpha}D_t\bar{\zeta}}{\zeta_{\alpha}}.
\end{equation}
\begin{equation}\label{d1d3}
iA\bar{\zeta}_{\alpha}=iA+iA\partial_{\alpha}(\bar{\zeta}-\alpha)=iA+iA\zeta_{\alpha}\frac{\bar{\zeta}_{\alpha}-1}{\zeta_{\alpha}}=iA+(D_t^2\zeta+i)\frac{\bar{\zeta}_{\alpha}-1}{\zeta_{\alpha}}.
\end{equation}
Since $\Psi_t, \frac{\bar{\zeta}_{\alpha}-1}{\zeta_{\alpha}}\in \mathcal{H}ol_{\mathcal{N}}$, by lemma \ref{goodgood}, we have 
\begin{equation}\label{d1d4}
(I-\mathcal{H}_{\zeta})\Psi_{\zeta}\circ \zeta=0, \quad \quad (I-\mathcal{H}_{\zeta})\Psi_t\circ \zeta=0, \quad \quad (I-\mathcal{H}_{\zeta})\frac{\bar{\zeta}_{\alpha}-1}{\zeta_{\alpha}}=0.
\end{equation}
Apply $(I-\mathcal{H}_{\zeta})$ on both sides of $D_t^2\bar{\zeta}+iA\bar{\zeta}_{\alpha}=i$,  use (\ref{d1d2}), (\ref{d1d3}), and (\ref{d1d4}), we have 
\begin{equation}\label{d1d5}
(I-\mathcal{H}_{\zeta})D_t\zeta  \frac{\partial_{\alpha}D_t\bar{\zeta}}{\zeta_{\alpha}}+(I-\mathcal{H}_{\zeta})i(A-1)+(I-\mathcal{H}_{\zeta})D_t^2\zeta\frac{\bar{\zeta}_{\alpha}-1}{\zeta_{\alpha}}=0.
\end{equation}
Use (\ref{d1d4}) again, we can write (\ref{d1d5}) in commutator form:
\begin{equation}\label{formulaA}
(I-\mathcal{H}_{\zeta})(A-1)=i[D_t\zeta, \mathcal{H}_{\zeta}]\frac{\partial_{\alpha}D_t\bar{\zeta}}{\zeta_{\alpha}}+i[D_t^2\zeta, \mathcal{H}_{\zeta}]\frac{\bar{\zeta}_{\alpha}-1}{\zeta_{\alpha}}.
\end{equation}
\noindent Use completely the same argument, we get a nonlocal version of formula for $A_0$:
\begin{equation}\label{nonlocalA0}
(I-\mathcal{H}_{\omega})(A_0-1)=i[D_t^0\omega, \mathcal{H}_{\omega}]\frac{\partial_{\alpha}D_t^0\bar{\omega}}{\omega_{\alpha}}+i[(D_t^0)^2\omega, \mathcal{H}_{\omega}]\frac{\bar{\omega}_{\alpha}-1}{\omega_{\alpha}}.
\end{equation}
\begin{remark}
Formula (\ref{formulaA}) implies that $A-1$ is quadratic. 
\end{remark}

\subsection{Formula for $\frac{a_t}{a}\circ \kappa^{-1}$ and $\frac{(a_0)_t}{a_0}\circ\kappa_0^{-1}$}
Apply $D_t$ on both sides of $(D_t^2+iA\partial_{\alpha})\zeta=-i$, we have 
\begin{equation}\label{zheng}
(D_t^2+iA\partial_{\alpha})D_t\bar{\zeta}=-i\frac{a_t}{a}\circ \kappa^{-1}A\bar{\zeta}_{\alpha}.
\end{equation}
Apply $(I-\mathcal{H}_{\zeta})$ on both sides of (\ref{zheng}), use $(I-\mathcal{H}_{\zeta})D_t\bar{\zeta}=0$, we have 
\begin{equation}\label{dayu}
\begin{split}
&-i(I-\mathcal{H}_{\zeta})\frac{a_t}{a}\circ \kappa^{-1}A\bar{\zeta}_{\alpha}\\
=&(I-\mathcal{H}_{\zeta})(D_t^2+iA\partial_{\alpha})D_t\bar{\zeta}\\
=&[(D_t^2+iA\partial_{\alpha}),  \mathcal{H}_{\zeta}]D_t\bar{\zeta}\\
=& 2[D_t^2\zeta, \mathcal{H}_{\zeta}]\frac{\partial_{\alpha}D_t\bar{\zeta}}{\zeta_{\alpha}}+2[D_t\zeta, \mathcal{H}_{\zeta}]\frac{\partial_{\alpha}D_t^2\bar{\zeta}}{\zeta_{\alpha}}-\frac{1}{\pi i}\int \Big(\frac{D_t\zeta(\alpha)-D_t\zeta(\beta)}{\zeta(\alpha)-\zeta(\beta)}\Big)^2\partial_{\beta}D_t\bar{\zeta}d\beta.
\end{split}
\end{equation}
So we have 
\begin{equation}\label{ata11111}
\begin{split}
&(I-\mathcal{H}_{\zeta})(\frac{a_t}{a})\circ\kappa^{-1}A\bar{\zeta}_{\alpha}\\
=& 2i[D_t^2\zeta, \mathcal{H}_{\zeta}]\frac{\partial_{\alpha}D_t\bar{\zeta}}{\zeta_{\alpha}}+2i[D_t\zeta, \mathcal{H}_{\zeta}]\frac{\partial_{\alpha}D_t^2\bar{\zeta}}{\zeta_{\alpha}}-\frac{1}{\pi }\int \Big(\frac{D_t\zeta(\alpha)-D_t\zeta(\beta)}{\zeta(\alpha)-\zeta(\beta)}\Big)^2\partial_{\beta}D_t\bar{\zeta}d\beta.
\end{split}
\end{equation}

\noindent Use the same argument, we get nonlocal version of formula for $\frac{(a_0)_t}{a_0}\circ\kappa_0^{-1}$:

\begin{equation}\label{nonlocalat}
\begin{split}
(I-\mathcal{H}_{\omega})(\frac{(a_0)_t}{a_0}\circ\kappa_0^{-1}A_0\bar{\omega}_{\alpha})=& 2i[(D_t^0)^2\omega, \mathcal{H}_{\omega}]\frac{\partial_{\alpha}D_t^0\bar{\omega}}{\omega_{\alpha}}+2i[D_t^0\omega, \mathcal{H}_{\omega}]\frac{\partial_{\alpha}(D_t^0)^2\bar{\omega}}{\omega_{\alpha}}\\
&-\frac{1}{\pi }\int \Big(\frac{D_t^0\omega(\alpha)-D_t^0\omega(\beta)}{\omega(\alpha)-\omega(\beta)}\Big)^2\partial_{\beta}D_t^0\bar{\omega}d\beta
\end{split}
\end{equation}

\subsection{Formula for $b-b_0$} Write $b=b_0+b_1$. By (\ref{formulab}), (\ref{formulab0}), we have 
\begin{align*}
(I-\mathcal{H}_{\zeta})b-(I-\mathcal{H}_{\omega})b_0=& -[D_t\zeta,\mathcal{H}_{\zeta}\frac{1}{\zeta_{\alpha}}](\bar{\zeta}_{\alpha}-1)+[D_t^0\omega,\mathcal{H}_{\omega}\frac{1}{\omega_{\alpha}}](\bar{\omega}_{\alpha}-1)\\
\end{align*}
So we obtain
\begin{equation}\label{b1}
\begin{split}
(I-\mathcal{H}_{\zeta})(b-b_0)=&-[D_t\zeta-D_t^0\omega,\mathcal{H}_{\zeta}]\frac{\bar{\zeta}_{\alpha}-1}{\zeta_{\alpha}}-[D_t^0\omega, \mathcal{H}_{\zeta}\frac{1}{\zeta_{\alpha}}-\mathcal{H}_{\omega}\frac{1}{\omega_{\alpha}}](\bar{\zeta}_{\alpha}-1)\\
&-[D_t^0\omega, \mathcal{H}_{\omega}\frac{1}{\omega_{\alpha}}](\bar{\xi}_1){\alpha}-(\mathcal{H}_{\zeta}-\mathcal{H}_{\omega})b_0.
\end{split}
\end{equation}

\subsection{Formula for $\frac{a_t}{a}\circ \kappa^{-1}-\frac{(a_0)_t}{a_0}\circ \kappa_0^{-1}$} We have
\begin{equation}\label{zetaomegadiff}
\begin{split}
&(I-\mathcal{H}_{\zeta})\Big\{A\bar{\zeta}_{\alpha}\Big[\frac{a_t}{a}\circ\kappa^{-1}-\frac{(a_0)_t}{a_0}\circ\kappa_0^{-1}\Big]\Big\}\\
=& (I-\mathcal{H}_{\zeta})\Big(A\bar{\zeta}_{\alpha}\Big(\frac{a_t}{a}\Big)\circ \kappa^{-1}\Big)-(I-\mathcal{H}_{\omega})\Big(A_0\bar{\omega}_{\alpha}\Big(\frac{(a_0)_t}{a_0}\Big)\circ \kappa_0^{-1}\Big)\\
&+(\mathcal{H}_{\zeta}-\mathcal{H}_{\omega})\Big(A_0\bar{\omega}_{\alpha}\Big(\frac{(a_0)_t}{a_0}\Big)\circ \kappa_0^{-1}\Big).
\end{split}
\end{equation}

\subsection{Formula for $D_tb_1$}The idea of deriving formula for $D_tb_1$ is the same as that for $b_1$: find formula for $D_tb$ and $D_t^0b_0$ and then consider their difference.  The derivation of formula for $D_tb$ and $D_t^0b_0$ are similar to that in S. Wu's paper (See Proposition 2.7 of \cite{Wu2009}). We record the formula as follows.
\begin{equation}\label{Dtb}
\begin{split}
(I-\mathcal{H}_{\zeta})D_tb=& [D_t\zeta,\mathcal{H}_{\zeta}]\frac{\partial_{\alpha}(2b-D_t\bar{\zeta})}{\zeta_{\alpha}}-[D_t^2\zeta,\mathcal{H}_{\zeta}]\frac{\bar{\zeta}_{\alpha}-1}{\zeta_{\alpha}}\\
&+\frac{1}{\pi i}\int \Big(\frac{D_t\zeta(\alpha)-D_t\zeta(\beta)}{\zeta(\alpha)-\zeta(\beta)}\Big)^2 (\bar{\zeta}_{\beta}(\beta)-1)d\beta.
\end{split}
\end{equation}
For the periodic part, we have
\begin{equation}\label{Dtb0}
\begin{split}
(I-\mathcal{H}_{\omega})D_t^0b_0=& [D_t^0\omega,\mathcal{H}_{\omega}]\frac{\partial_{\alpha}(2b_0-D_t^0\bar{\omega})}{\zeta_{\alpha}}-[(D_t^0)^2\omega,\mathcal{H}_{\omega}]\frac{\bar{\omega}_{\alpha}-1}{\omega_{\alpha}}\\
&+\frac{1}{\pi i}\int \Big(\frac{D_t^0\omega(\alpha)-D_t^0\omega(\beta)}{\omega(\alpha)-\omega(\beta)}\Big)^2 (\bar{\omega}_{\beta}(\beta)-1)d\beta.
\end{split}
\end{equation}
Subtract (\ref{Dtb0}) from  (\ref{Dtb}) , we have
\begin{equation}\label{Dtb1}
\begin{split}
&(I-\mathcal{H}_{\zeta})D_tb_1\\
=& (I-\mathcal{H}_{\zeta})D_tb-(I-\mathcal{H}_{\omega})D_t^0b_0+(\mathcal{H}_{\zeta}-\mathcal{H}_{\omega})D_t^0b_0-(I-\mathcal{H}_{\zeta})b_1\partial_{\alpha}b_0\\
=& [D_t\zeta,\mathcal{H}_{\zeta}]\frac{\partial_{\alpha}(2b-D_t\bar{\zeta})}{\zeta_{\alpha}}-[D_t^0\omega,\mathcal{H}_{\omega}]\frac{\partial_{\alpha}(2b_0-D_t^0\bar{\omega})}{\zeta_{\alpha}}\\
&-[D_t^2\zeta,\mathcal{H}_{\zeta}]\frac{\bar{\zeta}_{\alpha}-1}{\zeta_{\alpha}}+[(D_t^0)^2\omega,\mathcal{H}_{\omega}]\frac{\bar{\omega}_{\alpha}-1}{\omega_{\alpha}}\\
&+\frac{1}{\pi i}\int \Big(\frac{D_t\zeta(\alpha)-D_t\zeta(\beta)}{\zeta(\alpha)-\zeta(\beta)}\Big)^2 (\bar{\zeta}_{\beta}(\beta)-1)d\beta-\frac{1}{\pi i}\int \Big(\frac{D_t^0\omega(\alpha)-D_t^0\omega(\beta)}{\omega(\alpha)-\omega(\beta)}\Big)^2 (\bar{\omega}_{\beta}(\beta)-1)d\beta\\
&+ (\mathcal{H}_{\zeta}-\mathcal{H}_{\omega})D_t^0b_0\\
&-(I-\mathcal{H}_{\zeta})b_1\partial_{\alpha}b_0.
\end{split}
\end{equation}

\subsection{Formula for $A-A_0$} By (\ref{formulaA}) and (\ref{nonlocalA0}), we have 
\begin{equation}\label{a1a1}
\begin{split}
(I-\mathcal{H}_{\zeta})(A-A_0)=& i[D_t^2\zeta,\mathcal{H}_{\zeta}]\frac{\bar{\zeta}_{\alpha}-1}{\zeta_{\alpha}}-i[(D_t^0)^2\omega,\mathcal{H}_{\omega}]\frac{\bar{\omega}_{\alpha}-1}{\omega_{\alpha}}\\
&+i[D_t\zeta,\mathcal{H}_{\zeta}]\frac{\partial_{\alpha}D_t\bar{\zeta}}{\zeta_{\alpha}}-i[D_t^0\omega,\mathcal{H}_{\omega}]\frac{\partial_{\alpha}D_t^0 \bar{\omega}}{\omega_{\alpha}}  \\
&+(\mathcal{H}_{\omega}-\mathcal{H}_{\zeta})(A_0-1)  .
\end{split}
\end{equation}

\vspace*{1ex}

\noindent Now we have formula for $b_1, D_tb_1, A_1$. So that we have a quasilinear system. It's not difficult to obtain local well-posedness of this quasilinear system. We omit the details and  focus on the long time existence. 


\subsection{A discussion on long time existence}
In order to prove long time well-posedness, one idea is to find some quantity $\zeta$ with $D_t\theta\approx D_t\zeta$, such that $$\mathcal{P}\theta=cubic.$$
In \cite{Wu2009}, S. Wu take $\theta=(I-\mathcal{H}_{\zeta})(\zeta-\alpha)$ and show that $\mathcal{P}\theta$ consists of cubic and higher order terms. For water waves that is neither periodic nor vanishing at spatial infinity, if we take $\theta=(I-\mathcal{H}_{\zeta})(\zeta-\alpha)$,  then $\mathcal{P}\theta$ is still cubic, at least in $BMO$ sense. As was explained in the introduction, since $\theta$ is not in any $L^2(\RR)$ based spaces, it's difficult to associate $\mathcal{P}\theta=cubic$ with an appropriate energy which still preserves this cubic structure.

Note however that, given any compatible initial data $(\zeta(\alpha,0), D_t\zeta(\alpha,0), D_t^2\zeta(\alpha,0))$ be such that 
$$(\partial_{\alpha}\zeta(\alpha,0)-1, D_t\zeta(\alpha,0), D_t^2\zeta(\alpha,0))\in X^s\times X^{s+1/2}\times X^s,$$
by Theorem \ref{longperiodic}, $\omega(\alpha,t), D_t^0\omega(\alpha,t), (D_t^0)^2\omega(\alpha,t)$ exists on time scale $O(\epsilon^{-2})$, so we need only to consider long time existence for $\xi_1:=\zeta-\omega$ and $D_t\xi_1$:  and it's advantageous to do so, because $\xi_1(\alpha,t)$ and $D_t\xi_1(\alpha,t)$ vanish as $\alpha\rightarrow \infty$, while $\zeta-\alpha$ and $D_t\zeta$ oscillate at $\infty$.  It turns out that $\mathcal{P}(I-\mathcal{H}_{\zeta})\xi_1$ consists of cubic and higher order nonlinearities. So we are able to prove long time existence in our situation.

\subsection{Governing equation for $\xi_1$}  
\subsubsection{$\mathcal{P}(I-\mathcal{H}_{\zeta})(\zeta-\alpha)$} In \cite{Wu2009}, the key ingredients that S. Wu derived $\mathcal{P}(I-\mathcal{H}_{\zeta})(\zeta-\alpha)$ are:  \begin{equation}\label{keyingredients}
(I-\mathcal{H}_{\zeta})D_t\bar{\zeta}=0,  \quad (I-\mathcal{H}_{\zeta})(\bar{\zeta}-\alpha)=0,  \quad  (I-\mathcal{H}_{\zeta})\frac{\bar{\zeta}_{\alpha}-1}{\zeta_{\alpha}}=0, \quad \mathcal{P}\zeta=-i.
\end{equation}
In our situation, (\ref{keyingredients}) is still true, despite that we have non-vanishing water waves at $\infty$.  Use the same derivation as in  \cite{Wu2009}, we have
\begin{equation}\label{cubic1}
\begin{split}
&\mathcal{P}(I-\mathcal{H}_{\zeta})(\zeta-\alpha)\\=&-2[D_t\zeta, \mathcal{H}_{\zeta}\frac{1}{\zeta_{\alpha}}+\bar{\mathcal{H}}_{\zeta}\frac{1}{\bar{\zeta}_{\alpha}}]\partial_{\alpha}D_t\zeta+\frac{1}{\pi i}\int \Big(\frac{D_t\zeta(\alpha)-D_t\zeta(\beta)}{\zeta(\alpha)-\zeta(\beta)}\Big)^2 \partial_{\beta}(\zeta-\bar{\zeta})d\beta\\
:=& G. 
\end{split}
\end{equation}
Similarly, 
\begin{equation}\label{cubic0}
\begin{split}
&\mathcal{P}_0(I-\mathcal{H}_{\omega})(\omega-\alpha)\\
=&-2[D_t^0\omega, \mathcal{H}_{\omega}\frac{1}{\omega_{\alpha}}+\bar{\mathcal{H}}_{\omega}\frac{1}{\bar{\omega}_{\alpha}}]\partial_{\beta}D_t^0\omega+\frac{1}{\pi i}\int_{-\infty}^{\infty} \Big(\frac{D_t^0\omega(\alpha)-D_t^0\omega(\beta)}{\omega(\alpha)-\omega(\beta)}\Big)^2\partial_{\beta}(\omega-\bar{\omega})d\beta\\
:=&G_0.
\end{split}
\end{equation}
\subsubsection{An equivalent quantity of $\xi_1$}
Denote 
\begin{equation}
\lambda:=(I-\mathcal{H}_{\zeta})(\zeta-\alpha)-(I-\mathcal{H}_{\omega})(\omega-\alpha).
\end{equation}
The reason we consider this quantity is that at least formally, we have known $\mathcal{P}(I-\mathcal{H}_{\zeta})(\zeta-\bar{\zeta})$ and $\mathcal{P}_0(I-\mathcal{H}_{\omega})(\omega-\bar{\omega})$ consist of cubic and higher order terms. So the quantity $\mathcal{P}(I-\mathcal{H}_{\zeta})(\zeta-\bar{\zeta})-\mathcal{P}_0(I-\mathcal{H}_{\omega})(\omega-\bar{\omega})$ is at least cubic. Moreover, $\mathcal{P}-\mathcal{P}_0$ is quadratic.  So $\mathcal{P}\lambda$ is cubic. 

Note that $\lambda$ might not be holomorphic in $\Omega(t)$. To avoid loss of derivatives, we consider the quantity
\begin{equation}
\theta:=(I-\mathcal{H}_{\zeta})\lambda.
\end{equation}
First we derive water waves equation for $\mathcal{P}\lambda$, and then we derive $\mathcal{P}(I-\mathcal{H}_{\zeta})\lambda$. 
Direct calculation gives
\begin{equation}
\begin{split}
D_t^2-(D_t^0)^2=& D_t^2-D_tD_t^0+D_tD_t^0-(D_t^0)^2\\
=& D_t(b_1\partial_{\alpha})+b_1\partial_{\alpha}D_t^0\\
=& (D_t b_1)\partial_{\alpha}+b_1D_t\partial_{\alpha}+b_1\partial_{\alpha}D_t^0.
\end{split}
\end{equation}
Then we have 
\begin{equation}\label{lambda}
\begin{split}
\mathcal{P}\lambda
=& \mathcal{P}(I-\mathcal{H}_{\zeta})(\zeta-\alpha)-\mathcal{P}_0(I-\mathcal{H}_{\omega})(\omega-\alpha)+(\mathcal{P}-\mathcal{P}_0)(I-\mathcal{H}_{\omega})(\omega-\alpha)\\
=& -2[D_t\zeta,\mathcal{H}_{\zeta}\frac{1}{\zeta_{\alpha}}+\bar{\mathcal{H}}_{\zeta}\frac{1}{\bar{\zeta}_{\alpha}}]\partial_{\alpha}D_t\zeta+2[D_t^0\omega,\mathcal{H}_{\omega}\frac{1}{\omega_{\alpha}}+\bar{\mathcal{H}}_{\omega}\frac{1}{\bar{\omega}_{\alpha}}]\partial_{\alpha}D_t^0\omega\\
&+\frac{1}{\pi i}\int \Big(\frac{D_t\zeta(\alpha,t)-D_t\zeta(\beta,t)}{\zeta(\alpha,t)-\zeta(\beta,t)}\Big)^2(\zeta-\bar{\zeta})_{\beta}d\beta\\
&-\frac{1}{\pi i}\int \Big(\frac{D_t^0\omega(\alpha,t)-D_t^0\omega(\beta,t)}{\omega(\alpha,t)-\omega(\beta,t)}\Big)^2(\omega-\bar{\omega})_{\beta}d\beta\\
&+ D_tb_1 \partial_{\alpha}(I-\mathcal{H}_{\omega})(\omega-\alpha)+b_1 D_t\partial_{\alpha}(I-\mathcal{H}_{\omega})(\omega-\alpha)\\
&+b_1\partial_{\alpha}D_t^0(I-\mathcal{H}_{\omega})(\omega-\alpha)-iA_1\partial_{\alpha}(I-\mathcal{H}_{\omega})(\omega-\alpha).
\end{split}
\end{equation}
So $\mathcal{P}\lambda$ consists of cubic and higher order nonlinearities.  Note that
\begin{equation}
\begin{split}
\mathcal{P}\theta=&\mathcal{P}(I-\mathcal{H}_{\zeta})\lambda\\
=&(I-\mathcal{H}_{\zeta})\mathcal{P}\lambda-[\mathcal{P},\mathcal{H}_{\zeta}]\lambda.
\end{split}
\end{equation}
We have
\begin{equation}
\begin{split}
[\mathcal{P},\mathcal{H}_{\zeta}]\lambda
=& 2[D_t\zeta,\mathcal{H}_{\zeta}]\frac{\partial_{\alpha}D_t\lambda}{\zeta_{\alpha}}-\frac{1}{\pi i}\int \Big(\frac{D_t\zeta(\alpha)-D_t\zeta(\beta)}{\zeta(\alpha)-\zeta(\beta)}\Big)^2 \partial_{\beta}\lambda d\beta.
\end{split}
\end{equation}
Note that 
\begin{equation}
\begin{split}
\lambda=&(I-\mathcal{H}_{\zeta})(\zeta-\omega)+(\mathcal{H}_{\omega}-\mathcal{H}_{\zeta})(\omega-\alpha)\\
=&(I-\mathcal{H}_{\zeta})\xi_1+(\mathcal{H}_{\omega}-\mathcal{H}_{\zeta})(\omega-\alpha)\\
:=& \lambda_1+\lambda_2.
\end{split}
\end{equation}
$\lambda_2$ is quadratic, so $[D_t\zeta, \mathcal{H}_{\zeta}]\frac{\partial_{\alpha}D_t\lambda_2}{\zeta_{\alpha}}$ is cubic. $\lambda_1$ is holomorphic in $\Omega(t)^c$, so $[D_t\zeta, \bar{\mathcal{H}}_{\zeta}\frac{1}{\bar{\zeta}_{\alpha}}]\partial_{\alpha}D_t\lambda_1$ is cubic. So 
$$[D_t\zeta,\mathcal{H}_{\zeta}]\frac{\partial_{\alpha}D_t\lambda_1}{\zeta_{\alpha}}=-[D_t\zeta,\bar{\mathcal{H}}_{\zeta}]\frac{\partial_{\alpha}D_t\lambda_1}{\bar{\zeta}_{\alpha}}+[D_t\zeta, \mathcal{H}_{\zeta}\frac{1}{\zeta_{\alpha}}+\bar{\mathcal{H}}_{\zeta}\frac{1}{\bar{\zeta}_{\alpha}}]\partial_\alpha D_t\lambda_1$$ is cubic. 

\subsection{Governing equation for $D_t\theta$.} The nonlinearities $\mathcal{P}(I-\mathcal{H}_{\zeta})D_t\theta$ contains a term of the form $D_t^2 b_1$, which loses derivatives in energy estimates.  So we consider the quantity 
$$\sigma=(I-\mathcal{H}_{\zeta})[D_t(I-\mathcal{H}_{\zeta})(\zeta-\alpha)-D_t^0(I-\mathcal{H}_{\omega})(\omega-\alpha)].$$
Denote 
$$\chi:=D_t(I-\mathcal{H}_{\zeta})(\zeta-\alpha)-D_t^0(I-\mathcal{H}_{\omega})(\omega-\alpha).$$
\begin{remark}
If we replace $\sigma$ by $\chi$, then we'll lose one derivative in the energy estimates. The advantage of $(I-\mathcal{H}_{\zeta})$ acting on $\chi$ is that $(I+\mathcal{H}_{\zeta})\sigma=0$, so $(I+\mathcal{H}_{\zeta})\partial_{\alpha}\sigma=-[\partial_{\alpha}, \mathcal{H}_{\zeta}]\sigma$, which prevents losing one derivative. 
\end{remark}
We have 
\begin{align*}
&\mathcal{P}D_t(I-\mathcal{H}_{\zeta})(\zeta-\alpha)\\
=& D_t\mathcal{P}(I-\mathcal{H}_{\zeta})(\zeta-\alpha)+[\mathcal{P}, D_t](I-\mathcal{H}_{\zeta})(\zeta-\alpha)\\
=&D_tG+[\mathcal{P}, D_t](I-\mathcal{H}_{\zeta})(\zeta-\alpha).
\end{align*}
And we have
\begin{align*}
&\mathcal{P}_0D_t^0(I-\mathcal{H}_{\omega})(\omega-\alpha)\\
=& D_t^0\mathcal{P}_0(I-\mathcal{H}_{\omega})+[\mathcal{P}_0, D_t^0](I-\mathcal{H}_{\omega})(\omega-\alpha)\\
=& D_t^0 G_0+[\mathcal{P}_0, D_t^0](I-\mathcal{H}_{\omega})(\omega-\alpha).
\end{align*}
So we have 
\begin{equation}
\begin{split}
\mathcal{P}\chi
=& \mathcal{P}D_t(I-\mathcal{H}_{\zeta})(\zeta-\alpha)-\mathcal{P}_0D_t^0(I-\mathcal{H}_{\omega})(\omega-\alpha)\\
&+(\mathcal{P}-\mathcal{P}_0)D_t^0(I-\mathcal{H}_{\omega})(\omega-\alpha)\\
=& D_tG-D_t^0G_0+(\mathcal{P}-\mathcal{P}_0)D_t^0(I-\mathcal{H}_{\omega})(\omega-\alpha)\\
&+[\mathcal{P}, D_t](I-\mathcal{H}_{\zeta})(\zeta-\alpha)-[\mathcal{P}_0, D_t^0](I-\mathcal{H}_{\omega})(\omega-\alpha)
\end{split}
\end{equation}
We have
\begin{align*}
\mathcal{P}\sigma=(I-\mathcal{H})\mathcal{P}\chi-[\mathcal{P},\mathcal{H}_{\zeta}]\chi
\end{align*}
\begin{equation}
\begin{split}
[\mathcal{P},\mathcal{H}]\chi=& 2[D_t\zeta,\mathcal{H}_{\zeta}]\frac{\partial_{\alpha}D_t\chi}{\zeta_{\alpha}}-\frac{1}{\pi i}\int \Big(\frac{D_t\zeta(\alpha)-D_t\zeta(\beta)}{\zeta(\alpha)-\zeta(\beta)}\Big)^2 \partial_{\beta}\chi d\beta.
\end{split}
\end{equation}
Use the same argument as we did for $2[D_t\zeta,\mathcal{H}_{\zeta}]\frac{\partial_{\alpha}D_t\lambda}{\zeta_{\alpha}}$, we can show that $2[D_t\zeta,\mathcal{H}_{\zeta}]\frac{\partial_{\alpha}D_t\chi}{\zeta_{\alpha}}$ is indeed cubic. 
\vspace*{3ex}

\subsection{Long time existence}
With previous preparations, use standard energy method (similar to those in \cite{Wu2009}, \cite{Totz2012}), we can complete the proof of Theorem \ref{theorem1}.  A minor modification of the argument in \S \ref{govern} -\S \ref{energyestimatesection} also gives a proof of Theorem \ref{theorem1}. We omit the details of the proof here. 

\begin{remark}
In our set up, we need the periodic solution $\omega$ to have $\frac{3}{2}+$ more derivatives than the decaying part $\xi_1$. This requirement is of course not optimal.  However, it's enough for us to justify the Peregrine soliton from the full water waves. 
\end{remark}

\begin{remark}
Theorem \ref{theorem1} can be interpreted as: Periodic water wave system is stable under Sobolev perturbation (note that this perturbation is indeed not small relative to the periodic part).
\end{remark}

\section{Multiscale analysis and the derivation of NLS from full water waves equation}

Our goal of this section is to formally derive the NLS from full water waves, which is similar to that in (\cite{Totz2012}), except that the water waves we are considering do not vanish at infinity. The method we use to derive the NLS is the multiscale analysis. Let \begin{equation}
    \alpha_0:=\alpha,\quad \alpha_1:=\epsilon \alpha,\quad t_0:=t,\quad t_1:=\epsilon t,\quad t_2:=\epsilon^2t.
\end{equation} 
Assume that $\bar{\zeta}-\alpha, D_t\bar{\zeta}\in \mathcal{H}ol(\Omega(t))$.
 Assume $\zeta$ can be expanded as a power series of $\epsilon$, i.e.,
\begin{equation}
\zeta=\alpha+\sum_{n\geq 1} \epsilon^n \zeta^{(n)}.
\end{equation}
We assume $\zeta^{(1)}$ is wave packet like, i.e., $\zeta^{(1)}=B(\alpha_1,t_0, t_1, t_2)e^{i \phi}$, where $\phi=k\alpha+\gamma t$ for some constants $k, \gamma>0$. We don't assume $B\in H^s(\RR)$. Instead, we assume $B=B_0+B_1$, with $B_0=B_0(t)$ independent of $\alpha$, and $B_1\in H^s(\RR)$. 

Because $\bar{\zeta}-\alpha$ is holomorphic, the leading order of $\bar{\zeta}-\alpha$ must be close to a holomorphic function. If $B_0\equiv 0$, we the following result:
\begin{lemma}[Propositioin 3.1 in \cite{Totz2012}]\label{prop3-1}
Let $f=
g(\epsilon\alpha)e^{-ik\alpha}$, with $g\in H^{s+m}(\RR)$, $k\neq 0$ and $s,m\geq 0$ be given, assume $\epsilon\leq 1$ and $g\in H^{s+m}(\RR)$. Then 
$$\norm{(I-sgn(k)\mathbb{H})f}_{H^{s}(\RR)}\leq C\frac{\epsilon^{m-\frac{1}{2}}}{k^m}\|g\|_{H^{s+m}(\RR)},$$
for some constant $C=C(s)$.
\end{lemma}
\noindent If $f$ is oscillating at $\infty$, we have the following.
\begin{lemma}\label{regularitytodecay}
Let $c$ be a constant and assume $f=ce^{-ik\alpha}+
g(\epsilon\alpha)e^{-ik\alpha}$, with $g\in H^{s+m}(\RR)$, $k\neq 0$ and $s,m\geq 0$ be given. Assume $\epsilon\leq 1$. Then 
$$\norm{(I-sgn(k)\mathbb{H})f}_{H^{s}(\RR)}\leq C\frac{\epsilon^{m-\frac{1}{2}}}{k^m}\|g\|_{H^{s+m}(\RR)},$$
for some constant $C=C(s)$.
\end{lemma}

\begin{proof}
For $k\neq 0$, we have
\begin{equation}
    (I-sgn(k)\mathbb{H})e^{-ik\alpha}=0.
\end{equation}
Therefore, by lemma \ref{prop3-1} in \cite{Totz2012}, we have 
\begin{equation}
    \|(I-\mathbb{H})f\|_{H^s(\RR)}=\norm{(I-\mathbb{H})g(\epsilon \alpha)e^{-ik\alpha}}_{H^s(\RR)}\leq C\frac{\epsilon^{m-\frac{1}{2}}}{k^m}\|g\|_{H^{s+m}(\RR)}.
\end{equation}
\end{proof}
Now we are ready to carry out asymptotic expansion and derive the focusing NLS.
As in \cite{Totz2012}, we use the following equation to perform multiscale analysis.
\begin{equation}
\begin{split}
(D_t^2-iA\partial_{\alpha}) & (I-\mathcal{H}_{\zeta})(\zeta-\bar{\zeta})=-2[D_t\zeta, \mathcal{H}_{\zeta}\frac{1}{\zeta_{\alpha}}+\bar{\mathcal{H}_{\zeta}}\frac{1}{\bar{\zeta}_{\alpha}}]\partial_{\alpha}D_t\bar{\zeta}\\
&+\frac{1}{\pi i}\int \Big(\frac{D_t\zeta(\alpha)-D_t\zeta(\beta)}{\zeta(\alpha)-\zeta(\beta)}\Big)^2\partial_{\beta}(\zeta-\bar{\zeta})d\beta:= G.
\end{split}
\end{equation}
So we need to expand every quantity/operator as asymptotic series in $\epsilon$. These have been done in (\cite{Totz2012}).  We expand $b, A, G$ as 
\begin{equation}
    b=\sum_{n\geq 0}\epsilon^n b^{(n)},\quad  A=\sum_{n\geq 0}\epsilon^n A^{(n)}, \quad G=\sum_{n\geq 0}\epsilon^n G_n.
\end{equation}
Since $b$ and $A-1$ are quadratic and $G$ is cubic, we have 
\begin{equation}
b^{(0)}=b^{(1)}=A^{(1)}=G_1=G_2=0, \quad A^{(0)}=1.
\end{equation}
Expand $\mathcal{H}_{\zeta}=H_0+\sum_{n\geq 1}\epsilon^n H_n$. Then 
\begin{equation}
\begin{split}
H_0f(\alpha)=&\mathbb{H}f(\alpha),\\
H_1f(\alpha)=&[\zeta^{(1)}, H_0]\partial_{\alpha_0}f,\\
H_2f(\alpha)=&[\zeta^{(1)}, H_0]f_{\alpha_1}+[\zeta^{(2)}, H_0]f_{\alpha_0}-[\zeta^{(1)}, H_0]\zeta_{\alpha_0}^{(1)}f_{\alpha_0}+\frac{1}{2}[\zeta^{(1)}, [\zeta^{(1)}, H_0]]\partial_{\alpha_0}^2f.
\end{split}
\end{equation}
See \S 3.1 in \cite{Totz2012} for the derivation of $H_0, H_1$ and $H_2$. 

\subsection{$O(1)$ hierarchy} This simply gives 
$$A^{(0)}=1.$$

\subsection{$\epsilon$ hierarchy}   We have 
\begin{align*}
(\partial_{t_0}^2-i\partial_{\alpha_0})(I-H_0)\zeta^{(1)}=0.
\end{align*}
Since $\zeta^{(1)}=B(\alpha_1,t_0, t_1, t_2)e^{i\phi}$, by Lemma \ref{regularitytodecay}, we have 
\begin{equation}
    (I-H_0)\zeta^{(1)}=2\zeta^{(1)}+O(\epsilon^4).
\end{equation}
So we have 
$$(\partial_{t_0}^2-i\partial_{\alpha_0})\zeta^{(1)}=O(\epsilon^4).$$
Then we get  $\zeta^{(1)}=B(\alpha_1, t_1, t_2)e^{i(k\alpha+\gamma t)}$, with $\gamma^2=k$. We simply choose $\gamma=\sqrt{k}$, as what we expected.

\subsection{$\epsilon^2$ level} We need
\begin{align*}
(\partial_{t_0}^2-i\partial_{\alpha_0})(I-H_0)\zeta^{(2)}=-(2\partial_{t_0}\partial_{t_1}-i\partial_{\alpha_1})(I-H_0)\zeta^{(1)}+(\partial_{t_0}^2-i\partial_{\alpha_0})H_1 \zeta^{(1)}.
\end{align*}
Note that 
\begin{align*}
(\partial_{t_0}^2-i\partial_{\alpha_0})H_1 \zeta^{(1)}=&  (\partial_{t_0}^2-i\partial_{\alpha_0})[\zeta^{(1)},H_0]\partial_{\alpha_0}Be^{i\phi}\\
=&  ik(\partial_{t_0}^2-i\partial_{\alpha_0})[\zeta^{(1)},I+H_0]Be^{i\phi}\\
=& O(\epsilon^4).
\end{align*}
To avoid secular terms, we choose $\zeta^{(1)}$ such that 
$$-(2\partial_{t_0}\partial_{t_1}-i\partial_{\alpha_1})(I-H_0)\zeta^{(1)}=0.$$
This is equivalent to 
$$ B_{t_1}-\frac{1}{2\gamma}B_{\alpha_1}=0.$$
So we choose $B=B(X,T)$, with $X= \alpha_1+\frac{1}{2\gamma}t_1=\epsilon(\alpha+\frac{1}{2\gamma}t)$, $T=t_2=\epsilon^2 t$. Note that $\frac{1}{2\gamma}=\frac{\partial \gamma}{\partial k}$, so $B$ travels at the group velocity.

To choose $\zeta^{(2)}$, we use $(I-\mathcal{H})(\bar{\zeta}-\alpha)=0$.  
\begin{align*}
(I-H_0)\bar{\zeta}^{(2)}=&H_1 \bar{\zeta}^{(1)}=[\zeta^{(1)}, H_0]\partial_{\alpha_0} \bar{\zeta}^{(1)}\\
=&-ik[\zeta^{(1)},H_0]\bar{B}e^{-i\phi}\\
=& ik[\zeta^{(1)},I-H_0]\bar{B}e^{-i\phi}\\
=&-ik(I-H_0)|B|^2.
\end{align*}
We choose 
$$\zeta^{(2)}=\frac{ik}{2}(I+H_0)|B|^2+\frac{ik}{2}|B_0|^2=\frac{ik}{2}(I+H_0)(|B|^2-|B_0|^2)+ik|B_0|^2.$$
Note that 
$$(I-H_0)\bar{\zeta}^{(2)}=-\frac{ik}{2}(I-H_0)(I-H_0)(|B|^2-|B_0|^2)-ik(I-H_0)|B_0|^2=-ik(I-H_0)|B|^2.$$

\subsection{$\epsilon^3$ level}
First, we need to expand $b=\sum_{n\geq 0}\epsilon^n b^{(n)}$. Since $(I-\mathcal{H})b=-[D_t\zeta,\mathcal{H}]\frac{\bar{\zeta}_{\alpha}-1}{\zeta_{\alpha}}$ is quadratic, we have  $b^{(0)}=b^{(1)}=0$. 
For $b_2$, we have 
\begin{equation}
\begin{split}
(I-H_0)b^{(2)}=&-[\partial_{t_0}\zeta^{(1)},H_0]\partial_{\alpha_0}\bar{\zeta}^{(1)}\\
=&-\gamma k[\zeta^{(1)},H_0]\bar{\zeta}^{(1)}=\gamma k[\zeta^{(1)}, I-H_0]\bar{\zeta}^{(1)}=-\gamma k(I-H_0)|B|^2.
\end{split}
\end{equation}

\noindent Since $b^{(2)}$ is real, we have 
$$b^{(2)}=-\gamma k|B|^2.$$
We need also to expand $A=\sum_{n\geq 0}\epsilon^n A^{(n)}$. Clearly, $A^{(0)}=1$, and $A^{(1)}=0$. We have 
\begin{align*}
(I-H_0)A^{(2)}=i[\partial_{t_0}^2\zeta^{(1)},H_0]\partial_{\alpha_0}\bar{\zeta}^{(1)}+i[\partial_{t_0}\zeta^{(1)},H_0]\partial_{\alpha_0}\partial_{t_0}\bar{\zeta}^{(1)}=0.
\end{align*}
Since $A^{(2)}$ is real, we have $A^{(2)}=0$.  

Use exactly the same calculation as in (\cite{Totz2012}),  we obtain
$$G_3=2k^3B|B|^2 e^{i\phi}.$$
Then for $O(\epsilon^3)$ terms, we have 
\begin{align*}
(\partial_{t_0}^2-i\partial_{\alpha_0})(I-H_0)\zeta^{(3)}=&-(\partial_{t_0}^2-i\partial_{\alpha_0})(-H^{(1)})\zeta^{(2)}-(\partial_{t_0}^2-i\partial_{\alpha_0})(-H^{(2)})\zeta^{(1)}\\
&-(2\partial_{t_0}\partial_{t_1}-i\partial_{\alpha_1})(I-H_0)\zeta^{(2)}-(2\partial_{t_0\partial_{t_1}}-i\partial_{\alpha_1})(-H^{(1)})\zeta^{(1)}\\
&-(2\partial_{t_0t_2}+\partial_{t_1}^2+2b_2\partial_{t_0}\partial_{\alpha_0})(I-H_0)\zeta^{(1)}+G_3\\
=&-(2\partial_{t_0t_2}+\partial_{t_1}^2+2b_2\partial_{t_0}\partial_{\alpha_0})(I-H_0)\zeta^{(1)}+2k^3B|B|^2e^{i\phi}\\
=& -2\gamma(2iB_T-\gamma^{''}B_{XX}+k^2\gamma B|B|^2 )e^{i\phi},
\end{align*}
where $\gamma^{''}=\frac{d^2\gamma}{dk^2}=-\frac{1}{4k^{3/2}}$.  To avoid secular growth, we choose $B$ such that 
$$2iB_T-\gamma^{''}B_{XX}+k^2\gamma B|B|^2=0.$$
So $B$ solves the focusing cubic NLS. So we have 
$$(\partial_{t_0}^2-i\partial_{\alpha_0})(I-H_0)\zeta^{(3)}=0.$$
From $(I-H_0)(\bar{\zeta}-\alpha)=0$, we have 
\begin{align*}
(I-H_0)\bar{\zeta}^{(3)}=& H^{(1)}\bar{\zeta}^{(2)}+H^{(2)}\bar{\zeta}^{(2)}\\
=&[\zeta^{(1)},H_0]\partial_{\alpha_0}\bar{\zeta}^{(2)}+[\zeta^{(2)},H_0]\partial_{\alpha_0}\bar{\zeta}^{(1)}+[\zeta^{(1)},H_0]\partial_{\alpha_1}\zeta^{(1)}\\
&-[\zeta^{(1)},H_0]\overline{\partial_{\alpha_0}\zeta^{(1)}}\partial_{\alpha_0}\zeta^{(1)}+\frac{1}{2}[\zeta^{(1)}, [\zeta^{(1)}, H_0]]\partial_{\alpha_0}^2\bar{\zeta}^{(1)}\\
=&(I-H_0)B\bar{B}_X-k^2Be^{i\phi}(I+H_0)|B|^2+k^2Be^{i\phi}H_0|B|^2\\
=&-k^2B|B|^2e^{i\phi}+(I-H_0)B\bar{B}_X.
\end{align*}
We choose 
$$\zeta^{(3)}=-\frac{1}{2}k^2\bar{B}|B|^2e^{-i\phi}+\frac{1}{2}(I+H_0)(\bar{B}B_X).$$
So we have an approximated solution 
\begin{equation}
\begin{split}
\tilde{\zeta}=&\alpha+\epsilon B(X,T)e^{i\phi}+\epsilon^2\Big\{\frac{ik}{2}(I+H_0)|B|^2+\frac{ik}{2}|B_0|^2\Big\}\\
&+\epsilon^3\Big\{-\frac{1}{2}k^2\bar{B}|B|^2e^{-i\phi}+\frac{1}{2}(I+H_0)(\bar{B}B_X)\Big\}.
\end{split}
\end{equation}

To find $b_3$, we have 
\begin{align*}
&(I-H_0)b_3\\=&-[\partial_{t_0}\zeta^{(2)},H_0]\partial_{\alpha_0}\bar{\zeta}^{(1)}-[\partial_{t_1}\zeta^{(1)},H_0]\partial_{\alpha_0}\bar{\zeta}^{(1)}-[\partial_{t_0}\zeta^{(1)},H_1]\partial_{\alpha_0}\bar{\zeta}^{(1)}\\
&-[\partial_{t_0}\zeta^{(1)},H_0]\partial_{\alpha_1}\bar{\zeta}^{(1)}-[\partial_{t_0}\zeta^{(1)},H_0]\partial_{\alpha_0}\bar{\zeta}^{(2)}-[\partial_{t_0}\zeta^{(2)},H_0]\partial_{\alpha_0}\bar{\zeta}^{(1)}(-\partial_{\alpha_0}\zeta^{(1)})\\
=& i\gamma(I+H_0)(B\bar{B}_X-\frac{1}{2}\bar{B}B_X)-2ik^2\bar{B}|B|^2e^{-i\phi}.
\end{align*}
So we have 
\begin{equation}
b_3=\Re\{ i\gamma(I+H_0)(B\bar{B}_X-\frac{1}{2}\bar{B}B_X)-2ik^2\bar{B}|B|^2e^{-i\phi}\}.
\end{equation}

\vspace*{2ex}

To estimate the error, we write $\tilde{\zeta}$ as
\begin{equation}\label{approxzetaomega1}
    \tilde{\zeta}:=\alpha+\tilde{\xi}_0+\tilde{\xi}_1,   \quad \quad \tilde{\omega}:=\alpha+\tilde{\xi}_0.
\end{equation}
where
\begin{equation}\label{approxzetaomega2}
\begin{split}
\tilde{\xi}_1=&\epsilon B_1e^{i\phi}+\epsilon^2\frac{ik}{2}(I+H_0)(|B|^2-|B_0|^2)\\
&+\epsilon^3\Big\{-\frac{1}{2}k^2(\bar{B}|B|^2-\bar{B}_0|B_0|^2)e^{-i\phi}+\frac{1}{2}(I+H_0)(\bar{B}B_X)\Big\}.
\end{split}
\end{equation}
\begin{equation}\label{approxzetaomega3}
\begin{split}
\tilde{\xi}_0=&\tilde{\zeta}-\tilde{\xi}_1-\alpha=\epsilon B_0e^{i\phi}-\frac{ik}{2}|B_0|^2\epsilon^2-\frac{1}{2}\epsilon^3 k^2\bar{B}_0|B_0|^2e^{-i\phi}.
\end{split}
\end{equation}
So $\tilde{\omega}-\alpha$ is periodic.  Also, we decompose $\tilde{b}$ as $\tilde{b}_1+\tilde{b}_0$, where 
\begin{equation}
\begin{split}
\tilde{b}_1=&-\epsilon^2 \omega k(|B|^2-|B_0|^2)+\epsilon^3\Big\{\Re\{i\omega (I+H_0)(B\bar{B}_X-\frac{1}{2}\bar{B}B_X)\\
&-2ik^2(\bar{B}|B|^2-\bar{B}_0|B_0|^2)e^{-i\phi}\Big\}.
\end{split}
\end{equation}
\begin{equation}
\tilde{b}_0=-\epsilon^2 |B_0|^2+\epsilon^3 \Re\{-2ik^2\bar{B}_0|B_0|^2e^{-i\phi}\}.
\end{equation}
So $\tilde{b}_1\in H^s$, and $\tilde{b}_0$ is periodic. We choose $\tilde{A}$ as $$\tilde{A}:=A^{(0)}+\epsilon A^{(1)}+\epsilon^2 A^{(2)}=1.$$

\vspace*{2ex}

\noindent \textbf{Notation.} Denote
\begin{equation}
\begin{cases}
\xi=\zeta-\alpha,\\
\xi_0=\omega-\alpha, \quad \xi_1=\xi-\xi_0, \\
r_0=\omega-\tilde{\omega},\\
 r_1=\xi_1-\tilde{\xi}_1,
\tilde{D}_t:=\partial_t+\tilde{b}\partial_{\alpha},\\
 \tilde{D}_t^0=\partial_t+\tilde{b}_0\partial_{\alpha},\\
\tilde{\mathcal{P}}:=\tilde{D}_t^2-i\tilde{A}\partial_{\alpha},\\
 \tilde{\mathcal{P}}_0:=(\tilde{D}_t^0)^2-i\tilde{A}_0\partial_{\alpha}.
\end{cases}
\end{equation}
And recall that the leading order $\zeta^{(1)}$ is
$$\zeta^{(1)}=Be^{i\phi}:= \zeta_0^{(1)}+\zeta_1^{(1)},$$
where $B=B_0+B_1$ solves NLS, $B_0$ is a constant, and $B_1\in H^{s+7}(\RR)$.

\begin{equation}\label{zeta1}
\zeta=Be^{i\phi}, \quad \zeta_0^{(1)}=B_0e^{i\phi},\quad \zeta_1^{(1)}=B_1e^{i\phi},
\end{equation}
Because $B$ scales like $\epsilon^2$ in $t$. In order to observe the modulation of the amplitute, the solution $\zeta$ must exist on time interval whose length scales like $\epsilon^{-2}$, i.e., we must have long time existence for the water waves system. 
\subsection{Well posedness of NLS}
\begin{thm}\label{NLSlocal}
Let $s\geq 1$. There exists $e_0=e_0(\norm{B_1(0)}_{H^s})>0$ such that the cauchy problem
\begin{equation}\label{NLS}
\begin{cases}
iB_t+B_{xx}=-2|B|^2B,\\
B(0)=B_0(0)+B_1(0)=1+f, \quad B_1(0)\in H^s(\RR).
\end{cases}
\end{equation}
is locally well-posed on $[0, e_0]$, and satisfies
\begin{equation}
\norm{B}_{L_t^{\infty}H_x^s(\RR)}\leq C(\norm{B_1(0)}_{H^s(\RR)}).
\end{equation}
\end{thm}
\noindent For a proof, see for example \cite{gallo2004schrodinger}.
\begin{remark}
The global well-posedness of (\ref{NLS}) is still open, due to the lack of coercive conservation law. However, because our goal is to control the water wave in $O(\epsilon^{-2})$ scale, and the NLS approximation to full water waves scales like $\epsilon^2$ in time, a local well-posedness for NLS is enough for our purpose. 
\end{remark}

\noindent In the following sections, we obtain energy estimates for the remainder terms $r_0, r_1$, respectively. With good energy estimates on the remainder terms, we are able to prove existence of solutions to full water wave equations whose leading term modulated according to the NLS. 

\section{Energy estimate I:  $r_0$}
\noindent Because formally, $\zeta-\tilde{\zeta}=O(\epsilon^4)$, $\lim_{\alpha\rightarrow \pm \infty}(\zeta-\omega)=0$, and $\omega-\alpha$ periodic, we have $(\tilde{\omega}, \tilde{D}_t^0\tilde{\omega}, \tilde{b}_0, \tilde{A}_0)$ approximate (\ref{special}) with error $O(\epsilon^4)$, at least formally. 

\vspace*{1ex}

 In this section, we obtain a priori energy estimates for the remainder $r_0$. The idea is the same as that in \cite{Totz2012}: use the facts that $\mathcal{P}_0(I-\mathcal{H}_p)\xi_0=cubic$ and $\tilde{\omega}, \tilde{D}_t^0\tilde{\omega}, \tilde{b}_0, \tilde{A}_0$ approximates $\omega, D_t^0\omega, b_0, A_0$ up to $\epsilon^4$, respectively, we derive water wave equations for a quantity which is equivalent to $r_0$. With these equations, we can then obtain energy estimates for $r_0$ on time scale $\epsilon^{-2}$. 
\begin{remark}
As before, we use the periodic Hilbert transform. The nonlocal Hilbert transform $\mathcal{H}_{\omega}$ is used when we estimate the error term $r_1$. 
\end{remark}

\vspace*{1ex}
\noindent First, let's derive water wave equation for $r_0$.

\subsection{Governing equation for $r_0$}
We have 
\begin{align*}
&((D_t^0)^2-iA_0\partial_{\alpha})(I-\mathcal{H}_{p})r_0\\=&  ((D_t^0)^2-iA_0\partial_{\alpha})(I-\mathcal{H}_{p})\tilde{\omega}-2[D_t^0\omega,\mathcal{H}_{p}\frac{1}{\omega_{\alpha}}+\bar{\mathcal{H}}_{p}\frac{1}{\bar{\omega}_{\alpha}}]\partial_{\alpha}D_t^0\omega\\
&+\frac{1}{ 4\pi i}\int_{\TT} \Big(\frac{D_t^0\omega(\alpha,t)-D_t^0\omega(\beta,t)}{\sin(\frac{\pi}{2}(\omega(\alpha,t)-\omega(\beta,t)))}\Big)^2(\omega-\bar{\omega})(\beta)d\beta\\
:=& \mathcal{G}.
\end{align*}
We split $\mathcal{G}$ as $\mathcal{G}=\mathcal{G}_1+\mathcal{G}_2+\mathcal{G}_3+\mathcal{G}_4$, where 
\begin{equation}
\mathcal{G}_1:=((D_t^0)^2-iA_0\partial_{\alpha})(I-\mathcal{H}_{p})\tilde{\omega}-((\tilde{D}_t^0)^2-i\tilde{A}_0\partial_{\alpha})(I-\tilde{\mathcal{H}}_{p})\tilde{\omega}.
\end{equation}

\begin{equation}
\mathcal{G}_2:=-2[D_t^0\omega,\mathcal{H}_{p}\frac{1}{\omega_{\alpha}}+\bar{\mathcal{H}}_{p}\frac{1}{\bar{\omega}_{\alpha}}]\partial_{\alpha}D_t^0\omega+2[\tilde{D}_t^0\tilde{\omega},\tilde{\mathcal{H}}_{p}\frac{1}{\tilde{\omega}_{\alpha}}+\bar{\tilde{\mathcal{H}}}_{p}\frac{1}{\bar{\tilde{\omega}}_{\alpha}}]\partial_{\alpha}\tilde{D}_t^0\tilde{\omega},
\end{equation}
and
\begin{equation}
\begin{split}
\mathcal{G}_3:=&\frac{1}{4\pi i}\int_{\TT} \Big(\frac{D_t^0\omega(\alpha,t)-D_t^0\omega(\beta,t)}{\sin(\frac{\pi}{2}(\omega(\alpha,t)-\omega(\beta,t)))}\Big)^2(\omega-\bar{\omega})(\beta)d\beta\\
&-\frac{1}{4\pi i}\int_{\TT}\Big(\frac{\tilde{D}_t^0\tilde{\omega}(\alpha,t)-\tilde{D}_t^0\tilde{\omega}(\beta,t)}{\sin(\frac{\pi}{2}(\tilde{\omega}(\alpha,t)-\tilde{\omega}(\beta,t)))}\Big)^2 (\tilde{\omega}-\bar{\tilde{\omega}})(\beta)d\beta,
\end{split}
\end{equation}
and
\begin{equation}
\begin{split}
\mathcal{G}_4:=&(\tilde{D}_t^0)^2-i\tilde{A}_0\partial_{\alpha})(I-\mathcal{H}_{p})\tilde{\omega}-2[\tilde{D}_t^0\tilde{\omega},\tilde{\mathcal{H}}_{p}\frac{1}{\tilde{\omega}_{\alpha}}+\bar{\tilde{\mathcal{H}}}_{p}\frac{1}{\bar{\tilde{\omega}}_{\alpha}}]\partial_{\alpha}\tilde{D}_t^0\tilde{\omega}\\
&+\frac{1}{4\pi i}\int_{\TT}\Big(\frac{\tilde{D}_t^0\tilde{\omega}(\alpha,t)-\tilde{D}_t^0\tilde{\omega}(\beta,t)}{\sin(\frac{\pi}{2}(\tilde{\omega}(\alpha,t)-\tilde{\omega}(\beta,t)))}\Big)^2 (\tilde{\omega}-\bar{\tilde{\omega}})(\beta)d\beta.
\end{split}
\end{equation}

\subsection{Governing equation for $D_t^0(I-\mathcal{H}_{p})r_0$}
We need to derive an equation to control $D_t^0 r_0$ as well. If we use the quantity $D_t^0(I-\mathcal{H}_{p})r_0$, then there will be loss of derivatives in energy estimates. So we consider instead the quantity 
$$\sigma_0:=(I-\mathcal{H}_{p})\Big\{D_t^0(I-\mathcal{H}_{p})(\omega-\alpha)-\tilde{D}_t^0(I-\tilde{\mathcal{H}}_{p})(\tilde{\omega}-\alpha)\Big\}.$$

We have 
\begin{equation}\label{okok}
\begin{split}
&((D_t^0)^2-iA_0\partial_{\alpha})(I-\mathcal{H}_{p})D_t^0(I-\mathcal{H}_{p})(\omega-\alpha)\\
=& -[(D_t^0)^2-iA_0\partial_{\alpha},\mathcal{H}_{p}]D_t^0(I-\mathcal{H}_{p})(\omega-\alpha)\\
&+(I-\mathcal{H}_{p})((D_t^0)^2-iA_0\partial_{\alpha})D_t^0(I-\mathcal{H}_{p})(\omega-\alpha)\\
=&-2[D_t^0\zeta,\mathcal{H}_{p}]\frac{\partial_{\alpha}(D_t^0)^2(I-\mathcal{H}_{p})(\omega-\alpha)}{\omega_{\alpha}}\\
&+\frac{1}{4\pi i}\int_{\TT} \Big(\frac{D_t^0\omega(\alpha)-D_t^0\omega(\beta)}{\sin(\frac{\pi}{2}(\omega(\alpha)-\omega(\beta)))}\Big)^2 \partial_{\beta}D_t^0 (I-\mathcal{H}_{p})(\omega(\beta)-\beta)d\beta\\
&+(I-\mathcal{H}_{p})[(D_t^0)^2-iA_0\partial_{\alpha}, D_t^0](I-\mathcal{H}_{p})(\omega-\alpha)+(I-\mathcal{H}_{p})D_t^0 \mathcal{G}.
\end{split}
\end{equation}
And we have
\begin{equation}\label{okok1}
\begin{split}
&((D_t^0)^2-iA_0\partial_{\alpha})(I-\mathcal{H}_p)\tilde{D}_t^0(I-\tilde{\mathcal{H}}_{p})(\tilde{\omega}-\alpha)\\
=& -[(D_t^0)^2-iA_0\partial_{\alpha},\mathcal{H}_{p}]\tilde{D}_t^0(I-\tilde{\mathcal{H}}_{p})(\tilde{\omega}-\alpha)\\
&+(I-\mathcal{H}_{p})((D_t^0)^2-iA_0\partial_{\alpha})\tilde{D}_t^0(I-\tilde{\mathcal{H}}_{p})(\tilde{\omega}-\alpha)\\
=&-2[D_t^0\omega,\mathcal{H}_{p}]\frac{\partial_{\alpha}D_t^0 \tilde{D}_t^0(I-\tilde{\mathcal{H}}_{p})(\tilde{\omega}-\alpha)}{\omega_{\alpha}}\\
&+\frac{1}{4\pi i}\int_{\TT} \Big(\frac{D_t^0\omega(\alpha)-D_t^0\omega(\beta)}{\sin(\frac{\pi}{2}(\omega(\alpha)-\omega(\beta)))}\Big)^2 \partial_{\beta}\tilde{D}_t^0 (I-\tilde{\mathcal{H}}_{p})(\tilde{\omega}-\beta)(\beta)d\beta\\
&+(I-\mathcal{H}_{p})((D_t^0)^2-iA_0\partial_{\alpha}) \tilde{D}_t^0(I-\tilde{\mathcal{H}}_{p})(\tilde{\omega}-\alpha)\\
=& -2[D_t^0\omega,\mathcal{H}_{p}]\frac{\partial_{\alpha}D_t^0 \tilde{D}_t^0(I-\tilde{\mathcal{H}}_{p})(\tilde{\omega}-\alpha)}{\omega_{\alpha}}\\
&+\frac{1}{4\pi i}\int_{\TT} \Big(\frac{D_t^0\omega(\alpha)-D_t^0\omega(\beta)}{\sin(\frac{\pi}{2}(\omega(\alpha)-\omega(\beta)))}\Big)^2 \partial_{\beta}\tilde{D}_t^0 (I-\tilde{\mathcal{H}}_{p})(\tilde{\omega}-\beta)(\beta)d\beta\\
&+(I-\mathcal{H}_{p})((\tilde{D}_t^0)^2-i\tilde{A}_0\partial_{\alpha}) \tilde{D}_t^0(I-\tilde{\mathcal{H}}_{p})(\tilde{\omega}-\alpha)\\
&+(I-\mathcal{H}_{p})((D_t^0)^2-iA_0\partial_{\alpha}-(\tilde{D}_t^0)^2+i\tilde{A}_0\partial_{\alpha})\tilde{D}_t^0(I-\tilde{\mathcal{H}}_{p})(\tilde{\omega}-\alpha).
\end{split}
\end{equation}
Use (\ref{okok}) and (\ref{okok1}), we have 
\begin{equation}
((D_t^0)^2-iA_0\partial_{\alpha})\sigma_0=\text{fourth~order}.
\end{equation}
Use alsmost the derivation and the estimates as that in N. Totz and S. Wu's work (\cite{Totz2012}), we obtain the following theorem.
\begin{thm}\label{errorperiodic}
Let $s'\geq 6$. Let $\tilde{\omega}, \tilde{b}_0, \tilde{A}_0$ be given as in Section 5.  There is compatible initial data $$(\omega(0), D_t^0\omega(0), (D_t^0)^2\omega(0))$$ to water waves system (\ref{special}) such that 
\begin{equation}
\begin{split}
&\norm{(\omega_{\alpha}(0)-1, D_t^0\omega(0), (D_t^0)^2\omega(0))-(\tilde{\omega}_{\alpha}-1, \tilde{D}_t^0\tilde{\omega}(t), (\tilde{D}_t^0)^2\tilde{\omega}(t))}_{\mathcal{H}^{s'}(\TT)}\leq C \epsilon^2,
\end{split}
\end{equation}
where $\mathcal{H}^s(\mathbb{T}):=H^{s'}(\mathbb{T})\times H^{s'+1/2}(\mathbb{T})\times H^{s'}(\mathbb{T})$. 
 Moreover, for all such initial data, there is a a unique solution $(\omega, D_t^0\omega, b_0, A_0)$ to (\ref{special}) on time interval $[0, C_0\epsilon^{-2}]$ for some $C_0=C_0(s)>0$ such that 
\begin{equation}
\begin{split}
&\sup_{t\in [0, C_0\epsilon^{-2}]}\norm{(\omega_{\alpha}-1, D_t^0\omega(t), (D_t^0)^2\omega(t))-(\tilde{\omega}_{\alpha}-1, \tilde{D}_t^0\tilde{\omega}(t), (\tilde{D}_t^0)^2\tilde{\omega}(t))}_{\mathcal{H}^{s'}(\mathbb{T})}\leq C \epsilon^{3/2}.
\end{split}
\end{equation}
Also, there is some constant $C=C(s)$ such that
\begin{equation}
\sup_{t\in [0, C_0\epsilon^{-2}]}\Big(\norm{b_0-\tilde{b}_0}_{H^{s'}(\mathbb{T})}+\norm{A_0-\tilde{A}_0}_{H^{s'}(\mathbb{T})}+\norm{D_t^0b_0-\tilde{D}_t^0\tilde{b}_0}_{H^{s'}(\TT)}\Big)\leq C\epsilon^{5/2},.
\end{equation}
\begin{equation}\label{A0difference}
\sup_{t\in [0, C_0\epsilon^{-2}]}\Big(\norm{b_0-\tilde{b}_0}_{W^{s'-1,\infty}}+\norm{A_0-\tilde{A}_0}_{W^{s'-1,\infty}}+\norm{D_t^0b_0-\tilde{D}_t^0\tilde{b}_0}_{W^{s'-1,\infty}}\Big)\leq C\epsilon^{3}.
\end{equation}
\end{thm}

\section{Governing equation for $r_1$} \label{govern}
In this section, we derive a governing equation for the remainder term $r_1$. Because we need to obtain long time energy estimates for the error term (we'll prove that the error term $r_1$ has norm $O(\epsilon^{3/2})$ in Sobolev space), the nonlinearities of the equations governing $r_1$ has to be at least of fourth order. Since $(D_t^2-iA\partial_{\alpha})r_1$ is not obviously fourth order, we find some equivalent quantity of $r_1$ and consider its water wave equation. 
\subsection{Governing equation for $r_1$}  Recall that
$$\lambda:=(I-\mathcal{H}_{\zeta})\xi-(I-\mathcal{H}_{\omega})\xi_0.$$
We have shown that $\mathcal{P}\lambda$ consists of cubic and higher order nonlinearities. 
Define
$$\tilde{\lambda}:=(I-\mathcal{H}_{\tilde{\zeta}})\tilde{\xi}-(I-\mathcal{H}_{\tilde{\omega}})\tilde{\xi}_0.$$
Because $\tilde{\lambda}$ approximates $\lambda$ to $O(\epsilon^4)$, we expect $$\mathcal{P}\tilde{\lambda}=O(\epsilon^3),  \quad and \quad \mathcal{P}(\lambda-\tilde{\lambda})=O(\epsilon^4).$$
Also,  as we can see later, $\lambda-\tilde{\lambda}$ is equivalent to $r_1$ in appropriate sense. So it's natural to consider $\mathcal{P}(\lambda-\tilde{\lambda})$.  However, $\lambda-\tilde{\lambda}$ is not the boundary value of a holomorhpic function in $\Omega(t)^c$. There is the trouble of losing derivatives in energy estimates if we use $\lambda-\tilde{\lambda}$. To resolve this problem, we consider the quantity $\rho_1$ define by
\begin{equation}
\rho_1:=(I-\mathcal{H}_{\zeta})(\lambda-\tilde{\lambda}).
\end{equation}
Then $\rho_1$ is holomorphic in $\Omega(t)^c$.

We show that $\mathcal{P}\rho_1$ consists of fourth and higher order terms. The idea is to take advantage of the facts that $\mathcal{P}[(I-\mathcal{H}_{\zeta})(\zeta-\alpha)-(I-\mathcal{H}_{\omega})(\omega-\alpha)]$ is approximated by $\tilde{\mathcal{P}}[(I-\mathcal{H}_{\tilde{\zeta}})(\tilde{\zeta}-\alpha)-(I-\mathcal{H}_{\tilde{\omega}})(\tilde{\omega}-\alpha)]$ to $O(\epsilon^4)$, so their difference would be of order $O(\epsilon^4)$.

To be precise, because $\tilde{\zeta}, \tilde{b}$ approximate $\zeta, b$ to the order of $O(\epsilon^4)$, we have $(\tilde{\zeta}, \tilde{D}_t\tilde{\zeta}, \tilde{b}, \tilde{A})$ satisfy (\ref{lambda}) to the order of $O(\epsilon^4)$, i.e., 
\begin{align*}
&(\tilde{D}_t^2-i\tilde{A}\partial_{\alpha})[(I-\mathcal{H}_{\tilde{\zeta}})(\tilde{\zeta}-\alpha)-(I-\mathcal{H}_{\tilde{\omega}})(\tilde{\omega}-\alpha)]\\
=& -2[\tilde{D}_t\tilde{\zeta},\mathcal{H}_{\tilde{\zeta}}\frac{1}{\tilde{\zeta}_{\alpha}}+\bar{\mathcal{H}}_{\tilde{\zeta}}\frac{1}{\bar{\tilde{\zeta}}_{\alpha}}]\partial_{\alpha}\tilde{D}_t\tilde{\zeta}+2[\tilde{D}_t^0\tilde{\omega},\mathcal{H}_{\tilde{\omega}}\frac{1}{\tilde{\omega}_{\alpha}}+\bar{\mathcal{H}}_{\tilde{\omega}}\frac{1}{\bar{\tilde{\omega}}_{\alpha}}]\partial_{\alpha}\tilde{D}_t^0\tilde{\omega}\\
&+\frac{1}{\pi i}\int \Big(\frac{\tilde{D}_t\tilde{\zeta}(\alpha,t)-\tilde{D}_t\tilde{\zeta}(\beta,t)}{\tilde{\zeta}(\alpha,t)-\tilde{\zeta}(\beta,t)}\Big)^2(\tilde{\zeta}-\bar{\tilde{\zeta}})_{\beta}d\beta\\
&-\frac{1}{\pi i}\int \Big(\frac{\tilde{D}_t^0\tilde{\omega}(\alpha,t)-\tilde{D}_t^0\tilde{\omega}(\beta,t)}{\tilde{\omega}(\alpha,t)-\tilde{\omega}(\beta,t)}\Big)^2(\tilde{\omega}-\bar{\tilde{\omega}})_{\beta}d\beta\\
&+ \tilde{D}_t\tilde{b}_1 \partial_{\alpha}(I-\mathcal{H}_{\tilde{\omega}})(\tilde{\omega}-\alpha)+\tilde{b}_1\tilde{ D}_t\partial_{\alpha}(I-\mathcal{H}_{\tilde{\omega}})(\tilde{\omega}-\alpha)\\
&+\tilde{b}_1\partial_{\alpha}\tilde{D}_t^0(I-\mathcal{H}_{\omega})(\tilde{\omega}-\alpha)-i\tilde{A}_1\partial_{\alpha}(I-\mathcal{H}_{\tilde{\omega}})(\tilde{\omega}-\alpha)+\epsilon^4\mathcal{R},
\end{align*}
where $\epsilon^4\mathcal{R}$ is a known function (in terms of $B(X,T)$) which satisfies 
$$\norm{\epsilon^4 \mathcal{R}}_{H^s}\leq C\epsilon^{7/2}.$$

\begin{remark}
Throughout this paper, we'll use the notation $\epsilon^4\mathcal{R}$ frequently (and sometimes $\epsilon^4\mathcal{R}_1, \epsilon^4\mathcal{R}_2$). It might represent different quantities. However, it always represents a quantity which is in terms of $B(X,T)$ and $\epsilon$, and satisfies the estimate
$$\norm{\epsilon^4 \mathcal{R}}_{H^{s+7}}\leq C\epsilon^{7/2}.$$
\end{remark}

So we have 
\begin{equation}\label{lamlam}
\begin{split}
&(D_t^2-iA\partial_{\alpha})(\lambda-\tilde{\lambda})
=\mathcal{P}[(I-\mathcal{H}_{\zeta})(\zeta-\alpha)-(I-\mathcal{H}_{\omega})(\omega-\alpha)]\\
&-\tilde{\mathcal{P}}[(I-\mathcal{H}_{\tilde{\zeta}})(\tilde{\zeta}-\alpha)-(I-\mathcal{H}_{\tilde{\omega}})(\tilde{\omega}-\alpha)]\\
&+(\mathcal{P}-\tilde{\mathcal{P}})[(I-\mathcal{H}_{\tilde{\zeta}})(\tilde{\zeta}-\alpha)-(I-\mathcal{H}_{\tilde{\omega}})(\tilde{\omega}-\alpha)]\\
:=& \sum_{m=1}^5 \mathcal{R}_{1m}.
\end{split}
\end{equation}
where
\begin{equation}
\begin{split}
\mathcal{R}_{11}=&  -2[D_t\zeta,\mathcal{H}_{\zeta}\frac{1}{\zeta_{\alpha}}+\bar{\mathcal{H}}_{\zeta}\frac{1}{\bar{\zeta}_{\alpha}}]\partial_{\alpha}D_t\zeta+2[D_t^0\omega,\mathcal{H}_{\omega}\frac{1}{\omega_{\alpha}}+\bar{\mathcal{H}}_{\omega}\frac{1}{\bar{\omega}_{\alpha}}]\partial_{\alpha}D_t^0\omega\\
&-2[\tilde{D}_t\tilde{\zeta},\mathcal{H}_{\tilde{\zeta}}\frac{1}{\tilde{\zeta}_{\alpha}}+\bar{\mathcal{H}}_{\tilde{\zeta}}\frac{1}{\bar{\tilde{\zeta}}_{\alpha}}]\partial_{\alpha}\tilde{D}_t\tilde{\zeta}+2[\tilde{D}_t^0\tilde{\omega},\mathcal{H}_{\tilde{\omega}}\frac{1}{\tilde{\omega}_{\alpha}}+\bar{\mathcal{H}}_{\tilde{\omega}}\frac{1}{\bar{\tilde{\omega}}_{\alpha}}]\partial_{\alpha}\tilde{D}_t^0\tilde{\omega}.
\end{split}
\end{equation}

\begin{equation}
\begin{split}
\mathcal{R}_{12}=& \frac{1}{\pi i}\int \Big(\frac{D_t\zeta(\alpha,t)-D_t\zeta(\beta,t)}{\zeta(\alpha,t)-\zeta(\beta,t)}\Big)^2(\zeta-\bar{\zeta})_{\beta}d\beta
-\frac{1}{\pi i}\int \Big(\frac{D_t^0\omega(\alpha,t)-D_t^0\omega(\beta,t)}{\omega(\alpha,t)-\omega(\beta,t)}\Big)^2(\omega-\bar{\omega})_{\beta}d\beta\\
&-\Big\{\frac{1}{\pi i}\int \Big(\frac{\tilde{D}_t\tilde{\zeta}(\alpha,t)-\tilde{D}_t\tilde{\zeta}(\beta,t)}{\tilde{\zeta}(\alpha,t)-\tilde{\zeta}(\beta,t)}\Big)^2(\tilde{\zeta}-\bar{\tilde{\zeta}})_{\beta}d\beta-\frac{1}{\pi i}\int \Big(\frac{\tilde{D}_t^0\tilde{\omega}(\alpha,t)-\tilde{D}_t^0\tilde{\omega}(\beta,t)}{\tilde{\omega}(\alpha,t)-\tilde{\omega}(\beta,t)}\Big)^2(\tilde{\omega}-\bar{\tilde{\omega}})_{\beta}d\beta\Big\}
\end{split}
\end{equation}

\begin{equation}
\begin{split}
\mathcal{R}_{13}=&  D_tb_1 \partial_{\alpha}(I-\mathcal{H}_{\omega})(\omega-\alpha)+b_1 D_t\partial_{\alpha}(I-\mathcal{H}_{\omega})(\omega-\alpha)\\
&-\Big\{ \tilde{D}_t\tilde{b}_1 \partial_{\alpha}(I-\mathcal{H}_{\tilde{\omega}})(\tilde{\omega}-\alpha)+\tilde{b}_1\tilde{ D}_t\partial_{\alpha}(I-\mathcal{H}_{\tilde{\omega}})(\tilde{\omega}-\alpha)\Big\}
\end{split}
\end{equation}

\begin{equation}
\begin{split}
\mathcal{R}_{14}=&b_1\partial_{\alpha}D_t^0(I-\mathcal{H}_{\omega})(\omega-\alpha)-iA_1\partial_{\alpha}(I-\mathcal{H}_{\omega})(\omega-\alpha)\\
&-\Big\{\tilde{b}_1\partial_{\alpha}\tilde{D}_t^0(I-\mathcal{H}_{\omega})(\tilde{\omega}-\alpha)-i\tilde{A}_1\partial_{\alpha}(I-\mathcal{H}_{\tilde{\omega}})(\tilde{\omega}-\alpha)\Big\}
\end{split}
\end{equation}

\begin{equation}
\mathcal{R}_{15}=\epsilon^4\mathcal{R}.
\end{equation}
Denote $\mathcal{R}_{16}$ and $\mathcal{R}_{17}$ as follows:
\begin{equation}
\mathcal{R}_{16}=-2[D_t\zeta,\mathcal{H}_{\zeta}]\frac{\partial_{\alpha}D_t(\lambda-\tilde{\lambda})}{\zeta_{\alpha}}.
\end{equation}

\begin{equation}
\mathcal{R}_{17}=\frac{1}{\pi i}\int \Big(\frac{D_t\zeta(\alpha)-D_t\zeta(\beta)}{\zeta(\alpha)-\zeta(\beta)}\Big)^2\partial_{\beta}(\lambda-\tilde{\lambda})d\beta.
\end{equation}
Then by (\ref{lamlam}), lemma \ref{lemmaseven}, we have 
\begin{equation}\label{rho1equation}
\begin{split}
\mathcal{P}\rho_1=& \mathcal{P}(I-\mathcal{H}_{\zeta})(\lambda-\tilde{\lambda})
= (I-\mathcal{H}_{\zeta})\mathcal{P}(\lambda-\tilde{\lambda})-[\mathcal{P},\mathcal{H}_{\zeta}](\lambda-\tilde{\lambda})\\
=& (I-\mathcal{H}_{\zeta})\sum_{m=1}^5 \mathcal{R}_{1m}+\mathcal{R}_{16}+\mathcal{R}_{17}.
\end{split}
\end{equation}
Note that $\mathcal{R}_{16}$ is quadratic and $\mathcal{R}_{17}$ is cubic, so we need to explore the cancellations hidden behind when we estimate these terms.

\subsection{Governing equation for time evolution of $r_1$} We need to control $D_t r_1$ as well.  Denote
$$\delta:=D_t(I-\mathcal{H}_{\zeta})(\zeta-\alpha)-D_t^0(I-\mathcal{H}_{\omega})(\omega-\alpha).$$
We know that $\mathcal{P}\delta$ consists of cubic and higher order terms. 
Denote
$$\tilde{\delta}:=\tilde{D}_t(I-\mathcal{H}_{\tilde{\zeta}})(\tilde{\zeta}-\alpha)-\tilde{D}_t^0(I-\mathcal{H}_{\tilde{\omega}})(\tilde{\omega}-\alpha).$$
Then because $\tilde{\mathcal{P}}\tilde{\delta}$ approximates $\mathcal{P}\delta$ to $O(\epsilon^4)$, and $\mathcal{P}-\tilde{\mathcal{P}}=O(\epsilon^3)$,  we expect $\mathcal{P}(\delta-\tilde{\delta})=O(\epsilon^4)$. However, $\delta-\tilde{\delta}$ is not holomorphic in $\Omega(t)^c$, which would lose derivatives in energy estimates. So we consider the quantity $\sigma_1:=\delta-\tilde{\delta}$.

  By direct calculation, we have 

\begin{equation}
\begin{split}
\mathcal{P} \delta
=& D_tG-D_t^0G_0+(\mathcal{P}-\mathcal{P}_0)D_t^0(I-\mathcal{H}_{\omega})(\omega-\alpha)\\
&+[\mathcal{P}, D_t](I-\mathcal{H}_{\zeta})(\zeta-\alpha)-[\mathcal{P}_0, D_t^0](I-\mathcal{H}_{\omega})(\omega-\alpha).
\end{split}
\end{equation}

\noindent and
\begin{equation}
\begin{split}
\tilde{\mathcal{P}}\tilde{\delta}
=& \tilde{D}_t\tilde{G}-\tilde{D}_t^0\tilde{G}_0+(\tilde{\mathcal{P}}-\tilde{\mathcal{P}}_0)\tilde{D}_t^0(I-\mathcal{H}_{\tilde{\omega}})(\tilde{\omega}-\alpha)+\epsilon^4\mathcal{R}\\
&+[\tilde{\mathcal{P}},\tilde{D}_t](I-\mathcal{H}_{\tilde{\zeta}})(\tilde{\zeta}-\alpha)-[\tilde{\mathcal{P}}_0, \tilde{D}_t^0](I-\mathcal{H}_{\tilde{\omega}})(\tilde{\omega}-\alpha)
\end{split}
\end{equation}
So we have 
\begin{equation}
\begin{split}
\mathcal{P}(\delta-\tilde{\delta})=&\mathcal{P}\delta-\tilde{\mathcal{P}}\tilde{\delta}
+(\mathcal{P}-\tilde{\mathcal{P}})\tilde{\delta}+\epsilon^4\mathcal{R}\\
=& D_tG-D_t^0G_0-\tilde{D}_t\tilde{G}+\tilde{D}_t^0\tilde{G}_0+(\mathcal{P}-\mathcal{P}_0)D_t^0(I-\mathcal{H}_{\omega})(\omega-\alpha)\\
&-(\tilde{\mathcal{P}}-\tilde{\mathcal{P}}_0)\tilde{D}_t^0(I-\mathcal{\tilde{\omega}})(\tilde{\omega}-\alpha)+[\mathcal{P}, D_t](I-\mathcal{H}_{\zeta})(\zeta-\alpha)\\
&-[\mathcal{P}_0, D_t^0](I-\mathcal{H}_{\omega})(\omega-\alpha)-[\tilde{\mathcal{P}},\tilde{D}_t](I-\mathcal{H}_{\tilde{\zeta}})(\tilde{\zeta}-\alpha)\\
&+[\tilde{\mathcal{P}}_0, \tilde{D}_t^0](I-\mathcal{H}_{\tilde{\omega}})(\tilde{\omega}-\alpha)+\epsilon^4\mathcal{R}\\
:=&\mathcal{S}_1.
\end{split}
\end{equation}
Then $\mathcal{S}_1$ is fourth order. So we have 
\begin{equation}\label{sigma1equation}
\begin{split}
\mathcal{P}\sigma_1=&\mathcal{P}(I-\mathcal{H}_{\zeta})(\delta-\tilde{\delta})\\
=&(I-\mathcal{H}_{\zeta})\mathcal{P}(\delta-\tilde{\delta})-[\mathcal{P},\mathcal{H}](\delta-\tilde{\delta})\\
=&(I-\mathcal{H}_{\zeta})\mathcal{S}_1-2[D_t\zeta,\mathcal{H}_{\zeta}]\frac{\partial_{\alpha}D_t(\delta-\tilde{\delta})}{\zeta_{\alpha}}\\
&+\frac{1}{\pi i}\int \Big(\frac{D_t\zeta(\alpha)-D_t\zeta(\beta)}{\zeta(\alpha)-\zeta(\beta)}\Big)^2\partial_{\beta}(\delta-\tilde{\delta})d\beta.
\end{split}
\end{equation}
We use equations (\ref{rho1equation}) and (\ref{sigma1equation}) to study the evolution of  $r_1$.  A first step is to construct an appropriate energy which controls certain norm of $r_1$, and then show that this control exists for a sufficiently long time.

\subsection{Construction of energy} In this subsection, we construct energy for the water wave equations (\ref{rho1equation}) and (\ref{sigma1equation}). The energy is essentially the same as the energy used by S. Wu in \cite{Wu2009} and the energy by N. Totz and S. Wu in \cite{Totz2012}.  

First, let's recall the basic energy estimates by S. Wu(\cite{Wu2009}):
\begin{lemma}[Basic lemma]\label{basic}
Let $\Theta$ satisfies the equation
$$(D_t^2-iA\partial_{\alpha})\Theta=G$$ and $\Theta$ is smooth and decays fast at infinity. Let
\begin{equation}
E_0(t):=\int\frac{1}{A}|D_t\Theta(\alpha,t)|^2+i\Theta(\alpha,t)\partial_{\alpha}\bar{\Theta}(\alpha,t)d\alpha.
\end{equation}
Then 
\begin{equation}
\frac{dE_0}{dt}=\int \frac{2}{A}\Re(D_t\Theta \bar{G})-\frac{1}{A} \frac{a_t}{a}\circ \kappa^{-1}|D_t\Theta|^2 d\alpha.
\end{equation}
Moreover, if $\Theta$ is the boundary value of a holomorphic function in $\Omega(t)^c$, then 
\begin{equation}
\int i\Theta\partial_{\alpha}\bar{\Theta}d\alpha=-\int i\bar{\Theta}\partial_{\alpha}\Theta d\alpha \geq 0.
\end{equation}
\end{lemma}

\vspace*{2ex}

\noindent \textbf{Notations:} Denote
\begin{equation}
\rho_1^{(n)}:=\partial_{\alpha}^n\rho_1,\quad \quad \quad \sigma_1^{(n)}:=\partial_{\alpha}^n \sigma_1.
\end{equation}
Because $\rho_1^{(n)}$ and $\sigma_1^{(n)}$ are not necessarily holomorphic in $\Omega(t)^c$, if we decompose them as
\begin{equation}\label{decompose}
\begin{split}
\rho_1^n=&\frac{1}{2}(I-\mathcal{H}_{\zeta})\rho_1^n+\frac{1}{2}(I+\mathcal{H}_{\zeta})\rho_1^{(n)}:=\phi_1^{(n)}+\mathcal{R}_1^{(n)}\\
\sigma_1^{(n)}=&\frac{1}{2}(I-\mathcal{H}_{\zeta})\sigma_1^{(n)}+\frac{1}{2}(I+\mathcal{H}_{\zeta})\sigma_1^n:=\Psi_1^{(n)}+\mathcal{S}_1^{(n)}.
\end{split}
\end{equation}
and define
\begin{equation}
\mathcal{E}_{n}(t):=\int \frac{1}{A}|D_t\rho_1^{(n)}|^2+i\phi_1^{(n)}\partial_{\alpha}\bar{\phi}_1^{(n)}d\alpha.
\end{equation}

\begin{equation}
\mathcal{F}_{n}(t):=\int \frac{1}{A}|D_t\sigma_1^{(n)}|^2+i\sigma_1^{(n)}\partial_{\alpha}\bar{\sigma}_1^{(n)}d\alpha.
\end{equation}
Define the energy as
\begin{equation}
\mathcal{E}(t):=\sum_{n=0}^s ( \mathcal{E}_{n}(t)+ \mathcal{F}_{n}(t)).
\end{equation}
By lemma \ref{basic},  each $\mathcal{E}_n$ is positive.  $\sigma_1^n$ might not be holomorphic in $\Omega(t)^c$.  However, we'll show that it is still essentially positive.  We'll show that this energy controls $\norm{D_tr_1}_{H^s}+\norm{(r_1)_{\alpha}}_{H^s}$.

\subsection{Evolution of $\mathcal{E}_n$ and $\mathcal{F}_n$}  To show that $r_1$ remains small (in the sense of some appropriate norm), we need to show that the energy $\mathcal{E}_s$ remains small for a long time. So we need to analyze the evolution of $\mathcal{E}_n$ and $\mathcal{F}_n$.  Note that
\begin{equation}\label{hhaa}
\begin{split}
(D_t^2-iA\partial_{\alpha})\rho_1^{(n)}=& \partial_{\alpha}^n (D_t^2-iA\partial_{\alpha})\rho_1+[D_t^2-iA\partial_{\alpha}, \partial_{\alpha}^n]\rho_1\\
=& \partial_{\alpha}^n(I-\mathcal{H}_{\zeta})\sum_{m=1}^5 \mathcal{R}_{1m}+\partial_{\alpha}^n(\mathcal{R}_{16}+\mathcal{R}_{17})+[D_t^2-iA\partial_{\alpha}, \partial_{\alpha}^n]\rho_1\\
:=& \mathcal{C}_{1,n}.
\end{split}
\end{equation}
Similarly, we derive governing equation for $\sigma_1^{(n)}=\partial_{\alpha}^n \sigma_1$. We have 
\begin{equation}\label{hhaa1}
\begin{split}
(D_t^2-iA\partial_{\alpha})\sigma_1^{(n)}=& \partial_{\alpha}^n (D_t^2-iA\partial_{\alpha})\sigma_1+[D_t^2-iA\partial_{\alpha}, \partial_{\alpha}^n]\sigma_1\\
:=& C_{2,n}.
\end{split}
\end{equation}

\noindent By basic lemma \ref{basic}, equations (\ref{hhaa}) and (\ref{hhaa1}), we have 
\begin{equation}
\begin{split}
\frac{d}{dt}\mathcal{E}_n(t)=& \int \frac{2}{A}\Re(D_t\rho_1^{(n)} \bar{\mathcal{C}}_{1,n}) -\frac{1}{A}\frac{a_t}{a}\circ \kappa^{-1} |D_t\rho_{1,n}^{(n)}|^2 d\alpha\\
&+2\Im \int \partial_t\mathcal{R}_1^{(n)}\partial_{\alpha}\bar{\phi}_1^{(n)}+\partial_t\mathcal{\phi}_1^{(n)}\partial_{\alpha}\bar{\mathcal{R}}_{1,n}^{(n)}
+\partial_t\mathcal{R}_{1}^{(n)}\partial_{\alpha}\bar{\mathcal{R}}_{1}^{(n)}
\end{split}
\end{equation}
And
\begin{equation}
\begin{split}
\frac{d}{dt}\mathcal{F}_n^{\sigma_1}(t)=& \int \frac{2}{A}\Re(D_t\sigma_1^{(n)} \bar{\mathcal{C}}_{2,n}) -\frac{1}{A}\frac{a_t}{a}\circ \kappa^{-1} |D_t\sigma_1^{(n)}|^2 d\alpha.
\end{split}
\end{equation}

\section{Bound for some quantities}\label{bounds}
In this section, we obtain bounds for the quantities which will be used in the energy estimates in next section. We bound these quantities in terms of an auxiliary quantity $E_s$, which is essentially equivalent to the energy $\mathcal{E}_s$. 

\subsection{An a priori assumption}
Let $T_0>0$, we make the following a priori assumption 
\begin{equation}\label{boot}
    \sup_{t\in [0,T_0]}\Big(\norm{(r_1)_{\alpha}}_{H^s}+\norm{D_t r_1}_{H^{s}}+\norm{D_t^2 r_1}_{H^{s}}\Big)\leq \epsilon.
\end{equation}
\begin{remark}
We'll eventually show that $\mathcal{E}_s\lesssim \epsilon^3$ and 
\begin{equation}
    \sup_{t\in [0,O(\epsilon^{-2})]}\Big(\norm{(r_1)_{\alpha}}_{H^s}+\norm{D_t r_1}_{H^{s+1/2}}+\norm{D_t^2 r_1}_{H^{s}}\Big)\lesssim \epsilon^{3/2},.
\end{equation}
which is much better than (\ref{boot}). Since this a priori assumption is easy to justify by a bootstrap argument, we won't provide the details for this justification.
\end{remark}
\vspace*{1ex}

\noindent \textbf{Convention.} In this and the next section, if not specified, then $$0\leq t\leq \min\{T_0, e_0\epsilon^{-2}, C_0\epsilon^{-2}, C_1\epsilon^{-2}\}$$
and the bootstrap assumption (\ref{boot}) holds. Here, $C_0$ is the same as that in Theorem \ref{longperiodic}, $e_0$ is the same as that in Theorem \ref{NLSlocal}, and $C_1$ is the same as that in Theorem \ref{theorem1}.

\subsubsection{Consequence of the a priori assumption}
\begin{lemma}\label{lemmazeta}
We have 
\begin{equation}
    \|\zeta_{\alpha}-1\|_{W^{s-1,\infty}}\leq C\epsilon.
\end{equation}
\end{lemma}
\begin{proof}
We have
\begin{align*}
\zeta_{\alpha}-1=(r_1+r_0+\tilde{\zeta})_{\alpha}-1=(r_1)_{\alpha}+(\tilde{\zeta}_{\alpha}-1+(r_0)_{\alpha}).
\end{align*}
So we have 
\begin{equation}\label{zetaalpha}
\norm{\zeta_{\alpha}-1}_{W^{s-1,\infty}}\leq \norm{(r_1)_{\alpha}}_{W^{s-1,\infty}}+\norm{r_0+\tilde{\zeta}_{\alpha}-1}_{W^{s-1,\infty}}\leq C\epsilon.
\end{equation}
\end{proof}

We'll need the following lemma. Similar versions of this lemma have been appeared in \cite{Wu2009}.
\begin{lemma}\label{realinverse}
Assume the bootstrap assumption (\ref{boot}), let $f, h$ be real functions. Assume 
$$(I-\mathcal{H}_{\zeta})h\bar{\zeta}_{\alpha}=g\quad \quad or\quad \quad (I-\mathcal{H}_{\zeta})h=g.$$
Then we have for any $t\in [0,T_0]$,
\begin{equation}
\|h\|_{H^s}\leq 2\|g\|_{H^s}.
\end{equation}
\end{lemma}

\subsubsection{The equivalence of $\rho_1$ and $r_1$}
\begin{lemma}\label{equivalencequantities}
Assume the a priori assumption (\ref{boot}).  We have 
\begin{equation}
    \norm{\partial_{\alpha}(\rho_1-2r_1)}_{H^s}\leq C(\epsilon E_s^{1/2}+\epsilon^{5/2}),\quad \quad \norm{D_t(\rho_1-2r_1)}_{H^{s+1/2}}\leq C(\epsilon E_s^{1/2}+\epsilon^{5/2}).
\end{equation}
\end{lemma}
\begin{proof}
We have $\lambda:=(I-\mathcal{H}_{\zeta})(\zeta-\alpha)-(I-\mathcal{H}_{\omega})(\omega-\alpha)$, and $\tilde{\lambda}:=(I-\mathcal{H}_{\tilde{\zeta}})(\tilde{\zeta}-\alpha)-(I-\mathcal{H}_{\tilde{\omega}})(\tilde{\omega}-\alpha)$. Recall that $\rho_1:=(I-\mathcal{H}_{\zeta})(\lambda-\tilde{\lambda})$. So we have
\begin{align*}
\partial_{\alpha}\rho_1=& \partial_{\alpha} (I-\mathcal{H}_{\zeta})(\lambda-\tilde{\lambda}).
\end{align*}
We have
\begin{align}
\lambda-\tilde{\lambda}
=& (I-\mathcal{H}_{\zeta})r_1+(\mathcal{H}_{\omega}-\mathcal{H}_{\zeta})(\omega-\alpha)+(\mathcal{H}_{\tilde{\zeta}}-\mathcal{H}_{\zeta})(\tilde{\zeta}-\alpha)-(\mathcal{H}_{\tilde{\omega}}-\mathcal{H}_{\zeta})(\tilde{\omega}-\alpha). 
\end{align}
Denote
\begin{equation}
    \tilde{\gamma}:=(\mathcal{H}_{\omega}-\mathcal{H}_{\zeta})(\omega-\alpha)+(\mathcal{H}_{\tilde{\zeta}}-\mathcal{H}_{\zeta})(\tilde{\zeta}-\alpha)-(\mathcal{H}_{\tilde{\omega}}-\mathcal{H}_{\zeta})(\tilde{\omega}-\alpha). 
\end{equation}
So we have 
\begin{align}
    \partial_{\alpha}\rho_1=&\partial_{\alpha}(I-\mathcal{H}_{\zeta})(\lambda-\tilde{\lambda})\\
    =&\partial_{\alpha}(I-\mathcal{H}_{\zeta})(I-\mathcal{H}_{\zeta})r_1+\partial_{\alpha}(I-\mathcal{H}_{\zeta})\tilde{\gamma}\\
    =& 2\partial_{\alpha} (I-\mathcal{H}_{\zeta})r_1+\partial_{\alpha}(I-\mathcal{H}_{\zeta})\tilde{\gamma}\\
    =& 2\partial_{\alpha}r_1-\partial_{\alpha}(I+\mathcal{H}_{\zeta})r_1+\partial_{\alpha}(I-\mathcal{H}_{\zeta})\tilde{\gamma}.
\end{align}
We are aiming to prove that 
\begin{equation}
    \norm{-\partial_{\alpha}(I+\mathcal{H}_{\zeta})r_1+\partial_{\alpha}(I-\mathcal{H}_{\zeta})\tilde{\gamma}}_{H^s}\leq C\epsilon E_s^{1/2}+C\epsilon^{5/2}.
\end{equation}
For $\norm{\partial_{\alpha}(I+\mathcal{H}_{\zeta})r_1}_{H^s}$, we have 

\begin{align}
    (I+\mathcal{H}_{\zeta})r_1=& (\overline{\mathcal{H}_{\zeta}}+\mathcal{H}_{\zeta})r_1+(I-\overline{\mathcal{H}_{\zeta}})r_1.
\end{align}
The kernal of $\overline{\mathcal{H}_{\zeta}}+\mathcal{H}_{\zeta}$ is of order one, so it's easy  to obtain that
\begin{align}
\norm{\partial_{\alpha}(\overline{\mathcal{H}_{\zeta}}+\mathcal{H}_{\zeta})r_1}_{H^s}\leq C\epsilon E_s^{1/2}.
\end{align}
Decompose $$r_1=(\zeta-\alpha)-(\tilde{\zeta}-\alpha)-(\omega-\alpha)+(\tilde{\omega}-\alpha).$$
Use $(I-\overline{\mathcal{H}_{\zeta}})(\zeta-\alpha)=0$, $(I-\overline{\mathcal{H}_{\omega}})(\omega-\alpha)=0$, we have 
\begin{align}
    (I-\overline{\mathcal{H}_{\zeta}})r_1=& -(I-\overline{\mathcal{H}_{\zeta}})(\tilde{\zeta}-\alpha)+(I-\overline{\mathcal{H}_{\zeta}})(\tilde{\omega}-\alpha)+(\overline{\mathcal{H}_{\zeta}}-\overline{\mathcal{H}_{\omega}})(\omega-\alpha)\\
    =& -(I-\overline{\mathcal{H}_{\zeta}})(\tilde{\zeta}-\alpha)+(I-\overline{\mathcal{H}_{\omega}})(\tilde{\omega}-\alpha)+(\overline{\mathcal{H}_{\zeta}}-\overline{\mathcal{H}_{\omega}})(\omega-\tilde{\omega})
\end{align}
By the construction of $\tilde{\zeta}$ and $\tilde{\omega}$, we have 
\begin{equation}
    (I-\overline{\mathcal{H}_{\zeta}})(\tilde{\zeta}-\alpha)=O(\epsilon^4),\quad \quad (I-\overline{\mathcal{H}_{\omega}})(\tilde{\omega}-\alpha)=O(\epsilon^4).
\end{equation}
Therefore,
\begin{equation}
    \norm{\partial_{\alpha}(I-\overline{\mathcal{H}_{\zeta}})(\tilde{\zeta}-\alpha)-\partial_{\alpha} (I-\overline{\mathcal{H}_{\omega}})(\tilde{\omega}-\alpha)}_{H^s}\leq C\epsilon^{7/2}.
\end{equation}
Since 
\begin{equation}
    \norm{\omega-\tilde{\omega}}_{W^{s,\infty}}=\norm{r_0}_{W^{s,\infty}}\leq C\epsilon^2,
\end{equation}
and
\begin{equation}
    \norm{\partial_{\alpha}(\zeta-\omega)}_{H^s}\leq C\epsilon^{1/2},
\end{equation}
we have 
\begin{equation}
    \norm{(\overline{\mathcal{H}_{\zeta}}-\overline{\mathcal{H}_{\omega}})(\omega-\tilde{\omega})}_{H^s}\leq C\epsilon^{5/2}.
\end{equation}
So we obtain
\begin{equation}
    \norm{\partial_{\alpha}(\rho_1-2r_1)}_{H^s}\leq C(\epsilon E_s^{1/2}+\epsilon^{5/2}).
\end{equation}

\noindent Use similar argument, we have 
\begin{equation}
    \norm{D_t(\rho_1-2r_1)}_{H^{s+1/2}}\leq C(\epsilon E_s^{1/2}+\epsilon^{5/2}).
\end{equation}
\end{proof}

\begin{corollary}\label{equivalencerho1}
Assume the a priori assumption (\ref{boot}), we have 
\begin{equation}
    \norm{\partial_{\alpha}\rho_1}_{H^s}\leq C\epsilon,\quad \quad \norm{D_t\rho_1}_{H^{s+1/2}}\leq C\epsilon.
\end{equation}
\end{corollary}

\subsection{An auxiliary quantity for the energy functional} The energy functional $\mathcal{E}_s$ is not very convenient in the energy estimates, so we introduce the quantity
\begin{equation}
E_s^{1/2}:=\norm{D_t r_1}_{H^s}+\norm{(r_1)_{\alpha}}_{H^s}+\norm{D_t^2r_1}_{H^s}.
\end{equation}
By (\ref{boot}), we have 
\begin{equation}
    \sup_{t\in [0,T_0]}E_s(t)\leq \epsilon.
\end{equation}
We'll show that 
$$E_s\leq C(\mathcal{E}+\epsilon^{5/2}),$$
and we'll control $\frac{d\mathcal{E}}{dt}$ in terms of $E_s$ and $\epsilon$, then we can obtain energy estimates on a lifespan of lengh $O(\epsilon^{-2})$. For this purpose, we control the quantities appear in the energy estimates in terms of $E_s$ and $\epsilon$.

\subsection{Bound $\tilde{b}$, $\tilde{b}_1$, $b_1$ and $b_1-\tilde{b}_1$} 
From the definition of $\tilde{b}$ and $\tilde{b}_1$, we have
\begin{equation}
    \norm{\tilde{b}}_{W^{s,\infty}}\leq C\epsilon^2,
\end{equation}
and
\begin{equation}
    \norm{\tilde{b}_1}_{H^s}\leq C\epsilon^{3/2}.
\end{equation}
It's not difficult to bound $\norm{b_1}_{H^s}$ by $C\epsilon^{3/2}$, and therefore $\norm{b_1-\tilde{b}_1}_{H^s}\leq C\epsilon^{3/2}$. However, it turns out that we need a bound better than $C\epsilon^{3/2}$ for the quantity $\norm{b_1-\tilde{b}_1}_{H^s}$. 

Because $\tilde{b}_1$ approximates $b_1$ to the order of $O(\epsilon^4)$,  we have 
\begin{equation}\label{b1tilde}
\begin{split}
(I-\mathcal{H}_{\tilde{\zeta}})\tilde{b}_1=& -[\tilde{D}_t\tilde{\zeta}-\tilde{D}_t^0\tilde{\omega},\mathcal{H}_{\tilde{\zeta}}]\frac{\bar{\tilde{\zeta}}_{\alpha}-1}{\tilde{\zeta}_{\alpha}}-[\tilde{D}_t^0\tilde{\omega}, \mathcal{H}_{\tilde{\zeta}}\frac{1}{\tilde{\zeta}_{\alpha}}-\mathcal{H}_{\tilde{\omega}}\frac{1}{\tilde{\omega}_{\alpha}}](\bar{\tilde{\zeta}}_{\alpha}-1)\\
&-[\tilde{D}_t^0\tilde{\omega}, \mathcal{H}_{\tilde{\omega}}\frac{1}{\tilde{\omega}_{\alpha}}]\bar{\xi}_{\alpha}-(\mathcal{H}_{\tilde{\zeta}}-\mathcal{H}_{\tilde{\omega}})b_0+\epsilon^4\mathcal{R},
\end{split}
\end{equation}
where $\epsilon^4\mathcal{R}$ is a function of $\tilde{b}_1$ such that $||\epsilon^4\mathcal{B}||_{H^{s+7}}\lesssim \epsilon^{7/2}$. 

To derive a formula for $b_1-\tilde{b}_1$, we subtract (\ref{b1}) from (\ref{b1tilde}).  In order to explore the cancellation relations and obtain good estimates, we group the similar terms together (the terminology 'similar' should be clear in the context).  We obtain the following 
\begin{align*}
&(I-\mathcal{H}_{\zeta})(b_1-\tilde{b}_1)= (I-\mathcal{H}_{\zeta})b_1-(I-\mathcal{H}_{\tilde{\zeta}})\tilde{b}_1-(\mathcal{H}_{\tilde{\zeta}}-\mathcal{H}_{\zeta})\tilde{b}_1\\
=& -[D_t\zeta-D_t^0\omega,\mathcal{H}_{\zeta}]\frac{\bar{\zeta}_{\alpha}-1}{\zeta_{\alpha}}+[\tilde{D}_t\tilde{\zeta}-\tilde{D}_t^0\tilde{\omega},\mathcal{H}_{\tilde{\zeta}}]\frac{\bar{\tilde{\zeta}}_{\alpha}-1}{\tilde{\zeta}_{\alpha}}    \quad \quad \quad \quad & :=B_1\\
&-[D_t^0\omega, \mathcal{H}_{\zeta}\frac{1}{\zeta_{\alpha}}-\mathcal{H}_{\omega}\frac{1}{\omega_{\alpha}}](\bar{\zeta}_{\alpha}-1)+[\tilde{D}_t^0\tilde{\omega}, \mathcal{H}_{\tilde{\zeta}}\frac{1}{\tilde{\zeta}_{\alpha}}-\mathcal{H}_{\tilde{\omega}}\frac{1}{\tilde{\omega}_{\alpha}}](\bar{\tilde{\zeta}}_{\alpha}-1) &:=B_2\\
&-[D_t^0\omega, \mathcal{H}_{\omega}\frac{1}{\omega_{\alpha}}]\bar{\xi}_{\alpha}+[\tilde{D}_t^0\tilde{\omega}, \mathcal{H}_{\tilde{\omega}}\frac{1}{\tilde{\omega}_{\alpha}}]\bar{\tilde{\xi}}_{\alpha}    &:=B_3\\
&-(\mathcal{H}_{\zeta}-\mathcal{H}_{\omega})b_0+(\mathcal{H}_{\tilde{\zeta}}-\mathcal{H}_{\tilde{\omega}})b_0   &:=B_4\\
&+\epsilon^4\mathcal{R}    &:=B_5\\
&-(\mathcal{H}_{\tilde{\zeta}}-\mathcal{H}_{\zeta})\tilde{b}_1    &:=B_6.
\end{align*}
To estimate $B_1$, we write $B_1$ as 
\begin{align*}
B_1=& -[D_t\zeta-D_t^0\omega-\tilde{D}_t\tilde{\zeta}+\tilde{D}_t^0\tilde{\omega}, \mathcal{H}_{\zeta}]\frac{\bar{\zeta}_{\alpha}-1}{\zeta_{\alpha}}\\
&+[\tilde{D}_t\tilde{\zeta}-\tilde{D}_t^0\tilde{\omega}, \mathcal{H}_{\tilde{\zeta}}\frac{1}{\tilde{\zeta}_{\alpha}}-\mathcal{H}_{\zeta}\frac{1}{\zeta_{\alpha}}](\bar{\zeta}_{\alpha}-1)\\
&-[\tilde{D}_t\tilde{\zeta}-\tilde{D}_t^0\tilde{\omega}, \mathcal{H}_{\tilde{\zeta}}\frac{1}{\tilde{\zeta}_{\alpha}}]\bar{r}_{\alpha}\\
:=& B_{11}+B_{12}+B_{13}.
\end{align*}

\vspace*{1ex}
\noindent The terms consist of $B_1$ are 'similar'.  The advantages of writing $B_1$ in this form are:
\begin{itemize}
\item  Each $B_{1j} (j=1,2,3)$ contains a factor which is in $H^s$ ( therefore $L^2$ estimate is possible). 
\item Each $B_{1j}$ contains a factor which explores the cancellation relations between the exact solution and the approximation.
\end{itemize}
\vspace*{1ex}

\subsubsection{Estimate $\norm{B_{11}}_{H^s}$} To estimate $B_{11}$, we rewrite the quantity
\begin{equation}\label{diffdiff}
\begin{split}
&D_t\zeta-D_t^0\omega-\tilde{D}_t\tilde{\zeta}+\tilde{D}_t^0\tilde{\omega}\\
=& (D_t-\tilde{D}_t)\zeta+\tilde{D}_t(\zeta-\tilde{\zeta})-(D_t^0-\tilde{D}_t^0)\omega-\tilde{D}_t^0(\omega-\tilde{\omega})\\
=& (b-\tilde{b})\zeta_{\alpha}+\tilde{D}_t r-(b_0-\tilde{b}_0)\omega_{\alpha}-\tilde{D}_t^0r_0\\
=& (b_1-\tilde{b}_1)\zeta_{\alpha}+(b_0-\tilde{b}_0)(\zeta_{\alpha}-\omega_{\alpha})+(\tilde{D}_t-\tilde{D}_t^0)r_0+\tilde{D}_tr_1\\
=& (b_1-\tilde{b}_1)\zeta_{\alpha}+(b_0-\tilde{b}_0)\partial_{\alpha}\xi_1+\tilde{b}_1\partial_{\alpha}r_0+(\tilde{D}_t-D_t)r_1+D_tr_1\\
=&  (b_1-\tilde{b}_1)\zeta_{\alpha}+(b_0-\tilde{b}_0)\partial_{\alpha}\xi_1+\tilde{b}_1\partial_{\alpha}r_0-(b-\tilde{b})\partial_{\alpha}r_1+D_tr_1\\
=& (b_1-\tilde{b}_1)(\zeta_{\alpha}-(r_1)_{\alpha})+(b_0-\tilde{b}_0)\partial_{\alpha}\tilde{\xi}_1+\tilde{b}_1(r_0)_{\alpha}+D_tr_1\\
=& (b_1-\tilde{b}_1)(\tilde{\zeta}_{\alpha}+(r_0)_{\alpha})+(b_0-\tilde{b}_0)\partial_{\alpha}\tilde{\xi}_1+\tilde{b}_1(r_0)_{\alpha}+D_tr_1.
\end{split}
\end{equation}
Use proposition \ref{singular} and estimate (\ref{zetaalpha}) for $\zeta_{\alpha}-1$,  we have 
\begin{align*}
&\norm{B_{11}}_{H^s}\\
\leq & C\norm{ (b_1-\tilde{b}_1)(\tilde{\zeta}_{\alpha}+(r_0)_{\alpha})+(b_0-\tilde{b}_0)\partial_{\alpha}\tilde{\xi}_1+\tilde{b}_1(r_0)_{\alpha}+D_tr_1}_{H^s}\norm{\zeta_{\alpha}-1}_{W^{s-1,\infty}}\\
\leq & C \Big(\norm{b_1-\tilde{b}_1}_{H^s}\norm{\tilde{\zeta}_{\alpha}+(r_0)_{\alpha}}_{W^{s,\infty}}+\norm{\tilde{b}_1}_{H^s}\norm{(r_0)_{\alpha}}_{W^{s,\infty}}+\norm{b_0-\tilde{b}_0}_{H^s}\norm{\partial_{\alpha}\tilde{\xi}_1}_{H^s}\\
&+\norm{D_tr_1}_{H^s}\Big)\norm{\zeta_{\alpha}-1}_{W^{s-1,\infty}}\\
\leq & C\Big(\norm{b_1-\tilde{b}_1}_{H^s}+\epsilon^{3/2}+E_s^{1/2}\Big)\epsilon.
\end{align*}
Here, we have used (see Theorem \ref{errorperiodic})
$$\norm{\tilde{b}_1}_{H^s}\leq C\epsilon^{3/2}, \quad \norm{(r_0)_{\alpha}}_{W^{s-1,\infty}}\leq C\epsilon^2,\quad \|b_0-\tilde{b}_0\|_{W^{s,\infty}}\leq C\epsilon^2.$$

\subsubsection{Estimate $\norm{B_{12}}_{H^s}$.}
For $B_{12}$, we write 
\begin{align*}
\tilde{D}_t\tilde{\zeta}-\tilde{D}_t^0\tilde{\omega}=& (\tilde{D}_t-\tilde{D}_t^0)\tilde{\zeta}+\tilde{D}_t^0(\tilde{\zeta}-\tilde{\omega})=\tilde{b}_1\partial_{\alpha}\tilde{\zeta}+\tilde{D}_t^0\tilde{\xi}_1.
\end{align*}
Then by proposition \ref{singular}, we have
\begin{equation}
\begin{split}
\norm{B_{12}}_{H^s}= & \norm{\tilde{D}_t\tilde{\zeta}-\tilde{D}_t^0\tilde{\omega}}_{H^s}\norm{\zeta_{\alpha}-1}_{W^{s-1,\infty}}\norm{\tilde{\zeta}-\tilde{\omega}}_{W^{s-1,\infty}}\\
\leq & ||\tilde{b}_1\partial_{\alpha}\tilde{\zeta}+\tilde{D}_t^0\tilde{\xi}_1||_{H^s}||\zeta_{\alpha}-1||_{W^{s-1,\infty}}||\tilde{\zeta}-\tilde{\omega}||_{W^{s-1,\infty}}\\
\leq & C\epsilon^{3/2}||\zeta_{\alpha}-1||_{W^{s-1,\infty}}\\
\leq & C\epsilon^{5/2}.
\end{split}
\end{equation}

\subsubsection{Estimate $\norm{B_{13}}_{H^s}$.} Use proposition \ref{singular}, we have 
\begin{align*}
\norm{B_{13}}_{H^s}=&\norm{[\tilde{b}_1\tilde{\zeta}_{\alpha}+\tilde{D}_t^0\tilde{\xi}_1,\mathcal{H}_{\tilde{\zeta}}\frac{1}{\tilde{\zeta}_{\alpha}}]\bar{r}_{\alpha}}_{H^s}\\
\leq &C\norm{\tilde{b}_1\tilde{\zeta}_{\alpha}+\tilde{D}_t^0\tilde{\xi}_1}_{W^{s,\infty}}\norm{\partial_{\alpha}r_1}_{H^s}+C\norm{\tilde{b}_1\tilde{\zeta}_{\alpha}+\tilde{D}_t^0\tilde{\xi}_1}_{H^{s}}\norm{\partial_{\alpha}r_0}_{W^{s-1,\infty}}\Big)\\
\leq & C\epsilon E_s^{1/2}+\epsilon^{5/2}.
\end{align*}
Here, we've used the estimates
\begin{equation}
    \|\tilde{b}_1\|_{H^s}\leq C\epsilon^{3/2}, \quad \|\tilde{D}_t^0\tilde{\xi}_1\|_{W^{s,\infty}}\leq C\epsilon, \quad \|\partial_{\alpha}r_0\|_{W^{s-1,\infty}}\leq C\epsilon^2,\quad \|\tilde{D}_t^0\tilde{\xi}_1\|_{H^s}\leq C\epsilon^{1/2}.
\end{equation}  
\vspace*{2ex}

\noindent So we have 
\begin{equation}
\norm{B_1}_{H^s}\leq C\epsilon\norm{b_1-\tilde{b}_1}_{H^s}+C\epsilon E_s^{1/2}+C\epsilon^{5/2}.
\end{equation}

\vspace*{1ex}

\noindent Use the same argument, we show that 
\begin{equation}
\norm{B_2}_{H^s}\leq C\epsilon^{5/2}+C\epsilon E_s^{1/2}.
\end{equation}

\begin{equation}
\norm{B_3}_{H^s}\leq C\epsilon^{5/2}+C\epsilon E_s^{1/2}.
\end{equation}
\begin{equation}
\norm{B_4}_{H^s}\leq C\epsilon^{5/2}+C\epsilon E_s^{1/2}.
\end{equation}
For $B_5$, it's trivially that
\begin{equation}
\norm{B_5}_{H^s}\leq C\epsilon^{5/2}.
\end{equation}
And 
\begin{equation}
\norm{B_6}_{H^s}\leq C\epsilon^{5/2}+CE_s^{1/2}.
\end{equation}
By lemma \ref{realinverse}, we have 
\begin{lemma}\label{bdiff}
Assume the a priori assumption (\ref{boot}), then we have 
\begin{equation}\label{diffb1}
\norm{b_1-\tilde{b}_1}_{H^s}\leq C\epsilon^{5/2}+C\epsilon E_s^{1/2}.
\end{equation}
\end{lemma}
\begin{proof}
From the estimates for $B_j,  j=1,...,m$, we have
$$\norm{b_1-\tilde{b}_1}_{H^s}\leq C\epsilon\norm{b_1-\tilde{b}_1}_{H^s}+C\epsilon^{5/2}+C\epsilon E_s^{1/2}.$$
For $\epsilon$ small such that $C\epsilon<1$, we have 
\begin{equation}
\norm{b_1-\tilde{b}_1}_{H^s}\leq C\epsilon^{5/2}+C\epsilon E_s^{1/2}.
\end{equation}
\end{proof}
\begin{corollary}\label{corollaryb}
Under the assumptions of lemma \ref{bdiff}, we have
\begin{equation}
    \|b_1\|_{H^s}\leq C\epsilon^{3/2},\quad \quad \norm{b}_{W^{s-1,\infty}}\leq C\epsilon^{2}.
\end{equation}
\end{corollary}

\begin{corollary}\label{corollaryDtzeta}
Under the assumptions of lemma \ref{bdiff}, we have 
\begin{equation}
     \norm{D_t\zeta}_{W^{s-1,\infty}}\leq C\epsilon.
\end{equation}
\end{corollary}

\begin{proof}We write $D_t\zeta$ as 
\begin{equation}
\begin{split}
D_t\zeta=& D_tr_1+D_tr_0+D_t\tilde{\zeta}\\
=&  D_tr_1+D_t^0r_0+(b-b_0)\partial_{\alpha}r_0+\tilde{D}_t\tilde{\zeta}+(b-\tilde{b})\partial_{\alpha}\tilde{\zeta}.
\end{split}
\end{equation}
So we have 
\begin{equation}\label{Dtzeta}
\norm{D_t\zeta}_{W^{s-1,\infty}}\leq C\epsilon+CE_s^{1/2}\leq C\epsilon.
\end{equation}
\end{proof}

\subsection{Bound $D_tb_1$ and $D_t(b_1-\tilde{b}_1)$} From the definition of $\tilde{D}_t\tilde{b}_1$, it's easy to obtain that 
\begin{equation}
    \norm{\tilde{D}_t\tilde{b}_1}_{H^s}\leq C\epsilon^{3/2},\quad \quad \norm{\tilde{D}_t\tilde{b}_1}_{W^{s,\infty}}\leq C\epsilon^{2}.
\end{equation}

To estimate $D_t(b_1-\tilde{b}_1)$, we need to derive a  formula for $D_t(b_1-\tilde{b}_1)$.  Since $D_t(b_1-\tilde{b}_1)$ is real,  it suffices to estimate $(I-\mathcal{H}_{\zeta})D_t(b_1-\tilde{b}_1)$.  We have 
\begin{align*}
(I-\mathcal{H}_{\zeta})D_t(b_1-\tilde{b}_1)=& \Big((I-\mathcal{H}_{\zeta})D_tb_1-(I-\mathcal{H}_{\tilde{\zeta}})\tilde{D}_t\tilde{b}_1\Big)\\
&+(\mathcal{H}_{\zeta}-\mathcal{H}_{\tilde{\zeta}})\tilde{D}_t\tilde{b}_1-(I-\mathcal{H}_{\zeta})(b_1-\tilde{b}_1)\partial_{\alpha}\tilde{b}_1.
\end{align*}
For $(I-\mathcal{H}_{\zeta})D_tb_1$, we have a formula given by (\ref{Dtb1}).  Since $\tilde{b}_1, \tilde{b}_0, \tilde{\zeta}, \tilde{\omega}$ approximate $b_1, b_0, \zeta, \omega$ to the order $O(\epsilon^4)$, respectively,  we have the following formula for $(I-\mathcal{H}_{\tilde{\zeta}})\tilde{D}_t\tilde{b}_1$:

\begin{equation}\label{tildeDtb1}
\begin{split}
& (I-\mathcal{H}_{\tilde{\zeta}})\tilde{D}_t\tilde{b}\\
=& [\tilde{D}_t\tilde{\zeta},\mathcal{H}_{\tilde{\zeta}}]\frac{\partial_{\alpha}(2\tilde{b}-\tilde{D}_t\bar{\tilde{\zeta}})}{\tilde{\zeta}_{\alpha}}-[\tilde{D}_t^0\tilde{\omega},\mathcal{H}_{\tilde{\omega}}]\frac{\partial_{\alpha}(2\tilde{b}_0-\tilde{D}_t^0\bar{\tilde{\omega}})}{\tilde{\zeta}_{\alpha}}-[\tilde{D}_t^2\tilde{\zeta},\mathcal{H}_{\tilde{\zeta}}]\frac{\bar{\tilde{\zeta}}_{\alpha}-1}{\tilde{\zeta}_{\alpha}}+[(\tilde{D}_t^0)^2\tilde{\omega},\mathcal{H}_{\tilde{\omega}}]\frac{\bar{\tilde{\omega}}_{\alpha}-1}{\tilde{\omega}_{\alpha}}\\
&+\frac{1}{\pi i}\int \Big(\frac{\tilde{D}_t\tilde{\zeta}(\alpha)-\tilde{D}_t\tilde{\zeta}(\beta)}{\tilde{\zeta}(\alpha)-\tilde{\zeta}(\beta)}\Big)^2 (\bar{\tilde{\zeta}}_{\beta}(\beta)-1)d\beta-\frac{1}{\pi i}\int \Big(\frac{\tilde{D}_t^0\tilde{\omega}(\alpha)-\tilde{D}_t^0\tilde{\omega}(\beta)}{\tilde{\omega}(\alpha)-\tilde{\omega}(\beta)}\Big)^2 (\bar{\tilde{\omega}}_{\beta}(\beta)-1)d\beta\\
&+ (\mathcal{H}_{\tilde{\zeta}}-\mathcal{H}_{\tilde{\omega}})\tilde{D}_t^0\tilde{b}_0-(I-\mathcal{H}_{\tilde{\zeta}})\tilde{b}_1\partial_{\alpha}\tilde{b}_0+\epsilon^4\mathcal{R},
\end{split}
\end{equation}
where $\epsilon^4\mathcal{R}$ satisfies
$||\epsilon^4\mathcal{R}||_{H^{s+7}}\leq C\epsilon^{7/2}.$
So we have 
\begin{align*}
&(I-\mathcal{H}_{\zeta})D_t(b_1-\tilde{b}_1):=\sum_{m=1}^6 F_m.
\end{align*}
where each $F_m$ are given as follows:

\vspace*{2ex}

\begin{equation}
\begin{split}
F_1=& [D_t\zeta,\mathcal{H}_{\zeta}]\frac{\partial_{\alpha}(2b-D_t\bar{\zeta})}{\zeta_{\alpha}}-  [\tilde{D}_t\tilde{\zeta},\mathcal{H}_{\tilde{\zeta}}]\frac{\partial_{\alpha}(2\tilde{b}-\tilde{D}_t\bar{\tilde{\zeta}})}{\tilde{\zeta}_{\alpha}}\\
&-[D_t^0\omega,\mathcal{H}_{\omega}]\frac{\partial_{\alpha}(2b_0-D_t^0\bar{\omega})}{\zeta_{\alpha}}+[\tilde{D}_t^0\tilde{\omega},\mathcal{H}_{\tilde{\omega}}]\frac{\partial_{\alpha}(2\tilde{b}_0-\tilde{D}_t^0\bar{\tilde{\omega}})}{\tilde{\zeta}_{\alpha}}.
\end{split}
\end{equation}

\begin{equation}
F_2= -[D_t^2\zeta,\mathcal{H}_{\zeta}]\frac{\bar{\zeta}_{\alpha}-1}{\zeta_{\alpha}}+[\tilde{D}_t^2\tilde{\zeta},\mathcal{H}_{\tilde{\zeta}}]\frac{\bar{\tilde{\zeta}}_{\alpha}-1}{\tilde{\zeta}_{\alpha}}+ [(D_t^0)^2\omega,\mathcal{H}_{\omega}]\frac{\bar{\omega}_{\alpha}-1}{\omega_{\alpha}}-[(\tilde{D}_t^0)^2\tilde{\omega},\mathcal{H}_{\tilde{\omega}}]\frac{\bar{\tilde{\omega}}_{\alpha}-1}{\tilde{\omega}_{\alpha}}.
\end{equation}

\begin{equation}
\begin{split}
F_3=&\frac{1}{\pi i}\int \Big(\frac{D_t\zeta(\alpha)-D_t\zeta(\beta)}{\zeta(\alpha)-\zeta(\beta)}\Big)^2 (\bar{\zeta}_{\beta}(\beta)-1)d\beta-\frac{1}{\pi i}\int \Big(\frac{\tilde{D}_t\tilde{\zeta}(\alpha)-\tilde{D}_t\tilde{\zeta}(\beta)}{\tilde{\zeta}(\alpha)-\tilde{\zeta}(\beta)}\Big)^2 (\bar{\tilde{\zeta}}_{\beta}(\beta)-1)d\beta\\
&-\frac{1}{\pi i}\int \Big(\frac{D_t^0\omega(\alpha)-D_t^0\omega(\beta)}{\omega(\alpha)-\omega(\beta)}\Big)^2 (\bar{\omega}_{\beta}(\beta)-1)d\beta+\frac{1}{\pi i}\int \Big(\frac{\tilde{D}_t^0\tilde{\omega}(\alpha)-\tilde{D}_t^0\tilde{\omega}(\beta)}{\tilde{\omega}(\alpha)-\tilde{\omega}(\beta)}\Big)^2 (\bar{\tilde{\omega}}_{\beta}(\beta)-1)d\beta
\end{split}
\end{equation}

\begin{equation}
F_4= (\mathcal{H}_{\zeta}-\mathcal{H}_{\omega})D_t^0b_0- (\mathcal{H}_{\tilde{\zeta}}-\mathcal{H}_{\tilde{\omega}})\tilde{D}_t^0\tilde{b}_0.
\end{equation}

\begin{equation}
F_5=-(I-\mathcal{H}_{\zeta})b_1\partial_{\alpha}b_0+(I-\mathcal{H}_{\tilde{\zeta}})\tilde{b}_1\partial_{\alpha}\tilde{b}_0.
\end{equation}

\begin{equation}
F_6=\epsilon^4\mathcal{R}
\end{equation}

\subsubsection{Estimate $\norm{F_1}_{H^s}$.} We write $F_1$ as 
\begin{align*}
F_1=& [D_t\zeta-\tilde{D}_t\tilde{\zeta}, \mathcal{H}_{\zeta}\frac{1}{\zeta_{\alpha}}]\partial_{\alpha}(2b-D_t\bar{\zeta})+[\tilde{D}_t\tilde{\zeta},\mathcal{H}_{\zeta}\frac{1}{\zeta_{\alpha}}-\mathcal{H}_{\tilde{\zeta}}\frac{1}{\tilde{\zeta}}]\partial_{\alpha}(2b-D_t\bar{\zeta})
\\
&+[\tilde{D}_t\tilde{\zeta},\mathcal{H}_{\tilde{\zeta}}\frac{1}{\tilde{\zeta}_{\alpha}}]\partial_{\alpha}(2(b-\tilde{b})-(D_t\bar{\zeta}-\tilde{D}_t\bar{\tilde{\zeta}}))\\
&-\Big\{ [D_t^0\omega-\tilde{D}_t^0\tilde{\omega}, \mathcal{H}_{\omega}\frac{1}{\omega_{\alpha}}]\partial_{\alpha}(2b_0-D_t^0\bar{\omega})+[\tilde{D}_t^0\tilde{\omega},\mathcal{H}_{\omega}\frac{1}{\omega_{\alpha}}-\mathcal{H}_{\tilde{\omega}}\frac{1}{\tilde{\omega}}]\partial_{\alpha}(2b_0-D_t^0\bar{\omega})
\\
&+[\tilde{D}_t^0\tilde{\omega},\mathcal{H}_{\tilde{\omega}}\frac{1}{\tilde{\omega}_{\alpha}}]\partial_{\alpha}(2(b_0-\tilde{b}_0)-(D_t^0\bar{\omega}-\tilde{D}_t^0\bar{\tilde{\omega}}))\Big\}\\
=&  \Big\{[D_t\zeta-\tilde{D}_t\tilde{\zeta}, \mathcal{H}_{\zeta}\frac{1}{\zeta_{\alpha}}]\partial_{\alpha}(2b-D_t\bar{\zeta})- [D_t^0\omega-\tilde{D}_t^0\tilde{\omega}, \mathcal{H}_{\omega}\frac{1}{\omega_{\alpha}}]\partial_{\alpha}(2b_0-D_t^0\bar{\omega})\Big\}\\
&+\Big\{[\tilde{D}_t\tilde{\zeta},\mathcal{H}_{\zeta}\frac{1}{\zeta_{\alpha}}-\mathcal{H}_{\tilde{\zeta}}\frac{1}{\tilde{\zeta}}]\partial_{\alpha}(2b-D_t\bar{\zeta})-[\tilde{D}_t^0\tilde{\omega},\mathcal{H}_{\omega}\frac{1}{\omega_{\alpha}}-\mathcal{H}_{\tilde{\omega}}\frac{1}{\tilde{\omega}}]\partial_{\alpha}(2b_0-D_t^0\bar{\omega})\Big\}\\
&+\Big\{[\tilde{D}_t\tilde{\zeta},\mathcal{H}_{\tilde{\zeta}}\frac{1}{\tilde{\zeta}_{\alpha}}]\partial_{\alpha}(2(b-\tilde{b})-(D_t\bar{\zeta}-\tilde{D}_t\bar{\tilde{\zeta}}))-[\tilde{D}_t^0\tilde{\omega},\mathcal{H}_{\tilde{\omega}}\frac{1}{\tilde{\omega}_{\alpha}}]\partial_{\alpha}(2(b_0-\tilde{b}_0)-(D_t^0\bar{\omega}-\tilde{D}_t^0\bar{\tilde{\omega}}))\Big\}\\
:=& F_{11}+F_{12}+F_{13}.
\end{align*}
The estimates for $F_{11}$, $F_{12}$ and $F_{13}$ are similar, so we give the details of $F_{11}$ only. We rewrite $F_{11}$ as
\begin{equation}
\begin{split}
F_{11}=& [D_t\zeta-\tilde{D}_t\tilde{\zeta}-(D_t^0\omega-\tilde{D}_t^0\tilde{\omega}), \mathcal{H}_{\zeta}\frac{1}{\zeta_{\alpha}}]\partial_{\alpha}(2b-D_t\bar{\zeta})\\
&+[D_t^0\omega-\tilde{D}_t^0\tilde{\omega}, \mathcal{H}_{\zeta}\frac{1}{\zeta_{\alpha}}-\mathcal{H}_{\omega}\frac{1}{\omega_{\alpha}}]\partial_{\alpha}(2b-D_t\bar{\zeta})\\
&+[D_t^0\omega-\tilde{D}_t^0\tilde{\omega}, \mathcal{H}_{\omega}\frac{1}{\omega_{\alpha}}]\partial_{\alpha}(2(b-b_0)-(D_t\bar{\zeta}-D_t^0\bar{\omega}))\\
:=& F_{111}+F_{112}+F_{113}.
\end{split}
\end{equation}
For $F_{111}$, use (\ref{diffdiff}), 
\begin{equation}
\begin{split}
&D_t\zeta-\tilde{D}_t\tilde{\zeta}-(D_t^0\omega-\tilde{D}_t^0\tilde{\omega})= (b_1-\tilde{b}_1)(\tilde{\zeta}_{\alpha}+(r_0)_{\alpha})+\tilde{b}_1(r_0)_{\alpha}+D_tr_1.
\end{split}
\end{equation}

Use (\ref{Dtzeta}),  (\ref{diffb1}), and Corollary \ref{corollaryb}, we have 
\begin{equation}
\begin{split}
\norm{F_{111}}_{H^s}\leq & C\norm{D_t\zeta-\tilde{D}_t\tilde{\zeta}-(D_t^0\omega-\tilde{D}_t^0\tilde{\omega})}_{H^s}\norm{2b-D_t\bar{\zeta}}_{W^{s-1,\infty}}\\
\leq & \norm{ (b_1-\tilde{b}_1)(\tilde{\zeta}_{\alpha}+(r_0)_{\alpha})+\tilde{b}_1(r_0)_{\alpha}+D_tr_1}_{H^s}\norm{2b-D_t\bar{\zeta}}_{W^{s-1,\infty}}\\
\leq &C\Big( (\epsilon^{5/2}+\epsilon E_s^{1/2})+\epsilon^{5/2}+E_s^{1/2}\Big)\epsilon\\
\leq & C\epsilon^{5/2}+C\epsilon E_s^{1/2}.
\end{split}
\end{equation}
For $F_{112}$, use proposition \ref{singular},  (\ref{diffb1}), (\ref{corollaryb}), (\ref{Dtzeta}),  it's easy to obtain the estimate
\begin{equation}
\norm{F_{112}}_{H^s}+\norm{F_{113}}_{H^s}\leq C\epsilon^{5/2}+C\epsilon E_s^{1/2}.
\end{equation}
So we obtain the estimate
\begin{equation}\label{F11}
\norm{F_{11}}_{H^s}\leq C\epsilon^{5/2}+C\epsilon E_s^{1/2}.
\end{equation}
Estimates for $F_{12}, F_{13}$ are similar to that of $F_{11}$, we obtain
\begin{equation}
\norm{F_{12}}_{H^s}+\norm{F_{13}}_{H^s}\leq C\epsilon^{5/2}+C\epsilon E_s^{1/2}.
\end{equation}
So we have 
\begin{equation}
\norm{F_{1}}_{H^s}\leq C\epsilon^{5/2}+C\epsilon E_s^{1/2}.
\end{equation}

\subsubsection{Estimate $\norm{F_2}_{H^s}$.}  We rewrite $F_2$ as
\begin{equation}
\begin{split}
 F_2=&-[D_t^2\zeta-\tilde{D}_t^2\tilde{\zeta},\mathcal{H}_{\zeta}\frac{1}{\zeta_{\alpha}}](\bar{\zeta}_{\alpha}-1)-[\tilde{D}_t^2\tilde{\zeta}, \mathcal{H}_{\zeta}\frac{1}{\zeta_{\alpha}}-\mathcal{H}_{\tilde{\zeta}}\frac{1}{\tilde{\zeta}_{\alpha}}](\bar{\zeta}_{\alpha}-1)-[\tilde{D}_t^2\tilde{\zeta},\mathcal{H}_{\tilde{\zeta}}\frac{1}{\tilde{\zeta}_{\alpha}}]\partial_{\alpha}(\bar{\zeta}-\bar{\tilde{\zeta}})\\
&+[(D_t^0)^2\omega-(\tilde{D}_t^0)^2\tilde{\omega},\mathcal{H}_{\omega}\frac{1}{\omega_{\alpha}}](\bar{\omega}_{\alpha}-1)+[(\tilde{D}_t^0)^2\tilde{\omega},\mathcal{H}_{\omega}\frac{1}{\omega_{\alpha}}-\mathcal{H}_{\tilde{\omega}}\frac{1}{\tilde{\omega}}](\bar{\omega}_{\alpha}-1)\\
&+[(\tilde{D}_t^0)^2\tilde{\omega}, \mathcal{H}_{\tilde{\omega}}\frac{1}{\tilde{\omega}_{\alpha}}](\bar{\omega}_{\alpha}-\bar{\tilde{\omega}}_{\alpha})\\
:=& \sum_{m=1}^3 F_{2m},
\end{split}
\end{equation}
where

\begin{equation}
F_{21}=-[D_t^2\zeta-\tilde{D}_t^2\tilde{\zeta},\mathcal{H}_{\zeta}\frac{1}{\zeta_{\alpha}}](\bar{\zeta}_{\alpha}-1)+[(D_t^0)^2\tilde{\omega}-(\tilde{D}_t^0)^2\omega,\mathcal{H}_{\omega}\frac{1}{\omega_{\alpha}}](\bar{\omega}_{\alpha}-1).
\end{equation}

\begin{equation}
F_{22}=-[\tilde{D}_t^2\tilde{\zeta}, \mathcal{H}_{\zeta}\frac{1}{\zeta_{\alpha}}-\mathcal{H}_{\tilde{\zeta}}\frac{1}{\tilde{\zeta}_{\alpha}}](\bar{\zeta}_{\alpha}-1)+[(\tilde{D}_t^0)^2\tilde{\omega},\mathcal{H}_{\omega}\frac{1}{\omega_{\alpha}}-\mathcal{H}_{\tilde{\omega}}\frac{1}{\tilde{\omega}}](\bar{\omega}_{\alpha}-1)
\end{equation}

\begin{equation}
F_{23}=-[\tilde{D}_t^2\tilde{\zeta},\mathcal{H}_{\tilde{\zeta}}\frac{1}{\tilde{\zeta}_{\alpha}}]\partial_{\alpha}(\bar{\zeta}-\bar{\tilde{\zeta}})+[(\tilde{D}_t^0)^2\tilde{\omega}, \mathcal{H}_{\tilde{\omega}}\frac{1}{\tilde{\omega}_{\alpha}}](\bar{\omega}_{\alpha}-\bar{\tilde{\omega}}_{\alpha}).
\end{equation}

\noindent The estimates for $F_{21}, F_{22}, F_{23}$ are similar, so we give the details of estimates of $F_{21}$ only. We rewrite $F_{21}$ as 
\begin{align*}
F_{21}=& -[D_t^2\zeta-\tilde{D}_t^2\tilde{\zeta}-(D_t^0)^2\omega+(\tilde{D}_t^0)^2\tilde{\omega},\mathcal{H}_{\zeta}\frac{1}{\zeta_{\alpha}}](\bar{\zeta}_{\alpha}-1)\\
&-[(D_t^0)^2\omega-(\tilde{D}_t^0)^2\tilde{\omega}, \mathcal{H}_{\zeta}\frac{1}{\zeta_{\alpha}}-\mathcal{H}_{\omega}\frac{1}{\omega_{\alpha}}](\bar{\zeta}_{\alpha}-1)\\
&-[(D_t^0)^2\omega-(\tilde{D}_t^0)^2\omega, \mathcal{H}_{\omega}\frac{1}{\omega_{\alpha}}](\bar{\zeta}_{\alpha}-\bar{\omega}_{\alpha})\\
:=& \sum_{m=1}^3 F_{21m}.
\end{align*}
To estimate $F_{211}$,  we rewrite $D_t^2\zeta-\tilde{D}_t^2\tilde{\zeta}-(D_t^0)^2\omega+(\tilde{D}_t^0)^2\tilde{\omega}$ as 
\begin{equation}\label{seconddiff}
\begin{split}
&D_t^2\zeta-\tilde{D}_t^2\tilde{\zeta}-(D_t^0)^2\omega+(\tilde{D}_t^0)^2\tilde{\omega}\\
=& D_t^2(\xi_1+\omega)-\tilde{D}_t^2(\tilde{\xi}_1+\tilde{\omega})-(D_t^0)^2\omega+(\tilde{D}_t^0)^2\tilde{\omega}\\
=& D_t^2\xi_1+(D_t^2-(D_t^0)^2)\omega-\tilde{D}_t^2\tilde{\xi}_1-((D_t^0)^2-(\tilde{D}_t^0)^2)\tilde{\omega}\\
=& D_t^2r_1+(D_t^2-\tilde{D}_t^2)\tilde{\xi}_1+(D_t^2-(D_t^0)^2)\omega-((D_t^0)^2-(\tilde{D}_t^0)^2)\tilde{\omega}\\
\end{split}
\end{equation}

\begin{lemma}\label{lalala}
Assume the a priori assumption (\ref{boot}). We have 
\begin{equation}
\begin{split}
&\norm{(D_t^2-\tilde{D}_t^2)\tilde{\xi}_1}_{H^s}+\norm{(D_t^2-(D_t^0)^2)\omega}_{H^s}+\norm{((D_t^0)^2-(\tilde{D}_t^0)^2)\tilde{\omega}}_{H^s}\\
\leq  & C\epsilon^{5/2}+C\epsilon \norm{D_t(b_1-\tilde{b}_1)}_{H^s}.
\end{split}
\end{equation}
where 
$$C=C(\norm{D_t^0\omega}_{H^{s'}(\mathbb{T})}, \norm{\tilde{D}\tilde{\xi}_1}_{H^{s}(\RR)})$$
for $s'>s+\frac{3}{2}$.

\end{lemma}
\begin{proof}
We have
\begin{align*}
&D_t^2-(D_t^0)^2
= D_t(D_t-D_t^0)+(D_t-D_t^0)D_t^0\\
=&D_t(b-b_0)\partial_{\alpha}+(b-b_0)\partial_{\alpha}D_t^0\\
=&(D_tb_1)\partial_{\alpha}+b_1D_t\partial_{\alpha}+b_1\partial_{\alpha}D_t^0.
\end{align*}
 So we have 
\begin{align*}
&\norm{(D_t^2-(D_t^0)^2)\omega}_{H^s}\\
\leq & \norm{(D_tb_1)\partial_{\alpha}\omega}_{H^s}+\norm{b_1D_t\partial_{\alpha}\omega}_{H^s}+\norm{b_1\partial_{\alpha}D_t^0\omega}_{H^s}\\
\leq & \norm{D_tb_1}_{H^s}\norm{\omega_{\alpha}}_{W^{s,\infty}}+\norm{b_1}_{H^s}\norm{\omega_{\alpha}}_{W^{s,\infty}}+\norm{b_1}_{H^s}\norm{\partial_{\alpha}D_t^0\omega}_{W^{s,\infty}}\\
\leq & \norm{D_t(b_1-\tilde{b}_1)}_{H^s}\norm{\omega_{\alpha}}_{W^{s,\infty}}+\norm{D_t \tilde{b}_1}_{H^s}\norm{\omega_{\alpha}}_{W^{s,\infty}}+\norm{b_1}_{H^s}\norm{\omega_{\alpha}}_{W^{s,\infty}}+\norm{b_1}_{H^s}\norm{\partial_{\alpha}D_t^0\omega}_{W^{s,\infty}}\\
\leq & C\epsilon^{5/2}+C\epsilon \norm{D_t(b_1-\tilde{b}_1)}_{H^s}.
\end{align*}
We decompose $D_t^2-(D_t^0)^2$ and $(D_t^0)^2-(\tilde{D}_t^0)^2$ in a similar way. With these decompositions,  the lemma follows easily.
\end{proof}

By (\ref{seconddiff}), lemma \ref{lalala}, and proposition \ref{singular}, we have 
\begin{equation}
\norm{F_{211}}_{H^s}\leq C\epsilon(\epsilon^{5/2}+\epsilon \norm{D_t(b_1-\tilde{b}_1)}_{H^s})\leq C\epsilon^{7/2}+C\epsilon^2\norm{D_t(b_1-\tilde{b}_1)}_{H^s}.
\end{equation}
The estimates for $F_{212}, F_{213}$ are the same and we obtain
\begin{equation}
\norm{F_{212}}_{H^s}+\norm{F_{213}}_{H^s}\leq C\epsilon^{5/2}+C\epsilon^2\norm{D_t(b_1-\tilde{b}_1)}_{H^s}.
\end{equation}
So we obtain
\begin{equation}
\norm{F_{21}}_{H^s}\leq C\epsilon^{5/2}+C\epsilon^2\norm{D_t(b_1-\tilde{b}_1)}_{H^s}.
\end{equation}
Similarly, we have
\begin{equation}
\norm{F_{22}}_{H^s}\leq C\epsilon^{5/2}+C\epsilon E_s^{1/2}+C\epsilon^2\norm{D_t(b_1-\tilde{b}_1)}_{H^s}.
\end{equation}

\begin{equation}
\norm{F_{23}}_{H^s}\leq C\epsilon^{5/2}+C\epsilon E_s^{1/2}+C\epsilon^2\norm{D_t(b_1-\tilde{b}_1)}_{H^s}.
\end{equation}

\noindent So we obtain
\begin{equation}
\norm{F_{2}}_{H^s}\leq C\epsilon^{5/2}+C\epsilon E_s^{1/2}+C\epsilon^2\norm{D_t(b_1-\tilde{b}_1)}_{H^s}.
\end{equation}

\vspace*{2ex}

\noindent Similar to the estimates for $F_1$ and $F_2$, we obtain
\begin{equation}
\norm{F_3}_{H^s}+\norm{F_4}_{H^s}+\norm{F_5}_{H^s}+\norm{F_6}_{H^s}\leq C\epsilon^{5/2}+C\epsilon E_s^{1/2}+C\epsilon^2\norm{D_t(b_1-\tilde{b}_1)}_{H^s}.
\end{equation}

So we obtain

\begin{equation}
\norm{D_t(b_1-\tilde{b}_1)}_{H^s}\leq C\epsilon^{5/2}+C\epsilon E_s^{1/2}+C\epsilon^2\norm{D_t(b_1-\tilde{b}_1)}_{H^s}.
\end{equation}
Therefore,
\begin{equation}\label{Dtb1diff}
\norm{D_t(b_1-\tilde{b}_1)}_{H^s}\leq C\epsilon^{5/2}+C\epsilon E_s^{1/2}\leq C\epsilon^2.
\end{equation}

\subsection{Bound $A_1-\tilde{A}_1$} Indeed, recall that $\tilde{A}_1=0$. We show that $\norm{A_1-\tilde{A}_1}_{H^s}$ can be bounded by $C\epsilon^{5/2}+C\epsilon E_s^{1/2}$.  Since $\tilde{A}_1$ satisfies the formula (\ref{a1a1}) for $A_1$ up to $O(\epsilon^3)$, we have 
\begin{equation}\label{tildea1}
\begin{split}
(I-\mathcal{H}_{\tilde{\zeta}})(\tilde{A}-\tilde{A}_0)=& i[\tilde{D}_t^2\tilde{\zeta},\mathcal{H}_{\tilde{\zeta}}]\frac{\bar{\tilde{\zeta}}_{\alpha}-1}{\tilde{\zeta}_{\alpha}}-i[(\tilde{D}_t^0)^2\tilde{\omega},\mathcal{H}_{\tilde{\omega}}]\frac{\bar{\tilde{\omega}}_{\alpha}-1}{\tilde{\omega}_{\alpha}}\\
&+i[\tilde{D}_t\tilde{\zeta},\mathcal{H}_{\tilde{\zeta}}]\frac{\partial_{\alpha}\tilde{D}_t\bar{\tilde{\zeta}}}{\tilde{\zeta}_{\alpha}}-i[\tilde{D}_t^0\tilde{\omega},\mathcal{H}_{\tilde{\omega}}]\frac{\partial_{\alpha}\tilde{D}_t^0 \bar{\tilde{\omega}}}{\tilde{\omega}_{\alpha}}  \\
&+(\mathcal{H}_{\tilde{\omega}}-\mathcal{H}_{\tilde{\zeta}})(\tilde{A}_0-1) +\epsilon^3\mathcal{R}
\end{split}
\end{equation}
Subtract (\ref{tildea1}) from (\ref{a1a1}), we obtain the following formula for $A_1-\tilde{A}_1$. We group similar terms together and write it in the following manner: 
\begin{align*}
&(I-\mathcal{H}_{\zeta})(A_1-\tilde{A}_1)
= (I-\mathcal{H}_{\zeta})A_1-(I-\mathcal{H}_{\tilde{\zeta}})\tilde{A}_1+(\mathcal{H}_{\zeta}-\mathcal{H}_{\tilde{\zeta}})\tilde{A}_1\\
=& \Big\{ i[D_t^2\zeta,\mathcal{H}_{\zeta}]\frac{\bar{\zeta}_{\alpha}-1}{\zeta_{\alpha}}- i[\tilde{D}_t^2\tilde{\zeta},\mathcal{H}_{\tilde{\zeta}}]\frac{\bar{\tilde{\zeta}}_{\alpha}-1}{\tilde{\zeta}_{\alpha}} -i[(D_t^0)^2\omega,\mathcal{H}_{\omega}]\frac{\bar{\omega}_{\alpha}-1}{\omega_{\alpha}}+i[(\tilde{D}_t^0)^2\tilde{\omega},\mathcal{H}_{\tilde{\omega}}]\frac{\bar{\tilde{\omega}}_{\alpha}-1}{\tilde{\omega}_{\alpha}}\Big\} \\
&+\Big\{i[D_t\zeta,\mathcal{H}_{\zeta}]\frac{\partial_{\alpha}D_t\bar{\zeta}}{\zeta_{\alpha}}-i[\tilde{D}_t\tilde{\zeta},\mathcal{H}_{\tilde{\zeta}}]\frac{\partial_{\alpha}\tilde{D}_t\bar{\tilde{\zeta}}}{\tilde{\zeta}_{\alpha}}-i[D_t^0\omega,\mathcal{H}_{\omega}]\frac{\partial_{\alpha}D_t^0 \bar{\omega}}{\omega_{\alpha}} +i[\tilde{D}_t^0\tilde{\omega},\mathcal{H}_{\tilde{\omega}}]\frac{\partial_{\alpha}\tilde{D}_t^0 \bar{\tilde{\omega}}}{\tilde{\omega}_{\alpha}}\Big\}\\
&+\Big\{(\mathcal{H}_{\omega}-\mathcal{H}_{\zeta})(A_0-1)-(\mathcal{H}_{\tilde{\omega}}-\mathcal{H}_{\tilde{\zeta}})(\tilde{A}_0-1) \Big\}\\
&+(\mathcal{H}_{\zeta}-\mathcal{H}_{\tilde{\zeta}})\tilde{A}_1\\
&+\epsilon^4 \mathcal{R}\\
:=& \sum_{m=1}^5 K_m.
\end{align*}
Note that we can estimate $K_1$ in exactly the same way as we did for the quantity $F_{211}$, and we obtain estimate
\begin{equation}
\norm{K_1}_{H^s}\leq C\epsilon^{5/2}+C\epsilon E_s^{1/2}.
\end{equation}
$K_2$ can be estimated the same way as we did for the quantity $F_1$, and we obtain
\begin{equation}
\norm{K_2}_{H^s}\leq C\epsilon^{5/2}+C\epsilon E_s^{1/2}.
\end{equation}
Estimates for $K_3$, $K_4$ and $K_5$ are straight forward , we have
\begin{equation}
\norm{K_3}_{H^s}+\norm{K_4}_{H^s}+\norm{K_5}_{H^s}\leq C\epsilon^{5/2}+C\epsilon E_s^{1/2}.
\end{equation}
So we obtain
\begin{equation}
\norm{(I-\mathcal{H}_{\zeta})(A_1-\tilde{A}_1)}_{H^s}\leq  C\epsilon^{5/2}+C\epsilon E_s^{1/2}.
\end{equation}
Therefore,
\begin{equation}\label{A1Diff}
\norm{A_1-\tilde{A}_1}_{H^s}\leq  C\epsilon^{5/2}+C\epsilon E_s^{1/2}.
\end{equation}
\begin{corollary}\label{A1estimate}
We have 
\begin{equation}
    \norm{A_1}_{H^s}\leq C\epsilon^{3/2}.
\end{equation}
\end{corollary}

\begin{corollary}\label{boundAAA}
Assume the a priori assumption (\ref{boot}), then 
\begin{equation}\label{taylorsign}
    \inf_{t\in [0,T_0]}\inf_{\alpha\in \mathbb{R}}A(\alpha,t)\geq \frac{1}{2},\quad \quad \sup_{t\in [0,T_0]}\sup_{\alpha\in \mathbb{R}}A(\alpha,t)\leq 2.
\end{equation}
\end{corollary}
\begin{proof}
We have $A=A_0+A_1$. By Theorem \ref{longperiodic}, $\norm{A_0-1}_{\infty}\leq C\epsilon^2$. By Corollary \ref{A1estimate}, $\norm{A_1}_{\infty}\leq C\epsilon^{3/2}$. Therefore, for $\epsilon$ sufficiently small, we have (\ref{taylorsign}).
\end{proof}

\begin{definition}
Denote $L$ by the quantity 
\begin{equation}
\begin{split}
L=&[\mathcal{P}, D_t](I-\mathcal{H}_{\zeta})(\zeta-\alpha)-[\mathcal{P}_0, D_t^0](I-\mathcal{H}_{\omega})(\omega-\alpha)\\&-[\tilde{\mathcal{P}},\tilde{D}_t](I-\mathcal{H}_{\tilde{\zeta}})(\tilde{\zeta}-\alpha)+[\tilde{\mathcal{P}}_0, \tilde{D}_t^0](I-\mathcal{H}_{\tilde{\omega}})(\tilde{\omega}-\alpha)
\end{split}
\end{equation}
\end{definition}
This quantity arises in the energy estimates in the next section, so we need also to bound it in terms of $E_s$ and $\epsilon$. 
\subsection{Bound $L$}
We know that:
\begin{equation}
\begin{split}
[\mathcal{P}, D_t](I-\mathcal{H}_{\zeta})(\zeta-\alpha)=\Big(\frac{a_t}{a}\Big)\circ \kappa^{-1}iA\partial_{\alpha}(I-\mathcal{H}_{\zeta})(\zeta-\alpha),
\end{split}
\end{equation}
where
 $\Big(\frac{a_t}{a}\Big)\circ \kappa^{-1}$ is given by
\begin{equation}\label{ata}
\begin{split}
(I-\mathcal{H}_{\zeta})\Big(A\bar{\zeta}_{\alpha}\Big(\frac{a_t}{a}\Big)\circ \kappa^{-1}\Big)=& 2i[D_t^2\zeta, \mathcal{H}]\frac{\partial_{\alpha}D_t\bar{\zeta}}{\zeta_{\alpha}}+2i[D_t\zeta, \mathcal{H}]\frac{\partial_{\alpha}D_t^2\bar{\zeta}}{\zeta_{\alpha}}\\
&-\frac{1}{\pi} \int \Big( \frac{D_t\zeta(\alpha)-D_t\zeta(\beta)}{\zeta(\alpha)-\zeta(\beta)}\Big)^2 \partial_{\beta} D_t\bar{\zeta}(\beta)d\beta.
\end{split}
\end{equation}
Similarly, we have 
\begin{equation}
\begin{split}
[\mathcal{P}_0, D_t^0](I-\mathcal{H}_{\omega})(\omega-\alpha)=\Big(\frac{(a_0)_t}{a_0}\Big)\circ \kappa_0^{-1}iA_0\partial_{\alpha}(I-\mathcal{H}_{\omega})(\omega-\alpha),
\end{split}
\end{equation}
where $\Big(\frac{(a_0)_t}{a_0}\Big)\circ \kappa_0^{-1}$ is given by
\begin{equation}\label{omegaat}
\begin{split}
(I-\mathcal{H}_{\omega})\Big(A_0\bar{\omega}_{\alpha}\Big(\frac{(a_0)_t}{a_0}\Big)\circ \kappa_0^{-1}\Big)=& 2i[(D_t^0)^2\omega, \mathcal{H}]\frac{\partial_{\alpha}D_t^0\bar{\omega}}{\omega_{\alpha}}+2i[D_t^0\omega, \mathcal{H}]\frac{\partial_{\alpha}D_t^2\bar{\zeta}}{\zeta_{\alpha}}\\
&-\frac{1}{\pi} \int \Big( \frac{D_t^0\omega(\alpha)-D_t^0\omega(\beta)}{\omega(\alpha)-\omega(\beta)}\Big)^2 \partial_{\beta} D_t^0\bar{\omega}(\beta)d\beta.
\end{split}
\end{equation}
For brevity, denote
$$\psi=\Big(\frac{a_t}{a}\Big)\circ \kappa^{-1}, \quad \psi_0=\Big(\frac{(a_0)_t}{a_0}\Big)\circ \kappa_0^{-1}.$$
Let $\tilde{\psi}$ be the approximation of $\psi$ to the order $O(\epsilon^4)$, and $\tilde{\psi}_0$ the approximation of $\psi_0$ to the order $O(\epsilon^4)$. 
We have formula for $(I-\mathcal{H}_{\tilde{\zeta}})\tilde{\psi}$ and $(I-\mathcal{H}_{\tilde{\omega}})\tilde{\psi}_0$:
\begin{equation}\label{tildeata}
\begin{split}
(I-\mathcal{H}_{\tilde{\zeta}})\Big(\tilde{A}\bar{\tilde{\zeta}}_{\alpha}\tilde{\psi}\Big)=& 2i[\tilde{D}_t^2\tilde{\zeta}, \mathcal{H}]\frac{\partial_{\alpha}\tilde{D}_t\bar{\tilde{\zeta}}}{\tilde{\zeta}_{\alpha}}+2i[\tilde{D}_t\tilde{\zeta}, \mathcal{H}]\frac{\partial_{\alpha}\tilde{D}_t^2\bar{\tilde{\zeta}}}{\tilde{\zeta}_{\alpha}}\\
&-\frac{1}{\pi} \int \Big( \frac{\tilde{D}_t\tilde{\zeta}(\alpha)-\tilde{D}_t\tilde{\zeta}(\beta)}{\tilde{\zeta}(\alpha)-\tilde{\zeta}(\beta)}\Big)^2 \partial_{\beta} \tilde{D}_t\bar{\tilde{\zeta}}(\beta)d\beta+\epsilon^4\mathcal{R}_1,
\end{split}
\end{equation}
and
\begin{equation}
\begin{split}
(I-\mathcal{H}_{\tilde{\omega}})\Big(\tilde{A}_0\bar{\tilde{\omega}}_{\alpha}\tilde{\psi}_0\Big)=& 2i[(\tilde{D}_t^0)^2\tilde{\omega}, \mathcal{H}]\frac{\partial_{\alpha}\tilde{D}_t^0\bar{\tilde{\omega}}}{\tilde{\omega}_{\alpha}}+2i[\tilde{D}_t^0\tilde{\omega}, \mathcal{H}]\frac{\partial_{\alpha}(\tilde{D}_t^0)^2\bar{\tilde{\omega}}}{\tilde{\zeta}_{\alpha}}\\
&-\frac{1}{\pi} \int \Big( \frac{\tilde{D}_t^0\tilde{\omega}(\alpha)-\tilde{D}_t^0\tilde{\omega}(\beta)}{\tilde{\omega}(\alpha)-\tilde{\omega}(\beta)}\Big)^2 \partial_{\beta} \tilde{D}_t^0\bar{\tilde{\omega}}(\beta)d\beta+\epsilon^4\mathcal{R}_2,
\end{split}
\end{equation}
where
$$\norm{\epsilon^4\mathcal{R}_1-\epsilon^4\mathcal{R}_2}_{H^s}\leq \epsilon^{7/2}.$$
Denote 
$$\theta=(I-\mathcal{H}_{\zeta})(\zeta-\alpha), \quad \theta_0=(I-\mathcal{H}_{\omega})(\omega-\alpha),\quad \tilde{\theta}=(I-\mathcal{H}_{\tilde{\zeta}})(\tilde{\zeta}-\alpha),\quad \tilde{\theta}_0=(I-\mathcal{H}_{\tilde{\omega}})(\tilde{\omega}-\alpha).$$
Then  
$$L=\psi \partial_{\alpha}\theta-\tilde{\psi}\partial_{\alpha}\tilde{\theta}-\psi_0\partial_{\alpha}\theta_0+\tilde{\psi}_0\partial_{\alpha}\tilde{\theta}_0.$$
We rewrite $L$ in the following form: 
\begin{align*}
L=& (\psi-\psi_0)\partial_{\alpha}\theta+\psi_0\partial_{\alpha}(\theta-\theta_0)-(\tilde{\psi}-\tilde{\psi}_0)\partial_{\alpha}\tilde{\theta}-\tilde{\psi}_0\partial_{\alpha}(\tilde{\theta}-\tilde{\theta}_0)\\
=&(\psi-\psi_0)\partial_{\alpha}(\theta-\tilde{\theta})+(\psi-\psi_0-(\tilde{\psi}-\tilde{\psi}_0))\partial_{\alpha}\tilde{\theta}\\
&+(\psi_0-\tilde{\psi}_0)\partial_{\alpha}(\theta-\theta_0)+\tilde{\psi}_0\partial_{\alpha}(\theta-\theta_0-(\tilde{\theta}-\tilde{\theta}_0))\\
:=&L_1+L_2+L_3+L_4.
\end{align*}
The advantage of writing $L$ in this form is that, each $L_i$ can be written in the form $L_i=yz$, where $y\in H^s$, and $z=z_1+z_2$, where $z_1\in H^s$, and $z_2\in W^{s,\infty}$. Note that we cannot estimate $z$ directly in $W^{s,\infty}$, because $z_1$ might lose one derivative.

\subsubsection{Estimate $L_1$.} First we estimate $\theta-\tilde{\theta}$. We have 
\begin{align*}
\theta-\tilde{\theta}=& (I-\mathcal{H}_{\zeta})r+(\mathcal{H}_{\tilde{\zeta}}-\mathcal{H}_{\zeta})(\tilde{\zeta}-\alpha).
\end{align*}
Let $\mathcal{H}_{\zeta}^{\ast}$ be the adjoint of $\mathcal{H}_{\zeta}$, i.e.,
$$\mathcal{H}_{\zeta}^{\ast}f=-\zeta_{\alpha}\mathcal{H}\frac{f}{\zeta}_{\alpha}.$$ Then 
\begin{align*}
\partial_{\alpha}(\theta-\tilde{\theta})=&(I-\mathcal{H}_{\zeta}^{\ast})r_{\alpha}+(\mathcal{H}_{\tilde{\zeta}}^{\ast}-\mathcal{H}_{\zeta}^{\ast})(\tilde{\zeta}_{\alpha}-1)\\
=& (I-\mathcal{H}_{\zeta}^{\ast})(r_1)_{\alpha}+(I-\mathcal{H}_{\zeta}^{\ast})(r_0)_{\alpha}\\
&+\frac{1}{\pi i}\int \frac{\zeta_{\alpha}(\tilde{\zeta}(\alpha)-\tilde{\zeta}(\beta))-\tilde{\zeta}_{\alpha}(\zeta(\alpha)-\zeta(\beta))}{(\zeta(\alpha)-\zeta(\beta))(\tilde{\zeta}(\alpha)-\tilde{\zeta}(\beta))}(\tilde{\zeta}_{\beta}-1)d\beta\\
=& (I-\mathcal{H}_{\zeta}^{\ast})(r_1)_{\alpha}+(I-\mathcal{H}_{\zeta}^{\ast})(r_0)_{\alpha}\quad &:=Z_1+Z_2\\
&+\frac{1}{\pi i}\int \frac{(r_1)_{\alpha}(\tilde{\zeta}(\alpha)-\tilde{\zeta}(\beta))-\tilde{\zeta}_{\alpha}((r_1)_{\alpha}-(r_1)_{\beta})}{(\zeta(\alpha)-\zeta(\beta))(\tilde{\zeta}(\alpha)-\tilde{\zeta}(\beta))}(\tilde{\zeta}_{\beta}-1)d\beta  &:=Z_3\\
&+\frac{1}{\pi i}\int \frac{(r_0)_{\alpha}(\tilde{\zeta}(\alpha)-\tilde{\zeta}(\beta))-\tilde{\zeta}_{\alpha}((r_0)_{\alpha}-(r_0)_{\beta})}{(\zeta(\alpha)-\zeta(\beta))(\tilde{\zeta}(\alpha)-\tilde{\zeta}(\beta))}(\tilde{\zeta}_{\beta}-1)d\beta  & :=Z_4\\
:=& Z_1+Z_2+Z_3+Z_4.
\end{align*}
We have $Z_1:=(I-\mathcal{H}_{\zeta}^{\ast})(r_1)_{\alpha}$, so 
\begin{align*}
\norm{Z_1}_{H^s}\leq CE_s^{1/2}.
\end{align*}
For $Z_3$, use proposition \ref{singular},  we obtain
\begin{align*}
\norm{Z_3}_{H^s}\leq C\epsilon E_s^{1/2}.
\end{align*}
Use the fact that $||\partial_{\alpha}r_0||_{W^{s,\infty}}\leq C\epsilon^2$, it's straightforward to prove that 
$$\norm{Z_2+Z_4}_{W^{s,\infty}}\leq C\epsilon^2.$$

\vspace*{2ex}

\noindent         Next we estiamte $\psi-\psi_0$, we consider $(I-\mathcal{H}_{\zeta})(\psi-\psi_0)A\bar{\zeta}_{\alpha}$.  Note that
\begin{align*}
&(I-\mathcal{H}_{\zeta})A\bar{\zeta}_{\alpha}\psi-(I-\mathcal{H}_{\omega})A_0\bar{\omega}_{\alpha}\psi_0\\
=& (I-\mathcal{H}_{\zeta})A\bar{\zeta}_{\alpha}\psi-(I-\mathcal{H}_{\zeta})A\bar{\zeta}_{\alpha}\psi_0+(I-\mathcal{H}_{\zeta})A\bar{\zeta}_{\alpha}\psi_0
-(I-\mathcal{H}_{\zeta})A\bar{\omega}_{\alpha}\psi_0\\
& +(I-\mathcal{H}_{\zeta})A\bar{\omega}_{\alpha}\psi_0-(I-\mathcal{H}_{\zeta})A_0\bar{\tilde{\omega}}_{\alpha}\psi_0+
(I-\mathcal{H}_{\zeta})A_0\bar{\tilde{\omega}}\psi-(I-\mathcal{H}_{\omega})A_0\bar{\tilde{\omega}}\psi_0\\
=& (I-\mathcal{H}_{\zeta})A\bar{\zeta}_{\alpha}(\psi-\psi_0)+(I-\mathcal{H}_{\zeta})A(\bar{\xi}_1)_{\alpha}\tilde{\psi}\\
&+(I-\mathcal{H}_{\zeta})A_1\bar{\omega}_{\alpha}\psi+(\mathcal{H}_{\omega}-\mathcal{H}_{\zeta})A_0\bar{\omega}_{\alpha}\psi_0.
\end{align*}
So we have 
\begin{align*}
(I-&\mathcal{H}_{\zeta})A\bar{\zeta}_{\alpha}(\psi-\psi_0)=(I-\mathcal{H}_{\zeta})A\bar{\zeta}_{\alpha}\psi-(I-\mathcal{H}_{\omega})A_0\bar{\omega}_{\alpha}\psi_0\\
&-\Big\{(I-\mathcal{H}_{\zeta})A(\bar{\xi}_1)_{\alpha}\tilde{\psi}+(I-\mathcal{H}_{\zeta})A_1\bar{\omega}_{\alpha}\psi+(\mathcal{H}_{\omega}-\mathcal{H}_{\zeta})A_0\bar{\omega}_{\alpha}\psi_0\Big\}
\end{align*}
For $(I-\mathcal{H}_{\zeta})A\bar{\zeta}_{\alpha}\psi-(I-\mathcal{H}_{\omega})A_0\bar{\omega}_{\alpha}\psi_0$, subtract  (\ref{omegaat}) from (\ref{ata}),  and then group similar terms. We have estimated terms of these kinds before, so we omitt the details. We have
\begin{equation}
\norm{(I-\mathcal{H}_{\zeta})A\bar{\zeta}_{\alpha}\psi-(I-\mathcal{H}_{\omega})\tilde{A}\bar{\tilde{\zeta}}_{\alpha}\tilde{\psi}}_{H^s}\leq C(\epsilon^{5/2}+\epsilon E_s^{1/2}).
\end{equation}
Since $\tilde{\psi}$, $\psi_0$  and $A_1$ are quadratic, it's easy to show that 
\begin{equation}
\norm{(I-\mathcal{H}_{\zeta})A(\bar{\xi}_1)_{\alpha}\tilde{\psi}+(I-\mathcal{H}_{\zeta})A_1\bar{\omega}_{\alpha}\psi+(\mathcal{H}_{\omega}-\mathcal{H}_{\zeta})A_0\bar{\omega}_{\alpha}\psi_0}_{H^s}\leq C\epsilon^{5/2}.
\end{equation}
Combined the above estimates, we obtain
\begin{equation}
\begin{split}
||\psi-\psi_0||_{H^s}\leq & C||(I-\mathcal{H}_{\zeta})A\bar{\zeta}_{\alpha}(\psi-\psi_0)||_{H^s}\\
\leq & C(\epsilon E_s^{1/2}+\epsilon^{5/2}).
\end{split}
\end{equation}
So we have
\begin{equation}
\norm{L_1}_{H^s}\leq C(\epsilon E_s^{1/2}+\epsilon^{5/2})(E_s^{1/2}+\epsilon^2)\leq C(\epsilon^{7/2}+\epsilon E_s).
\end{equation}
The quantities $L_2, L_3$ and $L_4$ can be estimated in similar manner, and obtain 
\begin{equation}
\norm{L_2}_{H^s}+\norm{L_3}_{H^s}+\norm{L_4}_{H^s}\leq C(\epsilon^{7/2}+ \epsilon E_s).
\end{equation}

So we have
\begin{equation}\label{ll}
||L||_{H^s} \leq C(\epsilon^{7/2}+ \epsilon E_s).
\end{equation}

\section{Energy estimates}\label{energyestimatesection}
In Section \ref{govern}, we derive equations governing the evolution of $r_1$ and $D_tr_1$, respectively, and define energy for these quantities. In Section \ref{bounds},  we obtain aproiri bounds for some quantities which will be used in energy estimates. In this section, we obtain bounds for the energy $\mathcal{E}_s$. For this purpose, we estimate the quantity appear in $\frac{d\mathcal{E}_s}{dt}$, and bound $\frac{d\mathcal{E}_s}{dt}$ in terms of $E_s$ and $\epsilon$.  Then we obtain
\begin{equation}\label{ES}
\frac{d\mathcal{E}}{dt}\leq C(E_s^2+\epsilon E_s^{3/2}+\epsilon^2 E_s+\epsilon^{7/2}E_s^{1/2}+\epsilon^5).
\end{equation}
Then we show that $E_s$ is essentially  controlled by $\mathcal{E}$:
\begin{equation}\label{control}
E_s^{1/2}\leq C(\mathcal{E}+\epsilon^{5/2}).
\end{equation}
(\ref{ES}) and (\ref{control}) together give the bound 
\begin{equation}
\mathcal{E}\leq C\epsilon^3
\end{equation}
on time scale $\epsilon^{-2}$.

\subsection{Estimate $(I-\mathcal{H}_{\zeta})\mathcal{R}_{11}$} It suffices to estimate $\mathcal{R}_{11}$. 
Recall that 
\begin{equation}
\begin{split}
\mathcal{R}_{11}=&  -2[D_t\zeta,\mathcal{H}_{\zeta}\frac{1}{\zeta_{\alpha}}+\bar{\mathcal{H}}_{\zeta}\frac{1}{\bar{\zeta}_{\alpha}}]\partial_{\alpha}D_t\zeta+2[D_t^0\omega,\mathcal{H}_{\omega}\frac{1}{\omega_{\alpha}}+\bar{\mathcal{H}}_{\omega}\frac{1}{\bar{\omega}_{\alpha}}]\partial_{\alpha}D_t^0\omega\\
&-2[\tilde{D}_t\tilde{\zeta},\mathcal{H}_{\tilde{\zeta}}\frac{1}{\tilde{\zeta}_{\alpha}}+\bar{\mathcal{H}}_{\tilde{\zeta}}\frac{1}{\bar{\tilde{\zeta}}_{\alpha}}]\partial_{\alpha}\tilde{D}_t\tilde{\zeta}+2[\tilde{D}_t^0\tilde{\omega},\mathcal{H}_{\tilde{\omega}}\frac{1}{\tilde{\omega}_{\alpha}}+\bar{\mathcal{H}}_{\tilde{\omega}}\frac{1}{\bar{\tilde{\omega}}_{\alpha}}]\partial_{\alpha}\tilde{D}_t^0\tilde{\omega}.
\end{split}
\end{equation}
First, we rewrite $I:=-2[D_t\zeta,\mathcal{H}_{\zeta}\frac{1}{\zeta_{\alpha}}+\bar{\mathcal{H}}_{\zeta}\frac{1}{\bar{\zeta}_{\alpha}}]\partial_{\alpha}D_t\zeta+2[D_t^0\omega,\mathcal{H}_{\omega}\frac{1}{\omega_{\alpha}}+\bar{\mathcal{H}}_{\omega}\frac{1}{\bar{\omega}_{\alpha}}]\partial_{\alpha}D_t^0\omega$ as
\begin{equation}
\begin{split}
I=& -2[D_t\zeta-D_t^0\omega, \mathcal{H}_{\zeta}\frac{1}{\zeta_{\alpha}}+\bar{\mathcal{H}}_{\zeta}\frac{1}{\bar{\zeta}_{\alpha}}]\partial_{\alpha}D_t\zeta\\
&-2[D_t^0\omega, \mathcal{H}_{\zeta}\frac{1}{\zeta_{\alpha}}+\bar{\mathcal{H}}_{\zeta}\frac{1}{\bar{\zeta}}-\mathcal{H}_{\omega}\frac{1}{\omega_{\alpha}}-\bar{\mathcal{H}}_{\omega}\frac{1}{\bar{\omega}_{\alpha}}]\partial_{\alpha}D_t\zeta\\
&-2[D_t^0\omega, \mathcal{H}_{\omega}\frac{1}{\omega_{\alpha}}+\bar{\mathcal{H}}_{\omega}\frac{1}{\bar{\omega}_{\alpha}}]\partial_{\alpha}D_t^0\xi_1\\
:=& I_1+I_2+I_3.
\end{split}
\end{equation}
And we rewrite $\it{II}:=-2[\tilde{D}_t\tilde{\zeta},\mathcal{H}_{\tilde{\zeta}}\frac{1}{\tilde{\zeta}_{\alpha}}+\bar{\mathcal{H}}_{\tilde{\zeta}}\frac{1}{\bar{\tilde{\zeta}}_{\alpha}}]\partial_{\alpha}\tilde{D}_t\tilde{\zeta}+2[\tilde{D}_t^0\tilde{\omega},\mathcal{H}_{\tilde{\omega}}\frac{1}{\tilde{\omega}_{\alpha}}+\bar{\mathcal{H}}_{\tilde{\omega}}\frac{1}{\bar{\tilde{\omega}}_{\alpha}}]\partial_{\alpha}\tilde{D}_t^0\tilde{\omega}$ as

\begin{equation}
\begin{split}
\it{II}=& -2[\tilde{D}_t\tilde{\zeta}-\tilde{D}_t^0\tilde{\omega}, \mathcal{H}_{\tilde{\zeta}}\frac{1}{\tilde{\zeta}_{\alpha}}+\bar{\mathcal{H}}_{\tilde{\zeta}}\frac{1}{\bar{\tilde{\zeta}}_{\alpha}}]\partial_{\alpha}\tilde{D}_t\tilde{\zeta}\\
&-2[\tilde{D}_t^0\tilde{\omega}, \mathcal{H}_{\tilde{\zeta}}\frac{1}{\tilde{\zeta}_{\alpha}}+\bar{\mathcal{H}}_{\tilde{\zeta}}\frac{1}{\bar{\tilde{\zeta}}}-\mathcal{H}_{\tilde{\omega}}\frac{1}{\tilde{\omega}_{\alpha}}-\bar{\mathcal{H}}_{\tilde{\omega}}\frac{1}{\bar{\tilde{\omega}}_{\alpha}}]\partial_{\alpha}\tilde{D}_t\tilde{\zeta}\\
&-2[\tilde{D}_t^0\tilde{\omega}, \mathcal{H}_{\tilde{\omega}}\frac{1}{\tilde{\omega}_{\alpha}}+\bar{\mathcal{H}}_{\tilde{\omega}}\frac{1}{\bar{\tilde{\omega}}_{\alpha}}]\partial_{\alpha}\tilde{D}_t^0\tilde{\xi}_1\\
:=& \it{II}_1+\it{II}_2+\it{II}_3.
\end{split}
\end{equation}
\subsubsection{Estimate $\norm{I_1+\it{II}_1}_{H^s}$} We have 
\begin{align*}
I_1+\it{II}_1=& -2[D_t\zeta-D_t^0\omega-(\tilde{D}_t\tilde{\zeta}-\tilde{D}_t^0\tilde{\omega}), \mathcal{H}_{\zeta}\frac{1}{\zeta_{\alpha}}+\bar{\mathcal{H}}_{\zeta}\frac{1}{\bar{\zeta}_{\alpha}}]\partial_{\alpha}D_t\zeta\\
&-2[\tilde{D}_t\tilde{\zeta}-\tilde{D}_t^0\tilde{\omega}, \mathcal{H}_{\zeta}\frac{1}{\zeta_{\alpha}}+\bar{\mathcal{H}}_{\zeta}\frac{1}{\bar{\zeta}_{\alpha}}-\Big(\mathcal{H}_{\tilde{\zeta}}\frac{1}{\tilde{\zeta}_{\alpha}}+\bar{\mathcal{H}}_{\tilde{\zeta}}\frac{1}{\bar{\tilde{\zeta}}_{\alpha}}\Big)]\partial_{\alpha}D_t\zeta\\
&-2[\tilde{D}_t\tilde{\zeta}-\tilde{D}_t^0\tilde{\omega},\mathcal{H}_{\tilde{\zeta}}\frac{1}{\tilde{\zeta}_{\alpha}}+\bar{\mathcal{H}}_{\tilde{\zeta}}\frac{1}{\bar{\tilde{\zeta}}_{\alpha}}]\partial_{\alpha}(D_t\zeta-\tilde{D}_t\tilde{\zeta})\\
:=& \Lambda_1+\Lambda_2+\Lambda_3.
\end{align*}
Denote 
$$h:=D_t\zeta-D_t^0\omega-(\tilde{D}_t\tilde{\zeta}-\tilde{D}_t^0\tilde{\omega}).$$
We have 
\begin{equation}
\begin{split}
&[h,  \mathcal{H}_{\zeta}\frac{1}{\zeta_{\alpha}}+\bar{\mathcal{H}}_{\zeta}\frac{1}{\bar{\zeta}_{\alpha}}]\partial_{\alpha}D_t\zeta\\
=& -\frac{2}{\pi i}\int \frac{\Im\{\zeta(\alpha)-\zeta(\beta)\}(h(\alpha)-h(\beta))}{|\zeta(\alpha)-\zeta(\beta)|^2}\partial_{\beta}D_t\zeta(\beta) d\beta.
\end{split}
\end{equation}
Use (\ref{diffdiff}), 
\begin{equation}
\begin{split}
&D_t\zeta-\tilde{D}_t\tilde{\zeta}-(D_t^0\omega-\tilde{D}_t^0\tilde{\omega})= (b_1-\tilde{b}_1)(\tilde{\zeta}_{\alpha}+(r_0)_{\alpha})+\tilde{b}_1(r_0)_{\alpha}+D_tr_1.
\end{split}
\end{equation}
and proposition \ref{singular}, lemmma \ref{bdiff}, corollary \ref{corollaryb}, lemma \ref{lemmazeta}, corollary \ref{corollaryDtzeta}, we have 
\begin{equation}
\begin{split}
\norm{[h,  \mathcal{H}_{\zeta}\frac{1}{\zeta_{\alpha}}+\bar{\mathcal{H}}_{\zeta}\frac{1}{\bar{\zeta}_{\alpha}}]\partial_{\alpha}D_t\zeta}_{H^s}\leq & \norm{h}_{H^s}\norm{\Im \zeta_{\alpha}}_{X^{s-1},\infty}\norm{D_t\zeta}_{W^{n-1,\infty}}\\
\leq & C(\epsilon^{3/2}+E_s^{1/2})\epsilon^2\\
\leq & C(\epsilon^{7/2}+\epsilon^2 E_s^{1/2}).
\end{split}
\end{equation}
So we obtain
\begin{equation}
\norm{\Lambda_1}_{H^s} \leq C(\epsilon^{7/2}+\epsilon^2 E_s^{1/2}+\epsilon E_s+E_s^{3/2}).
\end{equation}

\vspace*{2ex}
\noindent $\Lambda_2$ is a singular integral of the form $S_2(A, f)$, whose kernel is at least of order two.  Note that
\begin{equation}
\norm{\tilde{D}_t\tilde{\zeta}-\tilde{D}_t^0\tilde{\omega}}_{H^s}=\norm{\tilde{D}_t\tilde{\xi}_1+\tilde{b}_1\tilde{\omega}_{\alpha}}_{H^s}\leq \norm{\tilde{D}_t\tilde{\xi}_1}_{H^s}+\norm{\tilde{b}_1\tilde{\omega}_{\alpha}}_{H^s}\leq C\epsilon^{1/2}.
\end{equation}
 By proposition \ref{singular}, lemma \ref{lemmazeta}, corollary \ref{corollaryDtzeta}, we have
\begin{equation}
\begin{split}
\norm{\Lambda_2}_{H^s}\leq & C\norm{\tilde{D}_t\tilde{\zeta}-\tilde{D}_t^0\tilde{\omega}}_{H^s}\norm{\Im(\zeta-\tilde{\zeta})_{\alpha}}_{W^{s-1,\infty}}\norm{\zeta_{\alpha}-\tilde{\zeta}_{\alpha}}_{W^{s-1,\infty}}\norm{D_t\zeta}_{W^{s-1,\infty}}\\
\leq & C\epsilon^{7/2}.
\end{split}
\end{equation}
The same argument gives
\begin{equation}
||\Lambda_3||_{H^s}\leq C(\epsilon^{7/2}+\epsilon^{2}E_s^{1/2})
.
\end{equation}
So we obtain
\begin{equation}
||I_1+\it{II}_1||_{H^s}\leq C(\epsilon^{7/2}+\epsilon^{2}E_s^{1/2}).
\end{equation}
We estimate $I_2+\it{II}_2$ and $I_3+\it{II}_3$ in the same way as we did for $I_1+\it{II}_1$, and obtain
\begin{equation}
\norm{I_2+\it{II}_2}_{H^s}+\norm{I_3+\it{II}_3}_{H^s}\leq C(\epsilon^{7/2}+\epsilon^{2}E_s^{1/2}).
\end{equation}
So we obtain
\begin{equation}
\norm{I+\it{II}}_{H^s}\leq C(\epsilon^{7/2}+\epsilon^{2}E_s^{1/2}).
\end{equation}
So we obtain
\begin{equation}
||(I-\mathcal{H}_{\zeta})\mathcal{R}_{11}||_{H^s}\leq C(\epsilon^{7/2}+\epsilon^{2}E_s^{1/2}).
\end{equation}
\vspace*{2ex}

\subsection{Estimate $(I-\mathcal{H}_{\zeta})\mathcal{R}_{12}$} The way we estimate $\mathcal{R}_{12}$ is similar to that of $\mathcal{R}_{11}$.

\begin{equation}
\begin{split}
\mathcal{R}_{12}=& \frac{1}{\pi i}\int \Big(\frac{D_t\zeta(\alpha,t)-D_t\zeta(\beta,t)}{\zeta(\alpha,t)-\zeta(\beta,t)}\Big)^2(\zeta-\bar{\zeta})_{\beta}d\beta-\frac{1}{\pi i}\int \Big(\frac{D_t^0\omega(\alpha,t)-D_t^0\omega(\beta,t)}{\omega(\alpha,t)-\omega(\beta,t)}\Big)^2(\omega-\bar{\omega})_{\beta}d\beta\\
&-\Big\{\frac{1}{\pi i}\int \Big(\frac{\tilde{D}_t\tilde{\zeta}(\alpha,t)-\tilde{D}_t\tilde{\zeta}(\beta,t)}{\tilde{\zeta}(\alpha,t)-\tilde{\zeta}(\beta,t)}\Big)^2(\tilde{\zeta}-\bar{\tilde{\zeta}})_{\beta}d\beta-\frac{1}{\pi i}\int \Big(\frac{\tilde{D}_t^0\tilde{\omega}(\alpha,t)-\tilde{D}_t^0\tilde{\omega}(\beta,t)}{\tilde{\omega}(\alpha,t)-\tilde{\omega}(\beta,t)}\Big)^2(\tilde{\omega}-\bar{\tilde{\omega}})_{\beta}d\beta\Big\}
\end{split}
\end{equation}
The idea is again to decompose $\zeta=\omega+\tilde{\xi}_1+r_1$ and $D_t=D_t^0+b_1\partial_{\alpha}$ to explore the cancellations, and then use Proposition \ref{singular} to obtain appropriate estimates.  For example, in the decomposition we'll obtain terms like
$$\mathcal{R}_{121}:=\frac{1}{\pi i}\int \Big(\frac{D_tr_1(\alpha,t)-D_tr_1(\beta,t)}{\zeta(\alpha,t)-\zeta(\beta,t)}\Big)^2 (\zeta-\bar{\zeta})_{\beta}d\beta.$$
Then we have
\begin{equation}
\begin{split}
||\mathcal{R}_{121}||_{H^s}\leq & C\norm{D_tr_1}_{H^s}\norm{D_tr_1}_{W^{s-1,\infty}}\norm{\Im\zeta_{\alpha}}_{W^{s-1,\infty}}\leq C\epsilon^2 E_s^{1/2}.
\end{split}
\end{equation}
Other terms can be estimated in a similar way, and we obtain
\begin{equation}
\norm{\mathcal{R}_{12}}_{H^s}\leq C(\epsilon^{7/2}+\epsilon^2 E_s^{1/2}).
\end{equation}

\subsection{Estimate $\mathcal{R}_{13}$} Recall that 
\begin{equation}
\begin{split}
\mathcal{R}_{13}=&  D_tb_1 \partial_{\alpha}(I-\mathcal{H}_{\omega})(\omega-\alpha)+b_1 D_t\partial_{\alpha}(I-\mathcal{H}_{\omega})(\omega-\alpha)\\
&-\Big\{ \tilde{D}_t\tilde{b}_1 \partial_{\alpha}(I-\mathcal{H}_{\tilde{\omega}})(\tilde{\omega}-\alpha)+\tilde{b}_1\tilde{ D}_t\partial_{\alpha}(I-\mathcal{H}_{\tilde{\omega}})(\tilde{\omega}-\alpha)\Big\}
\end{split}
\end{equation}
We provide the detail for the estimate of $$D_tb_1\partial_{\alpha}(I-\mathcal{H}_{\omega})(\omega-\alpha)-\tilde{D}_t\tilde{b}_1 \partial_{\alpha}(I-\mathcal{H}_{\tilde{\omega}})(\omega-\alpha).$$ The estimate for $b_1 D_t\partial_{\alpha}(I-\mathcal{H}_{\omega})(\omega-\alpha)-\tilde{b}_1\tilde{ D}_t\partial_{\alpha}(I-\mathcal{H}_{\tilde{\omega}})(\tilde{\omega}-\alpha)$ can be obtained in the same way. We have
\begin{equation}
\begin{split}
&D_tb_1\partial_{\alpha}(I-\mathcal{H}_{\omega})(\omega-\alpha)-\tilde{D}_t\tilde{b}_1 \partial_{\alpha}(I-\mathcal{H}_{\tilde{\omega}})(\omega-\alpha)\\
=& D_t(b_1-\tilde{b}_1)\partial_{\alpha}(I-\mathcal{H}_{\omega})(\omega-\alpha)+(b_1-\tilde{b}_1)\partial_{\alpha}\tilde{b}_1\partial_{\alpha}(I-\mathcal{H}_{\omega})(\omega-\alpha)\\
&+\tilde{D}_t\tilde{b}_1\partial_{\alpha}(\mathcal{H}_{\tilde{\omega}}-\mathcal{H}_{\omega})(\omega-\alpha)+\tilde{D}_t\tilde{b}_1\partial_{\alpha}(I-\mathcal{H}_{\tilde{\omega}})(\omega-\tilde{\omega})\\
:=& J_1+J_2+J_3+J_4.
\end{split}
\end{equation}
By Theorem \ref{longperiodic}, Sobolev embedding, and (\ref{Dtb1diff}), we have
\begin{align*}
\norm{J_1}_{H^s}=&\norm{D_t (b_1-\tilde{b}_1)\partial_{\alpha}(I-\mathcal{H}_{\omega})(\omega-\alpha)}_{H^s}\\
\leq & \norm{D_t(b_1-\tilde{b}_1)}_{H^s}\norm{\omega_\alpha-1}_{W^{s,\infty}}\\
\leq & C(\epsilon^{5/2}+\epsilon E_s^{1/2})\epsilon\\
\leq & C(\epsilon^{7/2}+\epsilon^2 E_s^{1/2}).
\end{align*}
By Theorem \ref{longperiodic}, Sobolev embedding, and corollary \ref{corollaryb}, we have 
\begin{align*}
\norm{J_2}_{H^s}\leq  &\norm{b_1-\tilde{b}_1}_{H^s}\norm{\partial_{\alpha}\tilde{b}_1}_{W^{s,\infty}}\norm{\partial_{\alpha}(I-\mathcal{H}_{\omega})(\omega-\alpha)}_{W^{s,\infty}}\\
\leq & C\epsilon^{7/2}.
\end{align*}
Estimates for $J_3$, $J_4$ are similar. So we obtain
So we have
\begin{equation}
\norm{(I-\mathcal{H}_{\zeta})\mathcal{R}_{13}}_{H^s(\RR)}\leq C(\epsilon^{7/2}+\epsilon^2 E_s^{1/2}).
\end{equation}

\subsection{Estimate $(I-\mathcal{H}_{\zeta})\mathcal{R}_{14}$} Write $\mathcal{R}_{14}$ as 
$$\mathcal{R}_{14}=\mathcal{R}_{141}+\mathcal{R}_{142},$$
where
\begin{equation}
\mathcal{R}_{141}=b_1\partial_{\alpha}D_t^0(I-\mathcal{H}_{\omega})(\omega-\alpha)-\tilde{b}_1\partial_{\alpha}\tilde{D}_t^0(I-\mathcal{H}_{\omega})(\tilde{\omega}-\alpha),
\end{equation}
and
\begin{equation}
\mathcal{R}_{142}=-iA_1\partial_{\alpha}(I-\mathcal{H}_{\omega})(\omega-\alpha)+i\tilde{A}_1\partial_{\alpha}(I-\mathcal{H}_{\tilde{\omega}})(\tilde{\omega}-\alpha).
\end{equation}
Estimates for these two terms are straightforward, we obtain
\begin{equation}
\norm{(I-\mathcal{H}_{\zeta})\mathcal{R}_{14}}_{H^s(\RR)}\leq C(\epsilon^{7/2}+\epsilon^2 E_s^{1/2}).
\end{equation}

\subsection{Estimate $\mathcal{R}_{16}$} Recall that 
\begin{equation}
\mathcal{R}_{16}=-2[D_t\zeta,\mathcal{H}_{\zeta}]\frac{\partial_{\alpha}D_t(\lambda-\tilde{\lambda})}{\zeta_{\alpha}}.
\end{equation}
To obtain better estimates, we explore the fact that $\frac{\partial_{\alpha}D_t(\lambda-\tilde{\lambda})}{\zeta_{\alpha}}$ is almot holomorphic in $\Omega(t)^c$. Write $\mathcal{R}_{16}$ as
\begin{align*}
\mathcal{R}_{16}=&-2[D_t\zeta,\mathcal{H}_{\zeta}+\bar{\mathcal{H}}\frac{1}{\bar{\zeta}_{\alpha}}]\frac{\partial_{\alpha}D_t(\lambda-\tilde{\lambda})}{\zeta_{\alpha}}+2[D_t\zeta, \bar{\mathcal{H}}\frac{1}{\bar{\zeta}_{\alpha}}]\partial_{\alpha}D_t(\lambda-\tilde{\lambda})\\
:=& \mathcal{R}_{161}+\mathcal{R}_{162}.
\end{align*}
It's easy to see that  
\begin{equation}
\norm{\mathcal{R}_{161}}_{H^s}\leq C(\epsilon^{7/2}+\epsilon^2 E_s^{1/2}).
\end{equation}
To estimate $\mathcal{R}_{162}$, we write
\begin{align*}
\mathcal{R}_{162}=&-2D_t\zeta(I-\bar{\mathcal{H}}_{\zeta})\frac{\partial_{\alpha}D_t(\lambda-\tilde{\lambda})}{\bar{\zeta}_{\alpha}}+2(I-\bar{\mathcal{H}}_{\zeta})D_t\zeta\frac{\partial_{\alpha}D_t(\lambda-\tilde{\lambda})}{\bar{\zeta}_{\alpha}}\\
:=&\mathcal{R}_{1621}+\mathcal{R}_{1622}.
\end{align*}
Note that 
\begin{align*}
\lambda-\tilde{\lambda}=& (I-\mathcal{H}_{\zeta})(\zeta-\alpha)-(I-\mathcal{H}_{\omega})(\omega-\alpha)-((I-\mathcal{H}_{\tilde{\zeta}})(\tilde{\zeta}-\alpha)-(I-\mathcal{H}_{\tilde{\omega}})(\tilde{\omega}-\alpha))\\
=& (I-\mathcal{H}_{\zeta})r_1+(\mathcal{H}_{\omega}-\mathcal{H}_{\zeta})(\omega-\alpha)+(\mathcal{H}_{\tilde{\zeta}}-\mathcal{H}_{\zeta})(\tilde{\zeta}-\alpha)-(\mathcal{H}_{\tilde{\omega}}-\mathcal{H}_{\zeta})(\tilde{\omega}-\alpha). 
\end{align*}
The last three terms are quadratic, and it's quite easy to see that they are bounded  in $H^s$ by 
$$\epsilon E_s^{1/2}+\epsilon^{5/2}.$$
So to bound $\mathcal{R}_{1621}$, it suffices to bound $-2D_t\zeta(I-\bar{\mathcal{H}}_{\zeta})\frac{\partial_{\alpha}D_t(I-\mathcal{H}_{\zeta})r_1}{\bar{\zeta}_{\alpha}}$. We have
\begin{align*}
&(I-\bar{\mathcal{H}}_{\zeta})\frac{\partial_{\alpha}D_t(I-\mathcal{H}_{\zeta})r_1}{\bar{\zeta}_{\alpha}}\\=& (I-\bar{\mathcal{H}}_{\zeta})\frac{\partial_{\alpha}(I-\mathcal{H}_{\zeta})D_tr_1}{\bar{\zeta}_{\alpha}}-(I-\bar{\mathcal{H}}_{\zeta})\frac{\partial_{\alpha}[D_t\zeta, \mathcal{H}_{\zeta}]\frac{\partial_{\alpha}r_1}{\zeta_{\alpha}}}{\bar{\zeta}_{\alpha}}.
\end{align*}
$(I-\bar{\mathcal{H}}_{\zeta})\frac{\partial_{\alpha}[D_t\zeta, \mathcal{H}_{\zeta}]\frac{\partial_{\alpha}r_1}{\zeta_{\alpha}}}{\zeta_{\alpha}}$ satisfies the estimate
\begin{align*}
\norm{(I-\bar{\mathcal{H}}_{\bar{\zeta}_{\alpha}})\frac{\partial_{\alpha}[D_t\zeta, \mathcal{H}_{\zeta}]\frac{\partial_{\alpha}r_1}{\bar{\zeta}_{\alpha}}}{\zeta_{\alpha}}}_{H^s}\leq C\epsilon E_s^{1/2}.
\end{align*}
And 
\begin{align*}
&(I-\bar{\mathcal{H}}_{\zeta})\frac{\partial_{\alpha}(I-\mathcal{H}_{\zeta})D_tr_1}{\bar{\zeta}_{\alpha}}\\
=& (I-\bar{\mathcal{H}}_{\zeta})\frac{\partial_{\alpha}(I+\bar{\mathcal{H}}_{\bar{\zeta}})D_tr_1}{\bar{\zeta}_{\alpha}}-(I-\bar{\mathcal{H}}_{\zeta})\frac{\partial_{\alpha}(\mathcal{H}_{\zeta}+\bar{\mathcal{H}}_{\bar{\zeta}})D_tr_1}{\bar{\zeta}_{\alpha}}.
\end{align*}
Note that the first term is zero, while the second term satisfies desired estimates.
So we have 
$$\norm{(I-\bar{\mathcal{H}}_{\zeta})\frac{\partial_{\alpha}(I-\mathcal{H}_{\zeta})D_tr_1}{\bar{\zeta}_{\alpha}}}_{H^s}\leq C(\epsilon E_s^{1/2}+\epsilon^{5/2}).$$
We can estimate $\norm{\mathcal{R}_{1622}}_{H^s}$ in a similar way.
So we obtain
\begin{equation}
\norm{\mathcal{R}_{16}}_{H^s}\leq C(\epsilon^{7/2}+\epsilon^2 E_s^{1/2}).
\end{equation}

\subsection{Estimate $\mathcal{R}_{17}$} It's easy to obtain estimate

\begin{equation}
\norm{\mathcal{R}_{17}}_{H^2}\leq C(\epsilon^{7/2}+\epsilon^2 E_s^{1/2}).
\end{equation}
\vspace*{2ex}

\noindent Since $\norm{\mathcal{R}_{15}}_{H^s}\leq \epsilon^{7/2}$, we obtain
\begin{equation}
\norm{(I-\mathcal{H}_{\zeta})\sum_{j=1}^5\mathcal{R}_j+\mathcal{R}_{16}+\mathcal{R}_{17}}_{H^s}\leq C(\epsilon^{7/2}+\epsilon E_s^{1/2}).
\end{equation}

\subsection{Estimate $[D_t^2-iA\partial_{\alpha}, \partial_{\alpha}^n]\rho_1$} We have
\begin{equation}
\begin{split}
&[D_t^2-iA\partial_{\alpha}, \partial_{\alpha}^n]\rho_1\\
=&\sum_{m=1}^n \partial_{\alpha}^{n-m}[D_t^2-iA\partial_{\alpha},\partial_{\alpha}]\partial_{\alpha}^{m-1}\rho_1\\
=&\sum_{m=1}^n \partial_{\alpha}^{n-m}(D_t[D_t,\partial_{\alpha}]+[D_t,\partial_{\alpha}]D_t+iA_{\alpha}\partial_{\alpha})\partial_{\alpha}^{m-1}\rho_1\\
=&-\sum_{m=1}^n \partial_{\alpha}^{n-m}D_t(b_{\alpha}\partial_{\alpha})\partial_{\alpha}^{m-1}\rho_1-\sum_{m=1}^n \partial_{\alpha}^{n-m}b_{\alpha}D_t\partial_{\alpha}^{m-1}\rho_1+i\sum_{m=1}^n\partial_{\alpha}^{n-m}A_{\alpha}\partial_{\alpha}^{m}\rho_1\\
:=& K_1+K_2+K_3.
\end{split}
\end{equation}
To estimate $K_1$, we use
\begin{equation}
\begin{split}
&\partial_{\alpha}^{n-m}D_t(b_{\alpha}\partial_{\alpha})\partial_{\alpha}^{m-1}\rho_1= \partial_{\alpha}^{n-m}(D_tb_{\alpha}\partial_{\alpha}^{m}\rho_1+b_{\alpha}D_t\partial_{\alpha}^m\rho_1)\\
=& \sum_{j=1}^{n-m} C_{n-m,j}(\partial_{\alpha}^{n-m-j}D_tb_{\alpha}\partial_{\alpha}^{m+j}\rho_1+\partial_{\alpha}^{n-m-j}b_{\alpha} \partial_{\alpha}^j D_t\partial_{\alpha}^m\rho_1),
\end{split}
\end{equation}
where 
\begin{equation}
    C_{n-m,j}=\frac{(n-m)!}{j!(n-m-j)!}.
\end{equation}

\noindent Write $D_tb_{\alpha}=\partial_{\alpha}D_tb-(b_{\alpha})^2$.
If $m>1$, then use (\ref{Dtb1diff}), corollary \ref{corollaryb}, we have 
\begin{equation}
\begin{split}
\norm{\partial_{\alpha}^{n-m-j}D_tb_{\alpha} \partial_{\alpha}^{m+j}\rho_1}_{L^2}\leq & \norm{\partial_{\alpha}^{n-m-j}D_tb_{\alpha} }_{\infty}\partial_{\alpha}^{m+j}\rho_1||_{L^2}\\
\leq &(\norm{b_{\alpha}}_{W^{n-1,\infty}}^2+\norm{\partial_{\alpha}^{n-m-j+1}D_t b}_{\infty})E_s^{1/2}\\
\leq & C\epsilon^2 E_s^{1/2}.
\end{split}
\end{equation}
Similarly, we have 
\begin{equation}
\norm{\partial_{\alpha}^{n-m-j}b_{\alpha} \partial_{\alpha}^j D_t\partial_{\alpha}^m\rho_1)}_{L^2}\leq  C\epsilon^{2}E_s^{1/2}.
\end{equation}

If $m=1$ and $j=0$,  then we cannot simply estimate $\partial_{\alpha}^{n-1}D_tb_{\alpha}$ in $L^{\infty}$, because if $n=s$, we'll lose derivatives. To avoid loss of derivatives,  we decompose $$D_tb=D_tb_1+b_1\partial_{\alpha}b_0+D_t^0b_0=D_t(b_1-\tilde{b}_1)+\tilde{D}_t\tilde{b}_1+b_1\partial_{\alpha}(b_0+\tilde{b}_1)+D_t^0b_0.$$
Then for $n\geq 3$, by corollary \ref{equivalencerho1}, Theorem \ref{longperiodic}, corollary \ref{corollaryb}, we have 
\begin{equation}
\begin{split}
&\norm{\partial_{\alpha}\rho_1\partial_{\alpha}^n D_tb}_{L^2}\\
=&\norm{\partial_{\alpha}\rho_1(\partial_{\alpha}^n D_t(b_1-\tilde{b}_1)+\partial_{\alpha}^n \tilde{D}_t\tilde{b}_1+\partial_{\alpha}^n(b_1\partial_{\alpha}(\tilde{b}_1+b_0)+\partial_{\alpha}^n(D_t^0b_0))}_{L^2}\\
\leq & \norm{D_t(b_1-\tilde{b}_1)}_{H^n}\norm{\partial_{\alpha}\rho_1}_{L^{\infty}}+\norm{b_1}_{H^n}\norm{\tilde{b}_1+b_0}_{W^{n+1,\infty}}\norm{\partial_{\alpha}\rho_1}_{L^{\infty}}+\norm{\partial_{\alpha}\rho_1}_{L^2}\norm{D_t^0b_0}_{W^{n,\infty}}\\
\leq & C\epsilon^{2}E_s^{1/2}.
\end{split}
\end{equation}
The quantity $\partial_{\alpha}^{n-1}(b_{\alpha})^2 \partial_{\alpha}\rho_1$ can be estimated similarly, and we obtain
\begin{equation}
\norm{\partial_{\alpha}^{n-1}(b_{\alpha})^2 \partial_{\alpha}\rho_1}_{L^2}\leq C\epsilon^{2}E_s^{1/2}.
\end{equation}
So we have 
\begin{equation}
\norm{\partial_{\alpha}^{n-m-j}D_tb_{\alpha} \partial_{\alpha}^{m+j}\rho_1}_{L^2}\leq C\epsilon^{2} E_s^{1/2}.
\end{equation}
We use same argument to obtain
\begin{equation}
\norm{\partial_{\alpha}^{n-m-j}b_{\alpha} \partial_{\alpha}^j D_t\partial_{\alpha}^m\rho_1)}_{L^2}\leq C\epsilon^{2} E_s^{1/2}.
\end{equation}
So we obtain 
\begin{equation}
\norm{K_1}_{L^2}\leq C\epsilon^2 E_s^{1/2}.
\end{equation}
Estimate $K_2, K_3$ in a similar way, we obtain
\begin{equation}
\norm{K_2+K_3}_{L^2}\leq C(\epsilon^{7/2}+\epsilon^{2} E_s^{1/2}).
\end{equation}
So we have 
\begin{equation}
\norm{[D_t^2-iA\partial_{\alpha}, \partial_{\alpha}^n]\rho_1}_{L^2}\leq C(\epsilon^{7/2}+\epsilon^{2} E_s^{1/2}).
\end{equation}

\vspace*{2ex}

\noindent So we obtain estimate for $\mathcal{C}_{1,n}$: for $0\leq n\leq s$, 
\begin{equation}
\norm{\mathcal{C}_{1,n}}_{L^2}\leq C(\epsilon^{7/2}+\epsilon^2 E_s^{1/2}).
\end{equation}

\subsection{Estimate $\partial_{\alpha}^n (I-\mathcal{H}_{\zeta})\mathcal{S}_1$} In this subsection we obtain estimate for the quantity $\partial_{\alpha}^n (I-\mathcal{H}_{\zeta})\mathcal{S}_1$. Since
$$\partial_{\alpha}^n (I-\mathcal{H}_{\zeta})\mathcal{S}_1=(I-\mathcal{H}_{\zeta})\partial_{\alpha}^n \mathcal{S}_1-[\partial_{\alpha}^n, \mathcal{H}_{\zeta}]\mathcal{S}_1,$$
it suffices to estimate $(I-\mathcal{H}_{\zeta})\partial_{\alpha}^n \mathcal{S}_1$.

\noindent  Recall that 
\begin{align*}
\mathcal{S}_1=&  D_tG-D_t^0G_0-\tilde{D}_t\tilde{G}+\tilde{D}_t^0\tilde{G}_0+(\mathcal{P}-\mathcal{P}_0)\tilde{D}_t(I-\mathcal{H}_{\omega})(\omega-\alpha)\\
&-(\tilde{\mathcal{P}}-\tilde{\mathcal{P}}_0)\tilde{D}_t^0(I-\mathcal{\tilde{\omega}})(\tilde{\omega}-\alpha)+\epsilon^4\mathcal{R}.
\end{align*}

\subsubsection{Estimate $\norm{ D_tG-D_t^0G_0-\tilde{D}_t\tilde{G}+\tilde{D}_t^0\tilde{G}_0}_{H^s}$}
We have 
\begin{equation}\label{DTG}
\begin{split}
D_tG=&-2[D_t^2\zeta,\mathcal{H}_{\zeta}\frac{1}{\zeta_{\alpha}}+\bar{\mathcal{H}}_{\zeta}\frac{1}{\bar{\zeta}_{\alpha}}]\partial_{\alpha}D_t\zeta-2[D_t\zeta,\mathcal{H}_{\zeta}\frac{1}{\zeta_{\alpha}}+\bar{\mathcal{H}}_{\zeta}\frac{1}{\bar{\zeta}_{\alpha}}]\partial_{\alpha}D_t^2\zeta\\
&+\frac{2}{\pi i}\int \Big(\frac{D_t\zeta(\alpha)-D_t\zeta(\beta)}{\zeta(\alpha)-\zeta(\beta)}\Big)^2\partial_{\beta} D_t\zeta(\beta)d\beta\\
&-\frac{2}{\pi i}\int \frac{|D_t\zeta(\alpha)-D_t\zeta(\beta)|^2}{(\bar{\zeta}(\alpha)-\bar{\zeta}(\beta))^2}\partial_{\beta}D_t\zeta(\beta)d\beta\\
&+\frac{4}{\pi}\int \frac{(D_t\zeta(\alpha)-D_t\zeta(\beta))(D_t^2\zeta(\alpha)-D_t^2\zeta(\beta))}{(\zeta(\alpha)-\zeta(\beta))^2}\partial_{\beta}\Im \zeta(\beta)d\beta\\
&+\frac{2}{\pi i}\int \Big(\frac{D_t\zeta(\alpha)-D_t\zeta(\beta)}{\zeta(\alpha)-\zeta(\beta)}\Big)^2\partial_{\beta} \Im D_t\zeta(\beta)d\beta\\
&-\frac{4}{\pi }\int \Big(\frac{D_t\zeta(\alpha)-D_t\zeta(\beta)}{\zeta(\alpha)-\zeta(\beta)}\Big)^3\partial_{\beta} \Im \zeta(\beta)d\beta.
\end{split}
\end{equation}
And
\begin{equation}\label{DTOMEGA}
\begin{split}
D_t^0G_0=& -2[(D_t^0)^2\omega, \mathcal{H}_{\omega}\frac{1}{\omega_{\alpha}}+\bar{\mathcal{H}}_{\omega}\frac{1}{\omega_{\alpha}}]\partial_{\alpha}D_t^0\omega\\
&-2[D_t^0\omega,\mathcal{H}_{\omega}\frac{1}{\omega_{\alpha}}+\bar{\mathcal{H}}_{\omega}\frac{1}{\omega_{\alpha}}]\partial_{\alpha}(D_t^0)\omega\\
&+\frac{2}{\pi i}\int \Big(\frac{D_t^0\omega(\alpha,t)-D_t^0\omega(\beta,t)}{\omega(\alpha,t)-\omega(\beta,t)}\Big)^2\partial_{\beta}D_t^0\omega(\beta,t)d\beta\\
&\frac{4}{\pi }\int \frac{(D_t^0\omega(\alpha)-D_t^0\omega(\beta))((D_t^0)^2\omega(\alpha)-(D_t^0)^2\omega(\beta))}{(\omega(\alpha)-\omega(\beta))^2}\partial_{\beta}\Im\{\omega(\beta)\}d\beta\\
&-\frac{2}{\pi i} \int\frac{|D_t^0\omega(\alpha,t)-D_t^0\omega(\beta,t)|^2}{(\bar{\omega}(\alpha,t)-\bar{\omega}(\beta,t))^2}\partial_{\beta}D_t^0\omega(\beta,t)d\beta\\
&+\frac{2}{\pi}\int\Big(\frac{D_t^0\omega(\alpha,t)-D_t^0\omega(\beta,t)}{\omega(\alpha,t)-\omega(\beta,t)}\Big)^2\partial_{\beta}\Im\{D_t^0\omega(\beta,t)\}d\beta\\
&-\frac{4}{\pi}\int\Big(\frac{D_t^0\omega(\alpha,t)-D_t^0\omega(\beta,t)}{\omega(\alpha,t)-\omega(\beta,t)}\Big)^2\partial_{\beta}\Im\{\omega(\beta,t)\}d\beta\\
\end{split}
\end{equation}
$\tilde{D}_t \tilde{G}$ and $\tilde{D}_t^0\tilde{G}_0$ are given similarly:  For $\tilde{D}_t \tilde{G}$, replace $D_t$ by $\tilde{D}_t$ and replance $\zeta$ by $\tilde{\zeta}$ in (\ref{DTG}). For $\tilde{D}_t^0\tilde{G}_0$, replace $D_t^0$ by $\tilde{D}_t^0$, and replace $\omega$ by $\tilde{\omega}$ in (\ref{DTOMEGA}).

\vspace*{2ex}
\noindent $\bullet$ Estimate the quantity
\begin{align*}
\mathcal{S}_{11}:=&-2[D_t^2\zeta,\mathcal{H}_{\zeta}\frac{1}{\zeta_{\alpha}}+\bar{\mathcal{H}}_{\zeta}\frac{1}{\bar{\zeta}_{\alpha}}]\partial_{\alpha}D_t\zeta+2[\tilde{D}_t^2\tilde{\zeta},\mathcal{H}_{\tilde{\zeta}}\frac{1}{\tilde{\zeta}_{\alpha}}+\bar{\mathcal{H}}_{\tilde{\zeta}}\frac{1}{\bar{\tilde{\zeta}}_{\alpha}}]\partial_{\alpha}\tilde{D}_t\tilde{\zeta}\\
&+2[(D_t^0)^2\omega,\mathcal{H}_{\omega}\frac{1}{\omega_{\alpha}}+\bar{\mathcal{H}}_{\omega}\frac{1}{\bar{\omega}_{\alpha}}]\partial_{\alpha}D_t^0\omega-2[(\tilde{D}_t^0)^2\tilde{\omega},\mathcal{H}_{\tilde{\omega}}\frac{1}{\tilde{\omega}_{\alpha}}+\bar{\mathcal{H}}_{\tilde{\omega}}\frac{1}{\bar{\tilde{\omega}}_{\alpha}}]\partial_{\alpha}\tilde{D}_t^0\tilde{\omega}
\end{align*}
Rewrite $\mathcal{S}_{11}$ as 
\begin{align*}
\mathcal{S}_{11}=& -2[D_t^2\zeta-\tilde{D}_t^2\tilde{\zeta}, \mathcal{H}_{\zeta}\frac{1}{\zeta_{\alpha}}+\bar{\mathcal{H}}_{\zeta}\frac{1}{\bar{\zeta}_{\alpha}}]\partial_{\alpha}D_t\zeta   \quad \quad \quad \quad \quad &:=S_{111}\\
&-2[\tilde{D}_t^2\tilde{\zeta}, \mathcal{H}_{\zeta}\frac{1}{\zeta_{\alpha}}+\bar{\mathcal{H}}_{\zeta}\frac{1}{\bar{\zeta}_{\alpha}}-\Big(\mathcal{H}_{\tilde{\zeta}}\frac{1}{\tilde{\zeta}_{\alpha}}+\bar{\mathcal{H}}_{\tilde{\zeta}}\frac{1}{\bar{\tilde{\zeta}}_{\alpha}}\Big)]\partial_{\alpha}D_t\zeta   &:= S_{112}\\
&-2[\tilde{D}_t^2\tilde{\zeta},\mathcal{H}_{\tilde{\zeta}}\frac{1}{\tilde{\zeta}_{\alpha}}+\bar{\mathcal{H}}_{\tilde{\zeta}}\frac{1}{\bar{\tilde{\zeta}}_{\alpha}}]\partial_{\alpha}(D_t\zeta-\tilde{D}_t\tilde{\zeta})    &:=S_{113}\\
&+2[(D_t^0)^2\omega-(\tilde{D}_t^0)^2\tilde{\omega},\mathcal{H}_{\omega}\frac{1}{\omega_{\alpha}}+\bar{\mathcal{H}}_{\omega}\frac{1}{\bar{\omega}_{\alpha}}]\partial_{\alpha}D_t^0\omega   &:=S_{114}\\
&+2[(\tilde{D}_t^0)^2\tilde{\omega},\mathcal{H}_{\omega}\frac{1}{\omega_{\alpha}}+\bar{\mathcal{H}}_{\omega}\frac{1}{\bar{\omega}_{\alpha}}-\Big(\mathcal{H}_{\tilde{\omega}}\frac{1}{\tilde{\omega}_{\alpha}}+\bar{\mathcal{H}}_{\tilde{\omega}}\frac{1}{\bar{\tilde{\omega}}_{\alpha}}\Big)]\partial_{\alpha}D_t^0\omega   &:=S_{115}\\
&+2[(\tilde{D}_t^0)^2\tilde{\omega},\mathcal{H}_{\tilde{\omega}}\frac{1}{\tilde{\omega}_{\alpha}}+\bar{\mathcal{H}}_{\tilde{\omega}}\frac{1}{\bar{\tilde{\omega}}_{\alpha}}]\partial_{\alpha}(D_t^0\omega-\tilde{D}_t^0\tilde{\omega})   &:=S_{116}.
\end{align*}
We give details of estimate of $S_{111}+S_{114}$ only, the estimates for $S_{112}+S_{115}$ and $S_{113}+S_{116}$ are similar. Write $S_{111}+S_{114}$ as
\begin{align*}
S_{111}+S_{114}=& -2[D_t^2\zeta-\tilde{D}_t^2\tilde{\zeta}-((D_t^0)^2\omega-(\tilde{D}_t^0)^2\tilde{\omega}), \mathcal{H}_{\zeta}\frac{1}{\zeta_{\alpha}}+\bar{\mathcal{H}}_{\zeta}\frac{1}{\bar{\zeta}_{\alpha}}]\partial_{\alpha}D_t\zeta\\
&-2[(D_t^0)^2\omega-(\tilde{D}_t^0)^2\tilde{\omega},\mathcal{H}_{\zeta}\frac{1}{\zeta_{\alpha}}+\bar{\mathcal{H}}_{\zeta}\frac{1}{\bar{\zeta}_{\alpha}}-\Big(\mathcal{H}_{\omega}\frac{1}{\omega_{\alpha}}+\bar{\mathcal{H}}_{\omega}\frac{1}{\bar{\omega}_{\alpha}}\Big)]\partial_{\alpha}D_t\zeta\\
&-2[(D_t^0)^2\omega-(\tilde{D}_t^0)^2\tilde{\omega},\mathcal{H}_{\omega}\frac{1}{\omega_{\alpha}}+\bar{\mathcal{H}}_{\omega}\frac{1}{\bar{\omega}_{\alpha}}]\partial_{\alpha}(D_t\zeta-D_t^0\omega)\\
:=&M_1+M_2+M_3.
\end{align*}
To estimate $M_1$, we use (\ref{seconddiff})
\begin{equation}
\begin{split}
&D_t^2\zeta-\tilde{D}_t^2\tilde{\zeta}-(D_t^0)^2\omega+(\tilde{D}_t^0)^2\tilde{\omega}\\
=& D_t^2r_1+(D_t^2-\tilde{D}_t^2)\tilde{\xi}_1+(D_t^2-(D_t^0)^2)\omega-((D_t^0)^2-(\tilde{D}_t^0)^2)\omega\\
\end{split}
\end{equation}
Denote
$$q:=(D_t^2-\tilde{D}_t^2)\tilde{\xi}_1+(D_t^2-(D_t^0)^2)\omega-((D_t^0)^2-(\tilde{D}_t^0)^2)\omega.$$
We have the rough estimate
\begin{equation}
\norm{q}_{H^s}\leq C(E_s^{1/2}+\epsilon^{3/2}).
\end{equation} 
Decompose $D_t\zeta$ as 
$$D_t\zeta=D_tr_1+D_t^0\omega+b_1\omega_{\alpha}+\tilde{D}_t\tilde{\xi}_1+(b-\tilde{b})\partial_{\alpha}\tilde{\xi}_1.$$
Then $M_1$ can be written as
\begin{align}
    M_1=& -2[D_t^2r_1, \mathcal{H}_{\zeta}\frac{1}{\zeta_{\alpha}}+\bar{\mathcal{H}}_{\zeta}\frac{1}{\bar{\zeta}_{\alpha}}]\partial_{\alpha}(D_tr_1+b_1\omega_{\alpha}+(b-\tilde{b})\partial_{\alpha}\tilde{\xi}_1)\\
    &-2[D_t^2r_1, \mathcal{H}_{\zeta}\frac{1}{\zeta_{\alpha}}+\bar{\mathcal{H}}_{\zeta}\frac{1}{\bar{\zeta}_{\alpha}}]\partial_{\alpha}(D_t^0\omega+\tilde{D}_t\tilde{\xi}_1)\\
    &-2[q, \mathcal{H}_{\zeta}\frac{1}{\zeta_{\alpha}}+\bar{\mathcal{H}}_{\zeta}\frac{1}{\bar{\zeta}_{\alpha}}]\partial_{\alpha}(D_tr_1+b_1\omega_{\alpha}+(b-\tilde{b})\partial_{\alpha}\tilde{\xi}_1)\\
    & -2[q, \mathcal{H}_{\zeta}\frac{1}{\zeta_{\alpha}}+\bar{\mathcal{H}}_{\zeta}\frac{1}{\bar{\zeta}_{\alpha}}]\partial_{\alpha}(D_t^0\omega+\tilde{D}_t\tilde{\xi}_1)\\
    :=&M_{11}+M_{12}+M_{13}+M_{14}.
\end{align}
For $M_{11}$, we have 
\begin{align}
    \norm{M_{11}}_{H^s}\leq & C\norm{D_t^2r_1}_{H^s}\norm{Im \zeta_{\alpha}}_{W^{s-1,\infty}}\norm{D_tr_1+b_1\omega_{\alpha}+(b-\tilde{b})\partial_{\alpha}\tilde{\xi}_1}_{H^s}\\
    \leq & C\epsilon E_s^{1/2}(E_s^{1/2}+\epsilon^{3/2}) \\
    \leq & C\epsilon^2 E_s^{1/2}.
\end{align}
For $M_{12}$, we have 
\begin{align}
    \norm{M_{12}}_{H^s}\leq & C\norm{D_t^2r_1}_{H^s}\norm{Im \zeta_{\alpha}}_{W^{s-1,\infty}}\norm{D_t^0\omega+\tilde{D}_t\tilde{\xi}_1}_{W^{s,\infty}}\leq C\epsilon^2 E_s^{1/2}.
\end{align}
For $M_{13}$, we have 
\begin{align}
    \norm{M_{13}}_{H^s}\leq C\norm{q}_{H^s}\norm{Im\zeta_{\alpha}}_{W^{s-1,\infty}}\norm{D_tr_1+b_1\omega_{\alpha}+(b-\tilde{b})\partial_{\alpha}\tilde{\xi}_1}_{H^s}\leq C\epsilon^2 E_s^{1/2}.
\end{align}
For $M_{14}$, we have 
\begin{align}
    \norm{M_{14}}_{H^s}\leq C\norm{q}_{H^s}\norm{Im\zeta_{\alpha}}_{W^{s-1,\infty}}\norm{D_t^0\omega+\tilde{D}_t\tilde{\xi}_1}_{W^{s,\infty}}\leq C\epsilon^2 E_s^{1/2}.
\end{align}
So we have 
\begin{equation}
\begin{split}
\norm{M_1}_{H^s}\leq & C\epsilon^2 E_s^{1/2}.
\end{split}
\end{equation}
$M_2$ and $M_3$ can be estimated similarly, we obtain
\begin{equation}
\norm{M_2}_{H^s}+\norm{M_3}_{H^s}\leq C\epsilon^2 E_s^{1/2}+\epsilon^{7/2}.
\end{equation}
So we have 
\begin{equation}
\norm{S_{111}+S_{114}}_{H^s}\leq C(\epsilon^2 E_s^{1/2}+\epsilon^{7/2}).
\end{equation}
Similarly,
\begin{equation}
\norm{S_{112}+S_{115}}_{H^s}\leq  C(\epsilon^2 E_s^{1/2}+\epsilon^{7/2}).
\end{equation}

\begin{equation}
\norm{S_{113}+S_{116}}_{H^s}\leq C(\epsilon^2 E_s^{1/2}+\epsilon^{7/2}).
\end{equation}
So we have 
\begin{equation}
\norm{\mathcal{S}_{11}}_{H^s}\leq C(\epsilon^2 E_s^{1/2}+\epsilon^{7/2}).
\end{equation}
We can estimate other quantities in $ D_tG-D_t^0G_0-\tilde{D}_t\tilde{G}+\tilde{D}_t^0\tilde{G}_0$ similarly, and obtain
\begin{equation}
\norm{ D_tG-D_t^0G_0-\tilde{D}_t\tilde{G}+\tilde{D}_t^0\tilde{G}_0}_{H^s}\leq C(\epsilon^2 E_s^{1/2}+\epsilon^{7/2}).
\end{equation}

\vspace*{2ex}

\subsubsection{Estimate $\norm{(\mathcal{P}-\mathcal{P}_0)\tilde{D}_t(I-\mathcal{H}_{\omega})(\omega-\alpha)-(\tilde{\mathcal{P}}-\tilde{\mathcal{P}}_0)\tilde{D}_t^0(I-\mathcal{\tilde{\omega}})(\tilde{\omega}-\alpha)}_{H^s}$} We have
\begin{equation}
\begin{split}
& (\mathcal{P}-\mathcal{P}_0)\tilde{D}_t(I-\mathcal{H}_{\omega})(\omega-\alpha)-(\tilde{\mathcal{P}}-\tilde{\mathcal{P}}_0)\tilde{D}_t^0(I-\mathcal{\tilde{\omega}})(\tilde{\omega}-\alpha)\\
=&  (\mathcal{P}-\mathcal{P}_0-\tilde{\mathcal{P}}+\tilde{\mathcal{P}}_0)\tilde{D}_t(I-\mathcal{H}_{\omega})(\omega-\alpha)+(\tilde{\mathcal{P}}-\tilde{\mathcal{P}}_0)[\tilde{D}_t(I-\mathcal{H}_{\omega})(\omega-\alpha)-\tilde{D}_t^0(I-\mathcal{\tilde{\omega}})(\tilde{\omega}-\alpha)]\\
:=&U_1+U_2.
\end{split}
\end{equation}
$U_2$ is a known function, it's easy to obtain that
\begin{equation}
\norm{U_2}_{H^s}\leq C\epsilon^{7/2}.
\end{equation}
The operator 
\begin{align*}
& \mathcal{P}-\mathcal{P}_0-\tilde{\mathcal{P}}+\tilde{\mathcal{P}}_0= D_tb_1\partial_{\alpha}+b_1\partial_{\alpha}D_t^0-\tilde{D}_t\tilde{b}_1\partial_{\alpha}-\tilde{b}_1\partial_{\alpha}\tilde{D}_t^0\\
=& (b_1-\tilde{b}_1)\partial_{\alpha}+\tilde{D}_t(b_1-\tilde{b}_1)\partial_{\alpha}+(b_1-\tilde{b}_1)\partial_{\alpha}D_t^0+\tilde{b}_1\partial_{\alpha}(b_0-\tilde{b}_0)\partial_{\alpha}.
\end{align*}
Then it's easy to obtain
\begin{equation}
\norm{U_1}_{H^s}\leq C(\epsilon^{7/2}+\epsilon^2E_s^{1/2}).
\end{equation}
So we have 
\begin{equation}
\begin{split}
&\norm{(\mathcal{P}-\mathcal{P}_0)\tilde{D}_t(I-\mathcal{H}_{\omega})(\omega-\alpha)-(\tilde{\mathcal{P}}-\tilde{\mathcal{P}}_0)\tilde{D}_t^0(I-\mathcal{\tilde{\omega}})(\tilde{\omega}-\alpha)}_{H^s}\\
\leq & C(\epsilon^{7/2}+\epsilon^2E_s^{1/2}).
\end{split}
\end{equation}

\subsubsection{Estimate $L$}  We have obtained estimate for the quantity $L$ in the previous section, we have 
\begin{equation}
\begin{split}
&(I-\mathcal{H}_{\zeta})(\zeta-\alpha)-(I-\mathcal{H}_{\omega})(\omega-\alpha)-(I-\mathcal{H}_{\tilde{\zeta}})(\tilde{\zeta}-\alpha)+(I-\mathcal{H}_{\tilde{\omega}})(\tilde{\omega}-\alpha)\\
\leq & C(\epsilon^{7/2}+\epsilon^2 E_s^{1/2}).
\end{split}
\end{equation}

\vspace*{2ex}

\noindent Therefore 
\begin{equation}
\norm{\partial_{\alpha}^n \mathcal{S}_1}_{L^s}\leq  C(\epsilon^{7/2}+\epsilon^2 E_s^{1/2}).
\end{equation}
So we have 
\begin{equation}
\norm{\partial_{\alpha}^n(I-\mathcal{H}_{\zeta}) \mathcal{S}_1}_{L^s}\leq  C(\epsilon^{7/2}+\epsilon^2 E_s^{1/2}).
\end{equation}

\subsection{Estimate $[D_t^2-iA\partial_{\alpha}, \partial_{\alpha}^n]\sigma_1$} The way we estimate this term is similar to that of $[D_t^2-iA\partial_{\alpha}, \partial_{\alpha}^n]\rho_1$. We omitt the details. We have 
\begin{equation}
\norm{[D_t^2-iA\partial_{\alpha}, \partial_{\alpha}^n]\sigma_1}_{H^s}\leq C(\epsilon^{7/2}+\epsilon^2E_s^{1/2}).
\end{equation}

\subsection{Estimate $\norm{-2[D_t\zeta,\mathcal{H}_{\zeta}]\frac{\partial_{\alpha}D_t(\delta-\tilde{\delta})}{\zeta_{\alpha}}}_{H^s}$ and 
$\norm{\frac{1}{\pi i}\int \Big(\frac{D_t\zeta(\alpha)-D_t\zeta(\beta)}{\zeta(\alpha)-\zeta(\beta)}\Big)^2\partial_{\beta}(\delta-\tilde{\delta})d\beta}_{H^s}$} Estimates of these two quantities are similar to that of $\mathcal{R}_{16}, \mathcal{R}_{17}$, respectively. We have
\begin{equation}
\begin{split}
    &\norm{-2[D_t\zeta,\mathcal{H}_{\zeta}]\frac{\partial_{\alpha}D_t(\delta-\tilde{\delta})}{\zeta_{\alpha}}}_{H^s}+
\norm{\frac{1}{\pi i}\int \Big(\frac{D_t\zeta(\alpha)-D_t\zeta(\beta)}{\zeta(\alpha)-\zeta(\beta)}\Big)^2\partial_{\beta}(\delta-\tilde{\delta})d\beta}_{H^s}\\
\leq &
C(\epsilon^{7/2}+\epsilon^2 E_s^{1/2}).
\end{split}
\end{equation}

\noindent Sum up these estimates 
\begin{lemma}
We have
\begin{equation}
\norm{\mathcal{C}_{2,n}}_{L^2}=\norm{\mathcal{P}\sigma_1^{(n)}}_{L^2}\leq  C(\epsilon^{7/2}+\epsilon^2 E_s^{1/2}).
\end{equation}
\end{lemma}

\subsection{Estimate the quantity $\frac{a_t}{a}\circ \kappa^{-1}$} By (\ref{ata11111}), lemma \ref{realinverse}, corollary \ref{boundAAA}, it's easy to obtain
\begin{equation}
\norm{\frac{a_t}{a}\circ \kappa^{-1}}_{L^{\infty}}\leq \epsilon^{3/2}+\epsilon E_s^{1/2}.
\end{equation}

\subsection{Estimate $\mathcal{R}_1^{(n)}$} Recall that $$\mathcal{R}_1^{(n)}=\frac{1}{2}(I+\mathcal{H}_{\zeta})\rho_1^{(n)}=\frac{1}{2}(I+\mathcal{H}_{\zeta})\partial_{\alpha}^n \rho_1.$$
So we have
\begin{align*}
\mathcal{R}_1^{(n)}=& \frac{1}{2}\partial_{\alpha}^n (I+\mathcal{H}_{\zeta})\rho_1-\frac{1}{2}[\partial_{\alpha}^n, \mathcal{H}_{\zeta}]\rho_1\\
=&\frac{1}{2}\partial_{\alpha}^n(I+\mathcal{H}_{\zeta})(I-\mathcal{H}_{\zeta})(\lambda-\tilde{\lambda})-\frac{1}{2}\sum_{m=1}^n \partial_{\alpha}^{n-m}[\zeta_{\alpha}-1, \mathcal{H}_{\zeta}]\frac{\partial_{\alpha}^{m}\rho_1}{\zeta_{\alpha}}\\
=&-\frac{1}{2}\sum_{m=1}^n \partial_{\alpha}^{n-m}[\zeta_{\alpha}-1, \mathcal{H}_{\zeta}]\frac{\partial_{\alpha}^{m}\rho_1}{\zeta_{\alpha}}.
\end{align*}
Decompose $\zeta_{\alpha}-1=(r_1)_{\alpha}+(r_0)_{\alpha}+\tilde{\zeta}_{\alpha}-1$. We estimate $(r_1)_{\alpha}$ in $H^s(\RR)$ while estimate $(r_0)_{\alpha}+\tilde{\zeta}_{\alpha}-1$ in $W^{s,\infty}$. By Proposition \ref{singular}, 
\begin{equation}
\begin{split}
\norm{\partial_{\alpha}\mathcal{R}_1^{(n)}}_{L^2}\leq & \sum_{m=1}^n \norm{\partial_{\alpha}^{n-m+1}[\zeta_{\alpha}-1, \mathcal{H}_{\zeta}]\frac{\partial_{\alpha}^m \rho_1}{\zeta_{\alpha}}}_{L^2}\\
\leq & C\Big(\norm{(r_1)_{\alpha}}_{H^s}+\norm{(r_0)_{\alpha}+\tilde{\zeta}_{\alpha}-1}_{W^{s,\infty}}\Big)\norm{\partial_{\alpha}\rho_1}_{H^s}\\
\leq & C(E_s^{1/2}+\epsilon)E_s^{1/2}.
\end{split}
\end{equation}
We need also to bound $\partial_t \mathcal{R}_1^{(n)}$. We have
\begin{align*}
\partial_t\mathcal{R}_1^{(n)}=& -\frac{1}{2}\sum_{m=1}^n\partial_{\alpha}^{n-m}\partial_t[\zeta_{\alpha}-1,\mathcal{H}_{\zeta}]\frac{\partial_{\alpha}^m\rho_1}{\zeta_{\alpha}}\\
=&-\frac{1}{2}\sum_{m=1}^n \partial_{\alpha}^{n-m}\Big\{[\partial_{\alpha}\zeta_t, \mathcal{H}_{\zeta}]\frac{\partial_{\alpha}^m\rho_1}{\zeta_{\alpha}}+[\zeta_{\alpha}-1,\mathcal{H}_{\zeta}]\frac{\partial_{\alpha}^m \partial_t\rho_1}{\zeta_{\alpha}}\\
&-\frac{1}{\pi i}\int \frac{(\zeta_{\alpha}(\alpha)-\zeta_{\beta}(\beta))(\zeta_t(\alpha)-\zeta_t(\beta))}{(\zeta(\alpha)-\zeta(\beta))^2}\partial_{\beta}^m \rho_1(\beta)d\beta\Big\}
\end{align*}
Write $\partial_t=D_t-b\partial_{\alpha}$, then use Proposition \ref{singular} to obtain desired estimates. For example, the term 
$\partial_{\alpha}^{n-m}[\partial_{\alpha}D_t\zeta, \mathcal{H}_{\zeta}]\frac{\partial_{\alpha}^m\rho_1}{\zeta_{\alpha}}$ can be estimated as follows
\begin{align*}
&\norm{\partial_{\alpha}^{n-m}[\partial_{\alpha}D_t\zeta, \mathcal{H}_{\zeta}]\frac{\partial_{\alpha}^m\rho_1}{\zeta_{\alpha}}}_{L^2}\\
\leq & \norm{\partial_{\alpha}(D_t\zeta-D_t^0\omega)}_{H^{n-1}}+\norm{\partial_{\alpha}D_t^0\omega}_{W^{n-1,\infty}}\norm{\rho_1}_{H^{n}}\\
\leq & C(E_s+\epsilon E_s^{1/2}+\epsilon^{5/2}).
\end{align*}
Similar argument gives
\begin{equation}\label{partialtR1}
\norm{\partial_t \mathcal{R}_1^{(n)}}_{L^2}\leq  C(E_s+\epsilon E_s^{1/2}+\epsilon^{5/2}).
\end{equation}
So we have 
\begin{equation}
\int\partial_t\mathcal{R}_1^{(n)}\partial_{\alpha}\bar{\mathcal{R}}_1^{(n)}d\alpha\leq C(E_s^2+\epsilon E_s^{3/2}+\epsilon^2E_s+\epsilon^{7/2}E_s^{1/2}+\epsilon^5).
\end{equation}

\subsection{Estimate $\phi_1^{(n)}$} Recall that 
$$\phi_1^{(n)}=\frac{1}{2}(I-\mathcal{H}_{\zeta})\partial_{\alpha}^n\rho_1.$$
We use the rough estimate for $\phi_1^{(n)}$: for $n\leq s$, 
\begin{equation}\label{phi1n}
\begin{split}
\norm{\partial_{\alpha}\phi_1^{(n)}}_{L^2}=&\frac{1}{2}\norm{(I-\mathcal{H}_{\zeta})\partial_{\alpha}^{n+1}\rho_1-[\zeta_{\alpha}-1,\mathcal{H}_{\zeta}]\frac{\rho_1^{n+1}}{\zeta_{\alpha}}}_{L^2}\\
\leq & \norm{\partial_{\alpha}\rho_1}_{H^n}\\
\leq & C(E_s^{1/2}+\epsilon^{3/2}).
\end{split}
\end{equation}

\subsection{Estimate $\int \partial_{\alpha}\phi_1^{(n)}\overline{\partial_{t}\mathcal{R}_1^{(n)}}$} If we use (\ref{partialtR1}) and (\ref{phi1n}), then 
\begin{align}
    \int \partial_{\alpha}\phi_1^{(n)}\overline{\partial_{t}\mathcal{R}_1^{(n)}}\leq & \norm{\phi_1^{(n)}}_{L^2}\norm{\partial_t\mathcal{R}_1^{(n)}}_{L^2}\\ 
    \leq & C(E_s^{1/2}+\epsilon^{3/2})(E_s+\epsilon E_s^{1/2}+\epsilon^{5/2})\\
    \leq & C(E_s^{3/2}+\epsilon^4),
\end{align}
which is not good enough. So we need to explore the cancellation between $\phi_1^{(n)}$ and $\partial_t\mathcal{R}_1^{(n)}$.
We have
\begin{align*}
\partial_t\mathcal{R}_1^{(n)}=& \frac{1}{2}\partial_t (I+\mathcal{H}_{\zeta})\mathcal{R}_1^{(n)}\\
=& \frac{1}{2}(I+\mathcal{H}_{\zeta})\partial_t\mathcal{R}_1^{(n)}+[\zeta_t,\mathcal{H}]\frac{\partial_{\alpha}\mathcal{R}_1^{(n)}}{\zeta_{\alpha}}.
\end{align*}
For the second term, we have
\begin{align*}
\int \partial_{\alpha}\phi_1^{(n)}\overline{[\zeta_t,\mathcal{H}]\frac{\partial_{\alpha}\mathcal{R}_1^{(n)}}{\zeta_{\alpha}}}d\alpha\leq & (E_s^{1/2}+\epsilon^{3/2})((E_s^{1/2}+\epsilon)E_s^{1/2}\epsilon\\
\leq & C(\epsilon E_s^{3/2}+\epsilon^2 E_s+\epsilon^{7/2}E_s^{1/2}).
\end{align*}
For the first term, note that
\begin{equation}
    \partial_{\alpha}\phi_1^{(n)}=\frac{1}{2}(I-\mathcal{H}_{\zeta})\partial_{\alpha}^{n+1}\rho_1-\frac{1}{2}[\zeta_{\alpha}-1,\mathcal{H}_{\zeta}]\partial_{\alpha}^{n+1}\rho_1.
\end{equation}
We have 
\begin{align}
    \int [\zeta_{\alpha}-1,\mathcal{H}_{\zeta}]\partial_{\alpha}^{n+1}\rho_1\overline{(I+\mathcal{H}_{\zeta})\partial_t\mathcal{R}_1^{(n)}}d\alpha\leq & C\norm{\zeta_{\alpha}-1}_{W^{s-1,\infty}}\norm{\partial_{\alpha}\rho_1}_{H^s}\norm{\partial_{t}\mathcal{R}_1^{(n)}}_{L^2}\\
    \leq & C\epsilon (E_s^{1/2}+\epsilon^{3/2})(E_s+\epsilon E_s^{1/2}+\epsilon^{5/2})\\
    \leq & C(\epsilon E_s^{3/2}+\epsilon^2 E_s+\epsilon^{7/2}E_s^{1/2}+\epsilon^5).
\end{align}
We have 
\begin{align}
    &\int (I-\mathcal{H}_{\zeta})\partial_{\alpha}^{n+1}\rho_1\overline{(I+\mathcal{H}_{\zeta})\partial_t\mathcal{R}_1^{(n)}}d\alpha\\
    =& \int (I-\mathcal{H}_{\zeta})\partial_{\alpha}^{n+1}\rho_1(I-\mathcal{H}_{\zeta})\overline{\partial_t\mathcal{R}_1^{(n)}}d\alpha+\int (I-\mathcal{H}_{\zeta})\partial_{\alpha}^{n+1}\rho_1(\overline{\mathcal{H}_{\zeta}}
    +\mathcal{H}_{\zeta})\overline{\partial_t\mathcal{R}_1^{(n)}}d\alpha
\end{align}
We have 
\begin{align}
    &\int (I-\mathcal{H}_{\zeta})\partial_{\alpha}^{n+1}\rho_1(\overline{\mathcal{H}_{\zeta}}
    +\mathcal{H}_{\zeta})\overline{\partial_t\mathcal{R}_1^{(n)}}d\alpha\\
    \leq & \norm{\partial_{\alpha}^{n+1}\rho_1}_{L^2}\norm{Im \zeta_{\alpha}}_{W^{s-1,\infty}}\norm{\partial_t\mathcal{R}_1^{(n)}}_{L^2}\\
        \leq & C\epsilon (E_s^{1/2}+\epsilon^{3/2})(E_s+\epsilon E_s^{1/2}+\epsilon^{5/2})\\
    \leq & C(\epsilon E_s^{3/2}+\epsilon^2 E_s+\epsilon^{7/2}E_s^{1/2}+\epsilon^5).
\end{align}
By Cauchy integral formula, we have $$\int (I-\mathcal{H}_{\zeta})\partial_{\alpha}^{n+1}\rho_1(I-\mathcal{H}_{\zeta})\overline{\partial_t\mathcal{R}_1^{(n)}}\zeta_{\alpha}d\alpha=0$$ So we have 
\begin{align}
    \int (I-\mathcal{H}_{\zeta})\partial_{\alpha}^{n+1}\rho_1(I-\mathcal{H}_{\zeta})\overline{\partial_t\mathcal{R}_1^{(n)}}d\alpha=& \int (I-\mathcal{H}_{\zeta})\partial_{\alpha}^{n+1}\rho_1(I-\mathcal{H}_{\zeta})\overline{\partial_t\mathcal{R}_1^{(n)}}(1-\zeta_{\alpha})d\alpha\\
    \leq & \norm{\partial_{\alpha}^{n+1}\rho_1}_{L^2}\norm{ \zeta_{\alpha}-1}_{W^{s-1,\infty}}\norm{\partial_t\mathcal{R}_1^{(n)}}_{L^2}\\
        \leq & C\epsilon (E_s^{1/2}+\epsilon^{3/2})(E_s+\epsilon E_s^{1/2}+\epsilon^{5/2})\\
    \leq & C(\epsilon E_s^{3/2}+\epsilon^2 E_s+\epsilon^{7/2}E_s^{1/2}+\epsilon^5).
\end{align}
So we obtain
\begin{equation}\label{eeee}
\frac{d\mathcal{E}_n}{dt}\leq C(E_s^2+\epsilon E_s^{3/2}+\epsilon^2E_s+\epsilon^{7/2}E_s^{1/2}+\epsilon^5).
\end{equation}

\noindent The estimate of $\frac{d\mathcal{F}_n}{dt}$ is almost the same, and we obtain
\begin{equation}\label{ffff}
\frac{d\mathcal{F}_n}{dt}\leq C(\epsilon E_s^{3/2}+\epsilon^2E_s+\epsilon^{7/2}E_s^{1/2}+\epsilon^5).
\end{equation}
Combine (\ref{eeee}) and (\ref{ffff}), we obtain
\begin{equation}\label{ghijk}
\frac{d\mathcal{E}_s}{dt}\leq  C(\epsilon E_s^{3/2}+\epsilon^2E_s+\epsilon^{7/2}E_s^{1/2}+\epsilon^5).
\end{equation}

\subsection{Control $E_s$ in terms of $\mathcal{E}_s$}
To obtain bound on the energy $\mathcal{E}_s$, it's remaining to bound $E_s$ in terms of $\mathcal{E}_s$. First, we make the following reamrk.
\begin{lemma}\label{controlnegative}
We have  for $n\leq s$, 
$$\int \sigma_1^{(n)} \overline{\partial_{\alpha}\sigma_1^{(n)}}d\alpha \geq - C(\epsilon^{5}+\epsilon^{5/2} E_s^{1/2}+\epsilon E_s+E_s^{3/2}).$$
\end{lemma}
\begin{proof}
We decompose $\sigma_1^{(n)}=\Psi_1^{(n)}+\mathcal{S}_1^{(n)}$ as in (\ref{decompose}). So we have 
\begin{align*}
\int \sigma_1^{(n)} \overline{\partial_{\alpha}\sigma_1^{(n)}}d\alpha=& \int \Psi_1^{(n)} \overline{\partial_{\alpha}\Psi_1^{(n)}}d\alpha+\int \mathcal{S}_1^{(n)} \overline{\partial_{\alpha}\mathcal{S}_1^{(n)}}d\alpha +\int \Psi_1^{(n)} \overline{\partial_{\alpha}\mathcal{S}_1^{(n)}}d\alpha+\int \mathcal{S}_1^{(n)} \overline{\partial_{\alpha}\Psi_1^{(n)}}d\alpha\\
:=& W_1+W_2+W_3+W_4.
\end{align*}
Similar to the estimate for $\mathcal{R}_1^{(n)}$ and $\phi_1^{(n)}$, we have 
\begin{equation}
\norm{\partial_{\alpha}\mathcal{S}_1^{(n)}}_{H^s}\leq C(E_s+\epsilon E_s^{1/2}+\epsilon^{5/2}).
\end{equation}
\begin{equation}
\norm{\Psi_1^{(n)}}_{H^s}\leq C(E_s^{1/2}+\epsilon^{5/2}).
\end{equation}
The lemma follows directly from the above three equations. 
\end{proof}

\subsubsection{Bound $\norm{D_t^2r_1}_{H^s}$ by $\norm{D_tr_1}_{H^s}$ and $\norm{\partial_{\alpha}r_1}_{H^s}$} First we derive equation governing $D_t^2r_1$. We have by water waves equation
\begin{equation}
D_t^2\zeta=iA\zeta_{\alpha}-i,\quad \quad (D_t^0)^2\omega=iA_0\omega_{\alpha}-i.
\end{equation}

\begin{equation}
\tilde{D}_t^2\tilde{\zeta}=i\tilde{A}\tilde{\zeta}_{\alpha}-i +\epsilon^4\mathcal{R}_1, \quad \quad (\tilde{D}_t^0)^2\tilde{\omega}=i\tilde{A}_0\partial_{\alpha}\tilde{\omega}-i+\epsilon^4\mathcal{R}_2.
\end{equation}
where $$\norm{\epsilon^4\mathcal{R}_1-\epsilon^4\mathcal{R}_2}_{ H^{s+7}}\leq C\epsilon^{7/2}.$$
So we have 
\begin{equation}\label{ddddd}
D_t^2\zeta-(D_t^0)^2\omega-(\tilde{D}_t^2\tilde{\zeta}-(\tilde{D}_t^0)^2\tilde{\omega})= iA\zeta_{\alpha}-iA_0\omega_{\alpha}-(i\tilde{A}\tilde{\zeta}_{\alpha}-i\tilde{A}_0\tilde{\omega}_{\alpha})+error, 
\end{equation}
with 
$$\norm{error}_{H^s}\leq C\epsilon^{7/2}.$$
We write left hand side of (\ref{ddddd}) as 

\begin{equation}\label{LHSO}
\begin{split}
\text{LHS of } (\ref{ddddd})=& D_t^2\xi_1+(D_t^2-(D_t^0)^2)\omega-(\tilde{D}_t^2\tilde{\xi}_1+(\tilde{D}_t^2-(\tilde{D}_t^0)^2)\tilde{\omega})\\
=& D_t^2r_1+(D_t^2-\tilde{D}_t^2)\tilde{\xi}_1+(D_t^2-(D_t^0)^2)\omega+(\tilde{D}_t^2-(\tilde{D}_t^0)^2)\tilde{\omega}.
\end{split}
\end{equation}

\noindent Split $A=A_0+A_1$, $\tilde{A}=\tilde{A}_0+\tilde{A}_1$. We write the right hand side of (\ref{ddddd}) as  (omitt the $\epsilon^4$ term)
\begin{equation}\label{RHSO}
\begin{split}
\text{RHS of } (\ref{ddddd})=& iA_1\zeta_{\alpha}+iA_0\partial_{\alpha}\xi_1-(i\tilde{A}_1\tilde{\zeta}_{\alpha}+i\tilde{A}_0\partial_{\alpha}\tilde{\xi}_1)\\
=& i(A_1-\tilde{A}_1)\zeta_{\alpha}+i\tilde{A}_1\partial_{\alpha}r+i(A_0-\tilde{A}_0)\partial_{\alpha}\xi_1+i\tilde{A}_0\partial_{\alpha}r_1.
\end{split}
\end{equation}
By (\ref{ddddd}), (\ref{LHSO}), and (\ref{RHSO}),  we obtain
\begin{equation}\label{aaaaa}
\begin{split}
D_t^2r_1  =& i(A_1-\tilde{A}_1)\zeta_{\alpha}+i\tilde{A}_1\partial_{\alpha}r+i(A_0-\tilde{A}_0)\partial_{\alpha}\xi_1+i\tilde{A}_0\partial_{\alpha}r_1\\ &-\Big\{(D_t^2-\tilde{D}_t^2)\tilde{\xi}_1+(D_t^2-(D_t^0)^2)\omega+(\tilde{D}_t^2-(\tilde{D}_t^0)^2)\tilde{\omega}\Big\}
\end{split}
\end{equation}
By (\ref{A1Diff}), decompose $\zeta_{\alpha}=(r_1)_{\alpha}+(r_0)_{\alpha}+\tilde{\zeta}_{\alpha}$, it's easy to obtain 
\begin{equation}\label{qqq1}
\norm{i(A_1-\tilde{A}_1)\zeta_{\alpha}}_{H^s}\leq C\norm{A_1-\tilde{A}_1}_{H^s}\leq C\epsilon^{5/2}+C\epsilon E_s^{1/2}.
\end{equation}
By corollary \ref{A1estimate}, we have 
\begin{equation}\label{qqq2}
\begin{split}
\norm{iA_1\partial_{\alpha}\xi_1}_{H^s}\leq & \norm{A_1}_{H^s}\norm{\partial_{\alpha}\xi_1}_{H^s}\leq \epsilon^{5/2}.
\end{split}
\end{equation}
By (\ref{A0difference}) and Sobolev embedding, we have 
\begin{equation}\label{qqq3}
\begin{split}
\norm{i(A_0-\tilde{A}_0)\partial_{\alpha}\xi_1}_{H^s}\leq &\norm{A_0-\tilde{A}_0}_{W^{s,\infty}}\norm{\partial_{\alpha}\xi_1}_{H^s}\leq C\epsilon^{7/2}.
\end{split}
\end{equation}
Obviously,
\begin{equation}\label{qqq4}
\norm{i\tilde{A}_0\partial_{\alpha}r_1}_{H^s}\leq \norm{\tilde{A}_0}_{W^{s,\infty}}\norm{\partial_{\alpha}r_1}_{H^s}\leq C\norm{\partial_{\alpha}r_1}_{H^s}.
\end{equation}
By lemma \ref{lalala} and (\ref{Dtb1diff}), we have
\begin{equation}\label{qqq5}
\begin{split}
&\norm{(D_t^2-\tilde{D}_t^2)\tilde{\xi}_1+(D_t^2-(D_t^0)^2)\omega+(\tilde{D}_t^2-(\tilde{D}_t^0)^2)\tilde{\omega}}_{H^s}\leq C\epsilon^{5/2}.
\end{split}
\end{equation}
Recall that $\tilde{A}_1=0$. So $i\tilde{A}_1\partial_{\alpha}r=0$.

\noindent Use  (\ref{aaaaa}), together with (\ref{qqq1})-(\ref{qqq5}), and recall that $E_s^{1/2}:=\norm{\partial_{\alpha}r_1}_{H^s}+\norm{D_t r_1}_{H^s}+\norm{D_t^2r_1}_{H^s}$, we obtain
\begin{align}
\norm{D_t^2r_1}_{H^s}\leq & C\norm{\partial_{\alpha}r_1}_{H^s}+C\epsilon E_s^{1/2}+\epsilon^{5/2}\\
\leq & C(\norm{\partial_{\alpha}r_1}_{H^s}+\norm{D_tr_1}_{H^s})+C\epsilon \norm{D_t^2 r_1}_{H^s}+C\epsilon^{5/2}.
\end{align}
Therefore, for $\epsilon$ sufficiently small, we have 
\begin{equation}
\norm{D_t^2r}_{H^s}\leq C(\norm{\partial_{\alpha}r_1}_{H^s}+\norm{D_tr_1}_{H^s}+\epsilon^{5/2}).
\end{equation}
Now we are ready to bound $E_s$ by $\mathcal{E}_s$.
\begin{lemma}
We have
\begin{equation}
\norm{\partial_{\alpha}r_1}_{H^s}+\norm{D_tr_1}_{H^s}\leq C(\mathcal{E}^{1/2}+\epsilon^{5/2}).
\end{equation}
\end{lemma}
\begin{proof}
\noindent \textbf{Step 1.} Show that 
\begin{equation}\label{step1estimates}
\norm{D_tr_1}_{H^s}\leq C(\epsilon E_s^{1/2}+E_s+\epsilon^{5/2}+\norm{D_t\rho_1}_{H^s}).
\end{equation}
See the proof of lemma \ref{equivalencequantities}.

\vspace*{1ex}

\noindent \textbf{Step 2.} Show that 
$$\norm{(r_1)_{\alpha}}_{H^s}\leq  C(\norm{D_t\sigma_1}_{H^s}+\epsilon E_s^{1/2}+\epsilon^{5/2}).$$
Use the fact that $\bar{r}_1$ is almost holomorphic, similar to the argument in step 1, we have 
\begin{equation}\label{begin}
\norm{(r_1)_{\alpha}}_{H^s}\leq C(\norm{(I-\mathcal{H}_{\zeta})(r_1)_{\alpha}}_{H^s}+\epsilon E_s^{1/2}+\epsilon^{5/2}).
\end{equation}
To bound $\norm{(I-\mathcal{H}_{\zeta})(r_1)_{\alpha}}_{H^s}$ in terms of $\norm{D_t\sigma_1}_{H^s}$ plus an error term, consider 
$$U:=iA\zeta_{\alpha}-iA_0\omega_{\alpha}-i\tilde{A}\tilde{\zeta}_{\alpha}+i\tilde{A}_0\tilde{\omega}_{\alpha}.$$
Use the fact that $\tilde{A}=\tilde{A}_0=1$,  we have on one hand,
\begin{align*}
U=&i(\zeta_{\alpha}-\omega_{\alpha}-\tilde{\zeta}_{\alpha}+\tilde{\omega}_{\alpha})+i(A-1)\zeta_{\alpha}-i(A_0-1)\omega_{\alpha}-
i(\tilde{A}-1)\tilde{\zeta}_{\alpha}+i(\tilde{A}_0-1)\tilde{\omega}_{\alpha}\\
=& i(r_1)_{\alpha}+i(A-A_0)\omega_{\alpha}+i(A-1)\xi_{\alpha}.
\end{align*}
So 
\begin{equation}\label{okokokok}
\norm{(I-\mathcal{H}_{\zeta})(r_1)_{\alpha}}_{H^s}\leq \norm{(I-\mathcal{H}_{\zeta})U}_{H^s}+C(\epsilon E_s^{1/2}+E_s+\epsilon^{5/2}).
\end{equation}
On the other hand, use water wave equations $$(D_t^2-iA\partial_{\alpha})\zeta=-i, 
\quad \quad ((D_t^0)^2-iA_0\partial_{\alpha})\tilde{\omega}=-i,$$
 and the fact that $\tilde{\zeta}, \tilde{\omega}$ are good approximations:
$$(\tilde{D}_t^2-i\tilde{A}\partial_{\alpha})\tilde{\zeta}=-i+\epsilon^4\tilde{\mathcal{R}}, \quad \quad ((\tilde{D}_t^0)^2-i\tilde{A}_0\partial_{\alpha})\tilde{\omega}=-i+\epsilon^4\tilde{\mathcal{R}}_0,$$
with 
$$\norm{\epsilon^4\tilde{\mathcal{R}}-\epsilon^4\tilde{\mathcal{R}}_0}_{H^s}\leq \epsilon^{7/2}.$$
Denote $\epsilon^4\mathcal{R}:=\epsilon^4\tilde{\mathcal{R}}-\epsilon^4\tilde{\mathcal{R}}_0$. 
So we have
\begin{align*}
U=& D_t^2\zeta-(D_t^0)^2\omega-\tilde{D}_t^2\tilde{\zeta}+(\tilde{D}_t^0)^2\tilde{\omega}+\epsilon^4\mathcal{R}\\
=&  D_t(D_t\zeta-D_t^0\omega-\tilde{D}_t\tilde{\zeta}+\tilde{D}_t^0\tilde{\omega})+(b-b_0)\partial_{\alpha}D_t^0\omega+(b-\tilde{b}_0)\partial_{\alpha}\tilde{D}_t\tilde{\zeta}\\
&-(b-\tilde{b}_0)\partial_{\alpha}\tilde{D}_t^0\tilde{\omega}
+\epsilon^4\mathcal{R}.
\end{align*}
We have 
\begin{equation}
\begin{split}
&\norm{(b-b_0)\partial_{\alpha}D_t^0\omega+(b-\tilde{b}_0)\partial_{\alpha}\tilde{D}_t\tilde{\zeta}
-(b-\tilde{b}_0)\partial_{\alpha}\tilde{D}_t^0\tilde{\omega}}_{H^s}\\
\leq & C(\epsilon E_s^{1/2}+\epsilon^{5/2}).
\end{split}
\end{equation}
Denote 
$$U_{error}:=(b-b_0)\partial_{\alpha}D_t^0\omega+(b-\tilde{b}_0)\partial_{\alpha}\tilde{D}_t\tilde{\zeta}-(b-\tilde{b}_0)\partial_{\alpha}\tilde{D}_t^0\tilde{\omega}
+\epsilon^4\mathcal{R}.$$
We have
\begin{equation}\label{UUU}
\begin{split}
(I-\mathcal{H}_{\zeta})U=& (I-\mathcal{H}_{\zeta})D_t(D_t\zeta-D_t^0\omega-\tilde{D}_t\tilde{\zeta}+\tilde{D}_t^0\tilde{\omega})+(I-\mathcal{H}_{\zeta})U_{error}\\
=& D_t(I-\mathcal{H}_{\zeta})(D_t\zeta-D_t^0\omega-\tilde{D}_t\tilde{\zeta}+\tilde{D}_t^0\tilde{\omega})\\
&+[D_t\zeta, \mathcal{H}_{\zeta}]\frac{\partial_{\alpha}(D_t\zeta-D_t^0\omega-\tilde{D}_t\tilde{\zeta}+\tilde{D}_t^0\tilde{\omega})}{\zeta_{\alpha}}+(I-\mathcal{H}_{\zeta})U_{error}.
\end{split}
\end{equation}
Because we want to get estiamtes in terms of $D_t\sigma_1$, we rewrite $(I-\mathcal{H}_{\zeta})(D_t\zeta-D_t^0\omega-\tilde{D}_t\tilde{\zeta}+\tilde{D}_t^0\tilde{\omega})$ as
\begin{equation}\label{VVV}
\begin{split}
&(I-\mathcal{H}_{\zeta})(D_t\zeta-D_t^0\omega-\tilde{D}_t\tilde{\zeta}+\tilde{D}_t^0\tilde{\omega})\\
=& (I-\mathcal{H}_{\zeta})(D_t\xi-D_t^0\xi_0-\tilde{D}_t\tilde{\xi}+\tilde{D}_t^0\tilde{\xi}_0+b-b_0-\tilde{b}-\tilde{b}_0)\\
=& (I-\mathcal{H}_{\zeta})D_t\xi-(I-\mathcal{H}_{\omega})D_t^0\xi_0-(I-\mathcal{H}_{\tilde{\zeta}})\tilde{D}_t\tilde{\xi}+(I-\mathcal{H}_{\tilde{\omega}})\tilde{D}_t^0\tilde{\xi}_0
\\
&+(\mathcal{H}_{\zeta}-\mathcal{H}_{\omega})D_t^0\xi_0+(\mathcal{H}_{\zeta}-\mathcal{H}_{\tilde{\zeta}})\tilde{\xi}
-(\mathcal{H}_{\zeta}-\mathcal{H}_{\tilde{\omega}})\tilde{D}_t^0\tilde{\xi}_0+(I-\mathcal{H}_{\zeta})
(b-b_0-\tilde{b}-\tilde{b}_0)\\
:=&(\delta-\tilde{\delta})+ V_1+V_2,
\end{split}
\end{equation}
where
$$V_1:=(\mathcal{H}_{\zeta}-\mathcal{H}_{\omega})D_t^0\xi_0+(\mathcal{H}_{\zeta}-\mathcal{H}_{\tilde{\zeta}})\tilde{\xi}
-(\mathcal{H}_{\zeta}-\mathcal{H}_{\tilde{\omega}})\tilde{D}_t^0\tilde{\xi}_0,$$
and 
$$V_2:=(I-\mathcal{H}_{\zeta})
(b-b_0-\tilde{b}-\tilde{b}_0).$$
We have 
\begin{equation}\label{V}
\norm{D_tV_1}_{H^s}+\norm{D_tV_2}_{H^s}\leq C(\epsilon E_s^{1/2}+E_s+\epsilon^{5/2}).
\end{equation}
Combine (\ref{UUU}), (\ref{VVV}), and (\ref{V}), use the fact that 
$$\delta-\tilde{\delta}=\sigma_1+error,$$
where the $H^s$ norm of the error is controlled by $C(\epsilon E_s+E_s+\epsilon^{5/2}$. We obtain
\begin{equation}\label{good}
\norm{(I-\mathcal{H}_{\zeta})U}_{H^s}= \norm{D_t\sigma_1}_{H^s}+C(\epsilon E_s^{1/2}+\epsilon^{5/2})
\end{equation}
Combine (\ref{begin}),  (\ref{okokokok}) and (\ref{good}), we have 
\begin{equation}\label{end}
\norm{(r_1)_{\alpha}}_{H^s}\leq C(\norm{D_t \sigma_1}_{H^s}+\epsilon E_s^{1/2}+\epsilon^{5/2}).
\end{equation}

\vspace*{2ex}

\noindent \textbf{Step 3.} Control $\norm{D_t \sigma_1}_{H^s}$ and $\norm{D_t\rho_1}_{H^s}$ by $\mathcal{E}_s$.

\vspace*{2ex}

\noindent By corollary \ref{boundAAA} and lemma \ref{controlnegative}, we have 
\begin{align}
    \norm{D_t \sigma_1}_{H^s}+\norm{D_t\rho_1}_{H^s}\leq 2\mathcal{E}_s+C(\epsilon^{5}+\epsilon^{5/2} E_s^{1/2}+\epsilon E_s+E_s^{3/2}).
\end{align}

\noindent 
Combine (\ref{step1estimates}) and (\ref{end}), we obtain
$$E_s^{1/2}\leq C(\mathcal{E}_s+\epsilon E_s^{1/2}+\epsilon^{5/2}).$$
    So we obtian
\begin{equation}\label{abcde}
E_s^{1/2}\leq C(\mathcal{E}_s+\epsilon^{5/2}).
\end{equation}
\end{proof}
\noindent Combine (\ref{ghijk}) and (\ref{abcde}), we obtain
\begin{equation}
\frac{d\mathcal{E}_s}{dt}\leq C(\epsilon^{7/2} \mathcal{E}_s^{1/2}+\epsilon^2 \mathcal{E}_s+\epsilon \mathcal{E}_s^{3/2}+\epsilon^{5}).
\end{equation}

\noindent By bootstrap argument, we obtain
\begin{proposition}
Let $s\geq 4$ be given. $B(0)$, $e_0$ and $M_0$ be given in Theorem \ref{NLSlocal}. Let $C_0$, $C_1$ be given by Theorem \ref{longperiodic}, Theorem \ref{theorem1}, respectively. Denote $B(X,T)$ the solution of (\ref{NLS})  with initial data $B(0)$, and let $\zeta_1^{(1)}$ be defined as in  (\ref{zeta1}).  Let $\epsilon_0$ be given. Suppose $\mathcal{E}(0)=M_0\epsilon^3$. Then there exists a probably smaller $\epsilon_0=\epsilon_0(e_0, C_0, C_1, M_0, \delta, \norm{B(0)-1}_{H^{s+7}})$ so that for all $0<\epsilon<\epsilon_0$ and $0\leq t\leq \min\{C_0\epsilon^{-2}, C_1\epsilon^{-2}, e_0\epsilon^{-2}\}$, we have $\mathcal{E}(t)\leq C\epsilon^3$, where  $C=C(e_0, C_0, C_1, M_0, \norm{B(0)-1}_{H^{s+7}})$. 
\end{proposition}

\section{Justification of the NLS from full water waves}
In this section, we show that for non-vanishing wave packet-like data, the solution to the water wave system exists on the $O(\epsilon^{-2})$ time scale, and is well-approximated by the wave packet whose modulation evolves according the the 1d focusing NLS. Let's summarize what we have done so far: 

\begin{itemize}
\item [1.] Use the NLS to construct approximate solutions to the full water waves. Let $B$ be a  solution to 1d focusing NLS. We show that there is an approximate solution $\tilde{\zeta}=\alpha+\epsilon \zeta^{(1)}+\epsilon^2\zeta^{(2)}+\epsilon^3\zeta^{(3)}$ to the water waves system on time scale $O(\epsilon^{-2})$ such that $\zeta^{(1)}=Be^{i\phi}$. 

\item [2.] A priori energy estimate of error term.  Also, we show that if $(\zeta, D_t\zeta, D_t^2\zeta)$ is a solution to water waves system, with approximation $(\tilde{\zeta}, \tilde{D}_t\tilde{\zeta}, \tilde{D}^2\tilde{\zeta})$, then the remainder term $r:=\zeta-\tilde{\zeta}$ satisfies some good energy estimates on time scale $O(\epsilon^{-2})$. 

\end{itemize}
\noindent In the next subsection, we show that there exists initial data $(\zeta_0, v_0, w_0)$ such that $\bar{\zeta}_0-\alpha$ and $\bar{v}_0$ are holomorphic, and $(\zeta_0, v_0, w_0)-(\tilde{\zeta}(0), \tilde{D}_t\tilde{\zeta}(0), \tilde{D}_t^2\tilde{\zeta}(0))=O(\epsilon^{3/2})$ in appropriate sense. 

\subsection{Construction of appropriate initial data} In this subsection, we construct initial data to the water waves system which is close to the approximation $(\tilde{\zeta}, \tilde{D}_t\tilde{\zeta}, \tilde{D}_t^2\tilde{\zeta})$. To be precise, we need
\begin{itemize}
\item [(I-1)] $\bar{\zeta}_0-\alpha, D_t\bar{\zeta}_0\in \mathcal{H}ol_{\mathcal{N}}(\Omega(0))$, which is equivalent to 
$$(I-\mathcal{H}_{\zeta_0})(\bar{\zeta}_0-\alpha)=0,\quad \quad (I-\mathcal{H}_{\zeta_0})D_t\bar{\zeta}(0)=0.$$

\item [(I-2)] $\zeta_0=\omega(0)+\xi_1(0)$, such that 
$$(I-\mathcal{H}_p)(\bar{\omega}(0)-\alpha)=0, \quad \quad (I-\mathcal{H}_p)D_t^0\omega(0)=0.$$

\item [(I-3)] The distance between $\omega(0)$ and $\tilde{\omega}(0)$ is small:
$$\norm{\omega(0)-\tilde{\omega}(0)}_{W^{s'+1,\infty}}\leq C\epsilon^2.$$
$$\norm{D_t^0\omega(0)-\tilde{D}_t^0\tilde{\omega}(0)}_{W^{s',\infty}}\leq C\epsilon^2.$$
And the distance between $\xi_1$ and $\tilde{\xi}_1$ is also small:
$$
\norm{\partial_{\alpha}(\xi_1(0)-\tilde{\xi}_1(0))}_{H^s(\RR)}\leq C\epsilon^{3/2}.
$$

$$\norm{D_t\xi_1(0)-\tilde{D}_t\tilde{\xi}_1(0)}_{H^{s+1/2}(\RR)}\leq C\epsilon^{3/2}.$$

\item [(I-4)] $A_0(0)$, $(D_t^0)^2\omega(0)$, $D_t^0\omega(0)$ satisfy the following compatibility condition  \begin{equation}
(I-\mathcal{H}_p)(A_0(0)-1)=i[(D_t^0)^2\omega(0), \mathcal{H}_p]\frac{\partial_{\alpha}\bar{\omega}(0)-1}{\partial_{\alpha}\omega(0)}+i[D_t^0\omega(0), \mathcal{H}_p]\frac{\partial_{\alpha}D_t^0\bar{\omega}(0)}{\partial_{\alpha}\omega(0)}.
\end{equation}

\item [(I-5)]
$A(0)$, $D_t^2\zeta(0)$, $D_t\zeta(0)$ satisfy the following compatibility condition 
\begin{equation}
\begin{split}
(I-\mathcal{H}_{\zeta_0})(A(0)-1)=i[D_t^2\zeta(0), \mathcal{H}_{\zeta_0}]\frac{\bar{\zeta}_{\alpha}-1}{\zeta_{\alpha}}+i[D_t\zeta(0), \mathcal{H}_{\zeta(0)}]\frac{\partial_{\alpha}D_t\bar{\zeta}(0)}{\zeta_{\alpha}}
\end{split}
\end{equation}
\end{itemize}

\noindent In the following lemma, we show that initial data satisfy (I-1)-(I-5) exist.
\begin{lemma}\label{initial}
For sufficiently small $\epsilon_0>0$, there exist $\omega(0)$, $\zeta(0)=\zeta_0$ such that for all $\epsilon<\epsilon_0$, (I1)-(I5) hold. 
\end{lemma}

\noindent To prove lemma \ref{initial}, we first prove the following.
\begin{lemma}\label{periodicinitial}
Let $\tilde{\omega}_0=\alpha+\epsilon \omega^{(1)}+\epsilon^2\omega^{(2)}+\epsilon^3\omega^{(3)}$ be such that the $\omega^{(1)}=c_0e^{ik\alpha}$ for some constant $c_0$, and $\omega^{(2)}, \omega^{(3)}\in W^{s'+1,\infty}(\TT)$. Then there exists $\omega(0)$ be such that $\omega(0)-\alpha$ is periodic with period 2$\pi$,  
$$\norm{\omega(0)-\tilde{\omega}}_{W^{{s'+1,\infty}}(\mathbb{T})}\leq C\epsilon^2,$$
and $\bar{\omega}(0)-\alpha=\Phi_0(\omega(\alpha,0))$, where $\Phi_0$ is holomorphic in the domain bounded above by the curve $\omega(0)$, satisfying
$$\lim_{\Im z\rightarrow-\infty} \Phi_0(z)=0.$$
\end{lemma}

In the appendix, we show that if we define $\Gamma:=\{\alpha+ce^{ik\alpha}:\alpha\in \RR\}$ and $\Omega_-^0$ the domain which is below $\Gamma$, then $ce^{-ik\alpha}$ is not holomorphic in $\Omega_-^0$. So we cannot simply take $\omega(0)=\alpha+ce^{ik\alpha}$.

\vspace*{1ex}

\noindent Lemma \ref{periodicinitial} is a direct consequence of the following lemma.
\begin{lemma}\label{periodicinitialhelp}
Let $s'\geq 0$ and $\epsilon\geq 0$ be fixed. There exists $\epsilon_0=\epsilon_0(s')>0$ be sufficiently small such that for all $\epsilon$ with $0\leq \epsilon\leq \epsilon_0$, there exists  $\omega$ such that $\omega-\alpha\in W^{s'+1,\infty}(\TT)$, satisfying
\begin{equation}\label{equation1}
\bar{\omega}-\alpha=ce^{-i\omega},
\end{equation}
and
\begin{equation}
    \norm{\omega-\alpha-\bar{c}e^{i\alpha}}_{W^{s'+1,\infty}(\TT)}\leq C\epsilon^2,
\end{equation}
for some constant $C=C(s')>0$.
\end{lemma}
\begin{proof}
We prove the lemma by iteration. It suffices to prove the case that $s'=0$. Let $\omega_0:=\alpha$ and define
$$\bar{\omega}_{n+1}=\alpha+ce^{-i\omega_n}.$$
Then $\omega_1=\alpha+\bar{c}e^{i\alpha}$. We have 
$$\norm{\omega_1-\omega_0}_{\infty}=|c|=\epsilon.$$
Then we have 
\begin{align*}
\norm{\omega_{2}-\omega_1}_{\infty}=& |c|  \norm{e^{-i\omega_{0}}}_{\infty}\norm{e^{-i(\omega_1-\omega_{0})}-1}_{\infty}\\
= & \epsilon \norm{\sum_{k\geq 1}\frac{(-i(\omega_1-\omega_{0}))^k}{k!}}_{\infty}\\
\leq & \epsilon \norm{\omega_1-\omega_{0}}_{\infty} \sum_{n\geq 1}\frac{\norm{\omega_1-\omega_{0}}_{\infty}^{k-1}}{k!}\\
\leq &\epsilon\norm{\omega_1-\omega_0}_{\infty}\sum_{n\geq 1}\frac{\epsilon^{k-1}}{k!}\leq \epsilon \norm{\omega_1-\omega_0}_{\infty} \sum_{n\geq 1}\frac{\epsilon^{k-1}}{(k-1)!}\\
\leq & 2\epsilon \norm{\omega_1-\omega_{0}}_{\infty}.
\end{align*}
From this, we have also that 
$$\norm{\omega_2-\alpha}_{\infty}\leq \norm{\omega_2-\omega_1}_{\infty}+\norm{\omega_1-\alpha}_{\infty}\leq \epsilon+2\epsilon^2.$$
We claim that 
\begin{equation}
\norm{\omega_n-\alpha}_{\infty}\leq \epsilon+\sum_{k=2}^n (2\epsilon)^k,
\end{equation}
and
\begin{equation}
\norm{\omega_{n}-\omega_{n-1}}_{\infty}\leq 2\epsilon \norm{\omega_{n-1}-\omega_{n-2}}_{\infty}.
\end{equation}
Indeed,
\begin{align*}
\norm{\omega_{n+1}-\omega_n}_{\infty}=& |c|\norm{e^{-i\alpha}e^{-i(\omega_{n-1}-\alpha)}}_{\infty}\norm{e^{-i(\omega_n-\omega_{n-1})}-1}_{\infty}\\
\leq & \epsilon e^{2\epsilon}\norm{\omega_n-\omega_{n-1}}_{\infty}\sum_{k\geq 1}\frac{\norm{\omega_n-\omega_{n-1}}_{\infty}^{k-1}}{k!}\\
\leq &\epsilon e^{2\epsilon}\norm{\omega_n-\omega_{n-1}}_{\infty}\sum_{k\geq 1}\frac{(2\epsilon)^{k-1}}{(k-1)!} \quad \leq \epsilon e^{2\epsilon}\norm{\omega_n-\omega_{n-1}}_{\infty}e^{2\epsilon}\\
\leq & 2\epsilon \norm{\omega_n-\omega_{n-1}}_{\infty}.
\end{align*}
So the induction hypothesis is verified, and the claim follows. So 
$$\norm{\omega_n-\omega_{n-1}}_{\infty}\leq (2\epsilon)^n.$$
There exists $\xi_0\in L^{\infty}$ such that 
$$\omega_n-\alpha\rightarrow \xi_0.$$
It's easy to show that $\xi_0\in W^{s'+1}(\TT)$ and $\omega_n-\alpha\rightarrow \xi_0$ in $W^{s'+1,\infty}(\TT)$. Moreover, if we denote $\omega:=\xi_0+\alpha$, we have 
$$\norm{\omega-\alpha-\bar{c}e^{ik\alpha}}_{W^{s'+1,\infty}}\leq C\epsilon^2.$$
So the proof of Lemma \ref{periodicinitialhelp} and hence Lemma \ref{periodicinitial} is completed.
\end{proof}

\noindent With lemma \ref{periodicinitial}, we can  prove lemma \ref{initial}.
\begin{proof}[Proof of lemma \ref{initial}]
Given $\tilde{\zeta}$ and $\tilde{\omega}$ given by (\ref{approxzetaomega1})-(\ref{approxzetaomega3}),
by Lemma \ref{periodicinitial}, there exists $\omega_0$ with $\omega_0-\alpha\in W^{s'+1,\infty}(\TT)$ such that $\bar{\omega}_0-\alpha=B_0(0)e^{-i\omega_0}$ and $\norm{\omega_0-\alpha-B_0(0)\epsilon e^{i\alpha}}_{W^{s'+1,\infty}}\leq C(s)\epsilon^2$. By Lemma \ref{goodgood}, we have 
\begin{equation}\label{ooo}
(I-\mathcal{H}_{\omega_0})(\bar{\omega}_0-\alpha)=B_0(0)(I-\mathcal{H}_{\omega_0})e^{-i\omega_0}=0.
\end{equation}
We want to find $\zeta_0$ such that 
\begin{equation}\label{zzz}
(I-\mathcal{H}_{\zeta_0})(\bar{\zeta}_0-\alpha)=0.
\end{equation}
Write 
$$\zeta_0=\omega_0+\xi_1(0).$$
We want 
\begin{equation}\label{errorh}
\xi_1(0)-\tilde{\xi}_1(0)=O(\epsilon^{3/2}).
\end{equation}
We simply write $\xi_1(0)$ as $\xi_1$ and $\tilde{\xi}_1(0)$ as $\tilde{\xi}_1$.  By (\ref{zzz}), we expect that \begin{equation}\label{alaska}
\bar{\zeta}_0-\alpha=(I+\mathcal{H}_{\zeta_0})f,
\end{equation}
for some function $f$. By (\ref{errorh}), we expect that $f$ should be closed to $\tilde{\xi}$. It's easy to see that we can take 
$$f=\frac{1}{2}(\overline{\tilde{\xi}_1+\omega_0-\alpha}).$$
Denote $\xi_0:=\omega_0-\alpha$.  Recall that $(I-\mathcal{H}_{\omega_0})\bar{\omega}_0=0$, so we have
$$\frac{1}{2}(I+\mathcal{H}_{\omega_0})\bar{\xi}_0=\bar{\xi}_0.$$
So we have 
\begin{align*}
(I+\mathcal{H}_{\zeta_0})f=& \frac{1}{2}(I+\mathcal{H}_{\zeta_0})\bar{\tilde{\xi}}_1+ \frac{1}{2}(I+\mathcal{H}_{\zeta_0})\bar{\xi}_0\\
=&\frac{1}{2}(I+\mathcal{H}_{\zeta_0})\bar{\tilde{\xi}}_1+ \frac{1}{2}(I+\mathcal{H}_{\omega_0})\bar{\xi}_0+(\mathcal{H}_{\zeta_0}-\mathcal{H}_{\omega_0})\bar{\xi}_0\\
=& \frac{1}{2}(I+\mathcal{H}_{\zeta_0})\bar{\tilde{\xi}}_1+\bar{\xi}_0+(\mathcal{H}_{\zeta_0}-\mathcal{H}_{\omega_0})\bar{\xi}_0
\end{align*}
So (\ref{alaska}) is equivalent to 
\begin{equation}\label{north}
\bar{\xi}_1=\frac{1}{2}(I+\mathcal{H}_{\zeta_0})\bar{\tilde{\xi}}_1+(\mathcal{H}_{\zeta_0}-\mathcal{H}_{\omega_0})\bar{\xi}_0.
\end{equation}
(\ref{north}) can be solved by iteration: let $g_0=0$, $z_0=\alpha$. Assume $g_n$ has been constructed, define  $z_n=g_n+\omega_0$. Then define $g_{n+1}$ by
$$\bar{g}_{n+1}=\frac{1}{2}(I+\mathcal{H}_{z_n})\bar{\tilde{\xi}}_1+(\mathcal{H}_{z_n}-\mathcal{H}_{\omega_0})\bar{\xi}_0.$$
Then it's easy to prove that $\{g_n\}$ defines a Cauchy sequence in $H^{s+7}(\RR)$, given that $\tilde{\xi}_1\in H^{s+7}(\RR)$. 
See lemma 5.1 of \cite{Totz2012} for example. 

Use the same argument, we can show that $D_t\bar{\zeta}(0)\in \mathcal{H}ol_{\mathcal{N}}$, and (1), (2), (3)  hold.  (4) and (5) can be proved similarly.

\end{proof}

\subsection{Long time well-posedness}

By energy estimates in the previous section and the initial data constructed above, we can prove the following theorem.

\begin{thm}\label{wuwu}
Let $M_0>0$, $s\geq 4$, $k\geq 0$ be given. $B(0)=B_1(0)+B_0(0)$, with $B_0(0)=1$, $B_1(0)\in H^{s+7}(\RR)$, and $e_0$ be given in Theorem \ref{NLSlocal}. Let $C_0$ be as in Theorem \ref{longperiodic}, and $C_1$ be as in Theorem \ref{theorem1}. Denote 
\begin{equation}
    k_0:=\min\{e_0, C_0, C_1\}.
\end{equation}
Denote the $B(X,T)$ the solution of (\ref{NLS})  with initial data $B(0)$, and let $\zeta_1^{(1)}$ be defined as in  (\ref{zeta1}).  Then there is $\epsilon_0=\epsilon_0(s, M_0, \norm{B_1(0)}_{H^{s+7}(\RR)})>0$ so that for $\epsilon<\epsilon_0$, there exists compatible initial data $(\zeta_0, v_0, w_0)$ to water waves system such that
$$(\zeta_0, v_0, w_0)=(\omega_0+\xi_1(0), D_t^0\omega_0+(v_1)_0, (D_t^0)^2\omega_0+(w_1)_0),$$ 
where $(\omega_0, D_t^0\omega_0, (D_t^0)^2\omega_0)$ is a compatible initial data for periodic water waves sytem (\ref{special}),  satisfying
\begin{equation}
\begin{split}
&\norm{(\partial_{\alpha}\omega_0, D_t^0\omega_0, (D_t^0)^2\omega_0)-\epsilon(\partial_{\alpha}\zeta_0^{(1)}(0), \partial_t \zeta_0^{(1)}(0), \partial_t^2\zeta_0^{(1)}(0))}_{H^{s'}(\mathbb{T})\times H^{s'+1/2}(\mathbb{T})\times H^{s'}(\mathbb{T})}\\
&\leq M_0 \epsilon^2,
\end{split}
\end{equation}
\begin{equation}
\begin{split}
&\norm{(\partial_{\alpha}\xi_1(0), (v_1)_0, (w_1)_0) -\epsilon( \partial_{\alpha}\zeta_1^{(1)}(0),  \partial_t\zeta_1^{(1)}(0),  \partial_t^2 \zeta_1^{(1)}(0)}_{H^{s}\times H^{s+1/2}\times H^{s}}\\
\leq &M_0\epsilon^{3/2},
\end{split}
\end{equation}
and for all such initial data, there exists a possibly smaller $\epsilon_0=\epsilon_0(\norm{B_1(0)}_{H^{s+7}(\RR)}, M_0, k_0)>0$ such that the water waves system has a unique solution $\zeta(\alpha,t)$ with $(\partial_{\alpha}(\zeta-\alpha), D_t\zeta, D_t^2\zeta)\in C([0, k_0\epsilon^{-2}]; X^s\times X^{s+1/2}\times X^s)$ satisfying 
\begin{equation}
\begin{split}
\sup_{0\leq t\leq k_0\epsilon^{-2}}\norm{(\zeta_{\alpha}(t)-1, D_t\zeta(t), D_t^2\zeta(t)-\epsilon (\zeta_{\alpha}^{(1)}(t), \zeta_t^{(1)}, \zeta_{tt}^{(1)})}_{X^s\times X^{s+1/2}\times X^s}\leq C\epsilon^{3/2},
\end{split}
\end{equation}
for some constant $C=C(s,M_0, \norm{B(0)-1}_{H^{s+7}})$.
\end{thm}
In particular, if we take $B$  to be the Peregrine soliton, then we justify the Peregrine soliton from the full water waves.

\subsection{Rigorous justification of the Peregrine soliton in  Lagrangian coordinates}
Let's change of variables back the our more familiar lagrangian coordinates.
We have $\kappa_t=b(\kappa)$. This gives a smooth function $\kappa:\RR\rightarrow \RR$. Taking $k_0$ smaller if necessary, it's easy to show that 
\begin{equation}
\norm{\kappa_{\alpha}}_{W^{s-1, \infty}}\geq 1/2, \quad \forall ~t\in [0, k_0\epsilon^{-2}].
\end{equation}
So $\kappa$ is a diffeomorphism.  Let $z=\zeta\circ \kappa$, $a$ be such that $(a\kappa_{\alpha})\circ \kappa^{-1}=A$, we obtain water waves equation (\ref{system_boundary}), which is in Lagrangian coordinates.  We can then obtain estimates for the remainder term in Lagrangian coordinates. 
\begin{remark}\label{remarkchangeback}
In Lagrangian coordinates, $\zeta_{\alpha}-1$ becomes $z_{\alpha}-\kappa_{\alpha}$.  So we have 
\begin{equation}
    z_{\alpha}-1=z_{\alpha}-\kappa_{\alpha}+(\kappa_{\alpha}-1).
\end{equation}
We have 
\begin{equation}
   \sup_{t\in [0,k_0\epsilon^{-2}]} \norm{z_{\alpha}-\kappa_{\alpha}}_{X^{s-1/2}}\lesssim \epsilon^{3/2}.
\end{equation}
However, it seems that $\norm{\kappa_{\alpha}-1}_{X^{s-1/2}}$ can be as large as $\epsilon^{1/2}$ on time scale $O(\epsilon^{-2})$. So we are unable to rigorously justify the modulation approximation for $Re\{z_{\alpha}-1\}$. Please see \cite{Totz2012} for more details.

\end{remark}

\section*{Acknowledgement}
The author would like to thank his Ph.D advisor, Prof. Sijue Wu, for introducing him this interesting topic, for many helpful discussions and  invaluable comments, and for carefully reading the draft of this paper. The author would like to thank Prof. Peter Miller for providing references  regarding NLS with nonzero boundary values at $\infty$. The author would  like to thank Prof. Yongsheng Han for his help. The author would also like to thank Fan Zheng, S. Shahshahani for invalueable comments and Prof. Tao Luo for carefully reading the draft of this paper. This work is partially supported by NSF grant DMS-1361791.

\appendix

\section{Holomorphicity of plane waves}
Let $\zeta(\alpha)=\alpha+ce^{ik\alpha}, \alpha\in \RR$, $c$ is a small constant, $k>0$, for simplicity, assume $k$ is an integer. Let 
$$\Gamma:=\{\zeta(\alpha):\alpha\in \RR\}.$$
Then $\Gamma$ is a graph. Let $\Omega_+$ be the region above $\Gamma$, and $\Omega_-$ the region below $\Gamma$.
On one hand, it'  easy to prove that
\begin{lemma}
$\alpha$, $e^{-ik\alpha}$ and $e^{ik\alpha}$ are holomorphic in $\Omega_+$. 
\end{lemma}
On the other hand, we'll show that $e^{ik\alpha}$ cannot be boundary value of a bounded holomorphic function in $\Omega_-$.
\begin{lemma}
If $c\neq 0$, then $e^{ik\alpha}$ cannot be boundary value of a holomorphic function in $\Omega_-$.
\end{lemma}
\begin{proof}
If $e^{ik\alpha}$ is boundary value of a holomorphic function in $\Omega_-$, then $e^{ik\alpha}$ is entire, and so $\alpha$ is entire. Assume $\alpha=\Phi(\zeta(\alpha))$, $\Phi$ entire. Let $\Psi(\zeta)=\zeta+ce^{ik\zeta}$. Then $\Psi$ is entire, and $\Psi(\alpha)=\zeta(\alpha)$. So we have 
$$
\begin{cases}
\Psi\circ \Phi(\zeta(\alpha))=\zeta(\alpha)\\
\Phi\circ \Psi(\alpha)=\alpha.
\end{cases}
$$
$\Psi\circ \Phi$ and $\Phi\circ \Psi$ are entire, we must have $\Psi\circ \Phi(z)\equiv z, \Phi\circ\Psi(z)\equiv z$. So $\Psi$ and $\Phi$ are inverse of each other. 

If $c\neq 0$, then the function $z+ce^{ikz}$ has an essential singularity at $\infty$ because $ce^{ikz}$ does. By Picard's theorem, $z+ce^{ikz}$ attains all values in $\mathbb{C}$ infinitely many times with at most one exception.  Suppose $z_0$ is this exception, i.e., $z+ce^{ikz}=z_0$ has finitely many solutions (possibly none). But then $z+ce^{ikz}=z_0+2\pi$ has infinitely many solutions. Then 
$$z-2\pi+ce^{ikz}=z_0  \quad \Rightarrow \quad z-2\pi+ce^{ik(z-2\pi)}=z_0.$$
So $z+ce^{ikz}=z_0$ has infinitely many solutions, a contradiction. 

In particular, $z+ce^{ikz}=0$ has infinitely many solutions. So $\Psi$ is not invertible, contradiction. 
\end{proof}

\begin{lemma}
If $e^{-ik\alpha}$ is boundary value of a holomorphic function in $\Omega_-$, then $e^{ik\alpha}$ is also holomorphic in $\Omega_-$.
\end{lemma}
\begin{proof}
Let $e^{-ik\alpha}=G(\zeta(\alpha))$, where $G$ is holomorphic in $\Omega_-$. Then the zeros of $G$ is a discrete set, which we denote by $S$. We'll show that $S=\emptyset$.  Since $\zeta=\alpha+ce^{ik\alpha}$, we have 
$$\alpha=\zeta(\alpha)-\frac{c}{e^{-ik\alpha}}=\zeta(\alpha)-\frac{c}{G(\zeta(\alpha))}.$$
Define$H(\zeta):=\zeta-\frac{c}{G(\zeta)}, ~ \zeta\in \Omega_-$.  Then $H$ has boundary value $\alpha$. 
So $\alpha$ is boundary value of a meromorphic function in $\Omega_-$, with poles at $S$. 

Note that $e^{-ikH(\zeta(\alpha))}$ has boundary values $e^{-ik\alpha}$, and $e^{-ikH(\zeta)}$ is holomorphic in $\Omega_-\setminus S$, by uniqueness extension of holomorphic functions, we must have $e^{-ikH(\zeta)}=G(\zeta)$ on $\Omega_-\setminus S$. 

If $S\neq \emptyset$, then take $z_0\in S$. Then since $G(z_0)$ is defined,  $z_0$ must be a removable singularity of $e^{-ikH(\zeta)}$. However,  since $z_0$ is a pole of $H(\zeta)$, so $z_0$ is an essential singularity of $e^{-ikH(\zeta)}$, a contradiction. So $S=\emptyset$.  

So we conclude that $\alpha$ is holomorphic in $\Omega_-$, and so $e^{ik\alpha}$ is holomorphic in $\Omega_-$.
\end{proof}
\begin{corollary}
$e^{-ik\alpha}$ cannot be the boundary value of a holomorphic function in $\Omega_-$ if $c\neq 0$. 
\end{corollary}

\bibliography{qingtangbib}{}
\bibliographystyle{plain}

\end{document}